\documentclass[reqno]{amsart}
\usepackage{comment,amssymb,latexsym,upref,enumerate,fouridx}
\usepackage{mathrsfs,color}
\usepackage{centernot}

\usepackage[colorlinks,linkcolor=blue,citecolor=blue]{hyperref}
\allowdisplaybreaks
\numberwithin{equation}{section}

\theoremstyle{plain}
\newtheorem{thm}{Theorem}[section]
\newtheorem{lem}[thm]{Lemma}
\newtheorem{prop}[thm]{Proposition}

\theoremstyle{definition}

\theoremstyle{remark}
\newtheorem{rem}[thm]{Remark}

\newcommand{\ep}{\epsilon}
\newcommand{\gsl}{\fsl{g}}

\newcommand{\eqn}[2]{ \begin{equation*} #2  \end{equation*} }

\input Tdef.def

\begin{document}

\title[The Fuchsian approach to global existence for hyperbolic equations]{The Fuchsian approach to global existence for hyperbolic equations}

\author[F. Beyer]{Florian Beyer}
\address{Dept of Mathematics and Statistics\\
730 Cumberland St\\
University of Otago, Dunedin 9016\\ New Zealand}
\email{fbeyer@maths.otago.ac.nz }

\author[T.A. Oliynyk]{Todd A. Oliynyk}
\address{School of Mathematical Sciences\\
9 Rainforest Walk\\
Monash University, VIC 3800\\ Australia}
\email{todd.oliynyk@monash.edu}

\author[J.A. Olvera-Santamar\'{i}a]{J. Arturo Olvera-Santamar\'{i}a}
\address{School of Mathematical Sciences\\
9 Rainforest Walk\\
Monash University, VIC 3800\\ Australia}
\email{arturo.olvera@monash.edu }

\begin{abstract}
\noindent
We analyze the Cauchy problem for symmetric hyperbolic equations with a time singularity of Fuchsian type and establish a global existence theory along with decay estimates for evolutions towards the singular time under a small initial data assumption.
We then apply this theory to semilinear wave equations near spatial infinity on Minkowski and Schwarzschild spacetimes, and to the relativistic Euler equations
with Gowdy symmetry on Kasner spacetimes. 
\end{abstract}

\maketitle

\section{Introduction\label{intro}}
Systems of first order hyperbolic equations that can be expressed in the form
\begin{equation}
B^0(t,u)\del{t}u + B^i(t,u)\nabla_{i} u  = \frac{1}{t}\Bc(t,u)\Pbb u + F(t,u) \label{symivp}
\end{equation}
are said to be \textit{Fuchsian} if the right-hand side is formally singular at $t=0$.  Traditionally, these systems have been viewed as \emph{singular initial value problems} (SIVP). Here, \textit{asymptotic data} is prescribed at the singular time $t=0$ and then \eqref{symivp} is used to evolve the asymptotic data \emph{away from the singular time} to construct solutions on time intervals $t\in (0,T]$ for some possibly small $T>0$. Some earlier applications of this theory were restricted to solutions in a real analytic class \cite{andersson2001,choquet-bruhat2006,choquet-bruhat2004,ChruscielKlinger:2015,damour2002,heinzle2012,isenberg1999,isenberg2002,kichenassamy1998,Klinger:2015,Rendall:2004}; see also \cite{OliynykKunzle:2002a,OliynykKunzle:2002b} for applications 
in the ODE setting. More recently the theory of SIVPs for Fuchsian equations has been extended to classes of solutions with Sobolev regularity in \cite{ames2013a,ames2013b,claudel1998a,kichenassamy2007k} using more standard local-in-time PDE techniques for hyperbolic equations. The asymptotic data is typically allowed to be large, but the time of existence may be small. In general relativity and for Einstein's equations, such techniques were applied in \cite{ames2017,beyer2010b,beyer2011,beyer2017,Fournodavlos:2016,rendall2000,rendall2004a,stahl2002}. Fuchsian SIVPs also arise in many other areas; for a selection, see \cite{kichenassamy2007k}. 

While the SIVP approach for establishing the existence of solutions to \eqref{symivp} is useful for certain applications, there are many situations where
the standard initial value problem (IVP) for the Fuchsian system \eqref{symivp} is the relevant problem. In this case, initial data is prescribed at some finite time, say $t=1$, and the problem becomes
to establish the existence of solutions to \eqref{symivp} all the way up to the singular time at $t=0$, that is, for $t\in (0,1]$. The flavour of this problem is that of a global existence problem
and a prime example comes from \cite{Oliynyk:CMP_2016} where it was shown that the Einstein-Euler equations on cosmological spacetimes with a positive cosmological constant
can be cast into Fuchsian form. It was further established in \cite{Oliynyk:CMP_2016} that the IVP for a certain class of symmetric hyperbolic Fuchsian systems, where the initial data was specified at $t=1$, is solvable for $t$ on the whole interval $(0,1]$ provided the initial data is chosen suitably small. Together, these two results were used to deduce the
future non-linear stability of perturbations of 
Friedmann-Lema\^{i}tre-Robertson-Walker (FLRW)
solutions to the Einstein-Euler equations with a positive cosmological constant and a linear equation of state; see also 
\cite{HadzicSpeck:2015,LubbeKroon:2013,RodnianskiSpeck:2013,Speck:2012} for related results using different methods.

Since then, the small initial data existence
result from \cite{Oliynyk:CMP_2016} for systems of the form \eqref{symivp} has been employed in a variety of situations to establish the global existence of solutions to the future for
a number of different hyperbolic systems on expanding cosmological spacetimes that are either exactly or small perturbation of FLRW; see the articles \cite{LeFlochWei:2015,LiuOliynyk:2018b,LiuOliynyk:2018a,LiuWei:2019,Wei:2018} for details.
These results all rely heavily on the fact that the spacetime is either exactly or nearly FLRW
in order to transform the evolution equations into the Fuchsian form \eqref{symivp}, which then allows for the global existence theory from \cite{Oliynyk:CMP_2016} to be applied. 
This state of affairs leads one to question if the Fuchsian view point is of any use for establishing
the global existence of solutions to hyperbolic equations on spacetimes such as Minkowski or Schwarzschild spacetimes, or for cosmological spacetimes that are not necessarily close to FLRW such as Kasner spacetimes. 

The main aim of this article is to show that the Fuchsian approach to global existence applies to these new settings as well. However, we
cannot  apply the global existence theory from
 \cite{Oliynyk:CMP_2016} directly since the Fuchsian systems that arise in these new settings contain singular terms that are not covered by this theory. Roughly speaking, the global existence results from \cite{Oliynyk:CMP_2016} apply to Fuchsian systems of the form
 \eqref{symivp}  for which the coefficients $B^0$, $B^i$, $\Bc$ and $F$ are all regular in $t$ as $t\searrow 0$, while the Fuchsian systems that arise from the applications considered in this article have coefficients $B^i$ and $F$ that are singular in $t$ and can be expanded as follows
 \begin{align}
 B^i(t,u) &=  B_0^i(t,u)+\frac{1}{t^{\frac{1}{2}}} B_1^i(t,u) + \frac{1}{t}B_2^i(t,u) \label{newsing.1}\\
 F(t,u) &= F_0(t,u)+\frac{1}{t^{\frac{1}{2}}} F_1(t,u) + \frac{1}{t}F_2(t,u) \label{newsing.2}
 \end{align}
where now the coefficients $B_a^i$ and $F_a^i$, $a=0,1,2$, are all regular in $t$ as $t\searrow 0$ and all the coefficients appearing 
in the equation must satisfy the additional technical
conditions described in Section \ref{coeffassump} (see also Section \ref{sec:timetrafos} for a time rescaled version of these conditions). 
Of particular note is that these additional assumptions imply that the coefficients $B^0$, $B_a^i$ and $F_a^i$, $a=1,2$, must satisfy
\begin{gather*}
  \Pbb^\perp B^0(t,u)\Pbb = \Ord(t^\frac{1}{2}+\Pbb u),\quad
    \Pbb B^0(t,u)\Pbb^\perp = \Ord(t^\frac{1}{2}+\Pbb u),\\
 \Pbb^\perp B_1^i(t,u)\Pbb^\perp = \Ord(\Pbb u),\quad
  \Pbb^\perp B_2^i(t,u)\Pbb^\perp = \Ord(\Pbb u\otimes \Pbb u),\\
   \Pbb^\perp B_2^i(t,u)\Pbb = \Ord(\Pbb u),\quad
    \Pbb B_2^i(t,u)\Pbb^\perp = \Ord(\Pbb u),\\
\Pbb F_1(t,u) = \Ord(u), \quad \Pbb^\perp F_1(t,u) =  \Ord(\Pbb u),\\
\Pbb F_2(t,u) = 0 \AND \Pbb^\perp F_2(t,u) = \Ord(\Pbb u\otimes \Pbb u),
 \end{gather*}
where $\Pbb$ is the projection operator that appears in \eqref{symivp} and $\Pbb^\perp=\id-\Pbb$ is the complementary projection operator.
These conditions are fundamental to our proof and are required not only to establish the existence of solutions on time intervals of the form $(0,T_0]$, but also to obtain uniform decay estimates as $t\searrow 0$. It is worthwhile mentioning here that by relaxing some of these assumptions while strengthening others,
it is still possible to obtain the existence of solutions on the time interval $(0,T_0]$ that do not decay as $t\searrow 0$; see \cite{FOW:2020} for an example of this situation.

A large part of this article will be devoted to extending the global existence result from \cite{Oliynyk:CMP_2016} so that it is applicable to Fuchsian systems with these new types of singular terms. The precise statement of our global existence result for systems of the form \eqref{symivp} is given in Theorem \ref{symthm}. Three applications of this theorem, which involve Fuchsian systems with singular terms of
the form \eqref{newsing.1}-\eqref{newsing.2}, are considered in Section \ref{applications}. There, 
the first two applications involve systems of semilinear wave equations near spatial infinity on Minkowski and Schwarzschild spacetimes. The third concerns perfect fluids on Kasner spacetimes with Gowdy symmetry. See also \cite{Oliynyk:2020} for a recent application of the theory developed in this article. It is worth emphasizing here that it is precisely by generalizing the existence theory from \cite{Oliynyk:CMP_2016} to allow
for singular terms of the form \eqref{newsing.1}-\eqref{newsing.2} that allows for the Fuchsian approach to global existence to become applicable to a wide range of hyperbolic equations that includes wave equations on stationary spacetimes and the relativistic Euler equations
on spacetimes that are not close to FLRW spacetimes with accelerated expansion.

From a general perspective, Theorem \ref{symthm} can be viewed as yielding a classification of possible singular behaviours for solutions of Fuchsian systems with a broad class of singular terms.
It is essential to our applications that the singular terms allowed by Theorem \ref{symthm} go beyond what was allowed  previously in treatments of Fuchsian systems, particularly in regards to allowing for singular coefficients to appear in the spatial derivatives. This possibility to accommodate such
singular terms has been investigated in other recent analyses \cite{ames2019,ringstrom2017} of related, but less singular, \textit{linear} systems. These works have established, among other things, an asymmetry between estimates for evolutions \emph{towards the singular time} and those \emph{away from the singular time}. This asymmetry has the consequence that useful estimates of the first type, which our paper here fully relies on, can often be obtained under more singular conditions than estimates of the second type. It is precisely those latter, more restricted estimates that are necessary to study the singular initial value problem. An additional problem for formulating a singular initial value problem is that one needs to "guess" a leading-order term for the solutions at the singular time, i.e., asymptotic data. The standard approach in the literature is to construct this leading-order term from the ODE-part of the equations, and this only makes sense if the spatial derivative terms are not too singular.

\section{Preliminaries\label{prelim}}

\subsection{Spatial manifolds, coordinates, indexing and partial derivatives}
Throughout this article, unless stated otherwise, $\Sigma$ will denote a closed $n$-dimensional manifold, lower case Latin indices, e.g. $i,j,k$, will run from $1$ to $n$ and will index coordinate indices associated to a local coordinate system $x=(x^i)$ on
$\Sigma$, and $t$ will denote a time coordinate on intervals of the form $[T_0,T_1)$. Partial derivatives with respect to the coordinates will be denoted by
\begin{equation*}
\del{t} = \frac{\partial \;}{\partial t} \AND \del{i} = \frac{\partial \;}{\partial x^i}. 
\end{equation*}

\subsection{Vector Bundles\label{Vbundle}}
In the following, we will let $\pi \: :\: V \longrightarrow \Sigma$
denote a rank $N$ vector bundle with fibres $V_x = \pi^{-1}(\{x\})$, $x\in \Sigma$, and use $\Gamma(V)$ to denote the smooth sections of $V$. 
We will assume that $V$ is equipped
with a time-independent connection\footnote{$[\del{t},\nabla]=0$.} $\nabla$, and a time-independent, compatible\footnote{$\del{t}h=0$ and $\nabla_X (h(u,v)) = h(\nabla_X u,v)
+h(u,\nabla_X v)$ for all  $X\in \mathfrak{X}(\Sigma)$ and $u,v\in \Gamma(V)$.}, positive definite metric $h \in 
\Gamma(T^0_2(V))$. 
We will also denote the vector bundle of linear operators that act on the fibers of $V$ by $L(V)=\cup_{x\in \Sigma} L(V_x) \cong V\otimes V^*$. The \textit{transpose} of $A_x \in L(V_x)$, denoted by $A_x^{\tr}$, is then defined as the unique element of $L(V_x)$ satisfying
\begin{equation*}
h(x)(A_x^{\tr}u_x,v_x)= h(x)(u_x,A_x v_x), \quad \forall \: u_x,v_x\in V_x. 
\end{equation*}
Furthermore, given two vector bundles $V$ and $W$ over $\Sigma$, we will
use $L(V,W) = \cup_{x\in \Sigma} L(V_x,W_x) \cong W \otimes V^*$
to denote the vector bundle of linear maps from the fibres
of $V$ to the fibres of $W$.
For any given vector bundle over $\Sigma$, e.g. $V$, $L(V)$, $V\otimes V$, etc., we will generally use $\pi$ to denote the canonical projection onto $\Sigma$.

Here and below, unless stated otherwise, we will use upper case Latin indices, i.e. $I,J,K$, that run from $1$ to $N$ to index vector bundle associated to a local basis $\{e_I\}$ of $V$. By introducing a local basis $\{e_I\}$, we can represent $u\in \Gamma(V)$ and the inner-product
$h$ locally as
\begin{equation*}
u = u^I e_I \AND
h = h_{IJ}\theta^I\otimes\theta^J,
\end{equation*}
respectively,
where $\{\theta^I\}$ is local basis of $V^*$ determined from $\{e_I\}$ by duality.
Moreover, assuming that the local coordinates $(x^i)$ and the local basis $\{e_I\}$ are defined on the same open region of $\Sigma$,
we can represent the covariant derivative  $\nabla u \in \Gamma(V\otimes T^*\Sigma)$ locally by
\begin{equation*}
\nabla u = \nabla_i u^I e_I \otimes dx^i,
\end{equation*}
where 
\begin{equation*}
\nabla_i u^I = \del{i} u^I + \omega_{iJ}^I u^J,
\end{equation*}
and the $\omega_{iJ}^I$ are the connection coefficients determined, as usual, by
\begin{equation} \label{connect}
\nabla_{\del{i}} e_J = \omega_{iJ}^I e_I.
\end{equation}

We further assume that $\Sigma$ is equipped with a time-independent\footnote{$\del{t}g=0$.}, Riemannian metric $g\in \Gamma(T^0_2(\Sigma))$ that is given locally in coordinates $(x^i)$ by
\begin{equation*}
g = g_{ij} dx^i\otimes dx^j.
\end{equation*}
Since the metric determines the Levi-Civita
connection  on the tensor bundle $T^r_s(\Sigma)$ uniquely, we can, without confusion, use $\nabla$ to also denote the Levi-Civita connection. The connection on $V$ and
the Levi-Civita connection on $T^r_s(\Sigma)$ then determine a connection on the
tensor product $V\otimes T^r_s(\Sigma)$ in a unique fashion, which we will again denote by 
$\nabla$. This connection is compatible with the positive definite inner-product induced on $V\otimes T^r_s(\Sigma)$ by the inner-product $h$ on $V$
and the Riemannian metric $g$ on $\Sigma$. 
With this setup, the \textit{covariant derivative of order $s$} of a section $u \in \Gamma(V)$, denoted $\nabla^s u$, defines an
element of $\Gamma(V\otimes T^0_s(\Sigma))$ that is given locally by
\begin{equation*}
\nabla^s u = \nabla_{i_s} \cdots \nabla_{i_2} \nabla_{i_1} u^I e_I \otimes dx^{i_1}\otimes dx^{i_2}\otimes \cdots \otimes dx^{i_s}.
\end{equation*}
When $s=2$,  the components of $\nabla^2 u$ can be computed using the formula
\begin{equation*}
\nabla_j\nabla_i u^I = \del{j}\nabla_i u^I -\Gamma_{ji}^k\nabla_k u^I + \omega_{j J}^I \nabla_i u^J,
\end{equation*}
where $\nabla_i u^I$ is as defined above, and $\Gamma_{ij}^k$ are the Christoffel symbols of $g$. 
Similar formulas exist for the higher covariant derivatives.

\subsection{Inner-products and operator inequalities}
We define the \textit{norm} of $v\in V_x$, $x\in \Sigma$, by
\begin{equation*}
|v|^2 = h(x)(v,v).
\end{equation*}
Using this norm, we then define, for $R>0$, the \textit{bundle of open balls of radius $R$ in $V$} by 
\begin{equation*}
B_R(V) = \{\, v\in V \, | \, |v|<R\,\}.
\end{equation*}
Given elements $v, w \in V_x\otimes T^0_s(\Sigma_x)$, we can expand them in a local basis as
\begin{equation*}
v= v^I_{i_1 i_2 \cdots i_s} e_I\otimes dx^{i_1}\otimes dx^{i_2}\otimes \cdots \otimes dx^{i_s} \AND w= w^I_{i_1 i_2 \cdots i_s} e_I\otimes dx^{i_1}\otimes dx^{i_2}\otimes \cdots \otimes dx^{i_s},
\end{equation*}
respectively.
Using these expansions, we define the inner-product of $v$ and $w$ by
\begin{equation*}
\ipe{v}{w} = g^{i_1 j_1} g^{i_2 j_2} \cdots g^{i_s j_s} h_{IJ}v^I_{i_1 i_2 \cdots i_s}w^J_{j_1 j_2 \cdots j_s},
\end{equation*}
while the norm of $v$ is defined via
\begin{equation*}
|v|^2 = \ipe{v}{v}.
\end{equation*}

For $A\in L(V_x)$, the \textit{operator norm} of $A$, denoted $|A|_{\op}$, is defined by
\begin{equation*}
|A|_{\op} =\sup\bigl\{\, |\ipe{w}{A v}| \, \bigl| \, w,v \in B_1(V_{x}) \,\bigr\}. 
\end{equation*}
We also define  a related norm for elements $A\in L(V_x) \otimes T_x^*\Sigma$  by
\begin{equation*}
|A|_{\op} = \sup\bigl\{\, |\ipe{v}{Aw}| \, \bigl| \, (v,w) \in B_1(V_x \otimes T_x^*\Sigma)\times B_1(V_{x}) \,\bigr\}. 
\end{equation*}
From this definition, it not difficult to verify that
\begin{equation*}
|A|_{\op} = \sup \bigl\{\, |\ipe{v}{Aw}| \, \bigl|\, (v,w)\in B_1(V_x\otimes T^0_{s+1}(T_x \Sigma))\times B_1(V_x\otimes T^0_{s}(T_x \Sigma))\,\bigr\}.
\end{equation*}
This definition can also be extended to elements of $A\in L(V_x) \otimes T_x^*\Sigma\otimes T_x \Sigma$ in a similar fashion, that is
\begin{equation*}
|A|_{\op} = \sup\bigl\{\, |\ipe{v}{Aw}| \, \bigl| \, (v,w) \in B_1(V_x \otimes T_x^*\Sigma\otimes T_x \Sigma)\times B_1(V_{x}) \,\bigr\},
\end{equation*}
where again we have that
\begin{equation*}
|A|_{\op} = \sup \bigl\{\, |\ipe{v}{Aw}| \, \bigl|\, (v,w)\in B_1(V_x\otimes T^1_{s+1}(T_x \Sigma))\times B_1(V_x\otimes T^0_{s}(T_x \Sigma))\,\bigr\}.
\end{equation*}
Finally, for $A,B \in L(V_x)$, we define
\begin{equation*}
A\leq B
\end{equation*}
if and only if 
\begin{equation*}
\ipe{v}{A v} \leq \ipe{v}{B v}, \quad \forall\: v \in  V_x.
\end{equation*}

\subsection{Constants, inequalities and order notation}
We will use the letter $C$ to denote generic constants whose exact dependence on other quantities is not required and whose value may change from line to line. For such constants,
we will often employ 
the standard notation
\begin{equation*}
a \lesssim b
\end{equation*}
for inequalities of the form
\begin{equation*}
a \leq Cb.
\end{equation*}
On the other hand, when the dependence of the constant on other inequalities needs to be specified, for
example if the constant depends on the norm $\norm{u}_{L^\infty}$, we use the notation
\begin{equation*}
C=C(\norm{u}_{L^\infty}).
\end{equation*}
Constants of this type will always be \textit{non-negative, non-decreasing, continuous functions of their arguments}.

Given four vector bundles $V$, $W$, $Y$ and $Z$ that sit over $\Sigma$, and
maps 
\begin{equation*}
f\in C^0\bigl([T_0,0),C^\infty(B_R(W)\times B_R(V),Z)\bigr)
\AND
g\in C^0\bigl([T_0,0),C^\infty(B_{R}(V),Y)\bigr),
\end{equation*}
we say that
\begin{equation*}
f(t,w,v) = \Ordc(g(t,v))  
\end{equation*}
if there exist a $\Rt \in (0,R)$ and a map
\begin{equation*}
\ft \in C^0\bigl([T_0,0),C^\infty(B_{\Rt}(W)\times B_{\Rt}(V),L(Y,Z))\bigr)
\end{equation*}
such that\footnote{Here, we are using $\nabla_{w,v}$ to denote a covariant derivative operator on the product manifold $W\times V$. Since $\Sigma$ is compact, we know
that such a covariant derivative always exists and it does not matter for our purposes which one is employed.}
\begin{gather*}
f(t,w,v) = \tilde{f}(t,w,v)g(t,v),\\
|\ft(t,w,v)| \leq 1 \AND
|\nabla^s_{w,v}\ft(t,w,v)| \lesssim 1
\end{gather*}
for all $(t,w,v) \in  [T_0,0) \times B_{\Rt}(W) \times B_{\Rt}(V)$ and $s\geq 1$. For situations, where we want to bound $f(t,w,v)$ by $g(t,v)$ up to an undetermined constant of proportionality,
we define
\begin{equation*}
f(t,w,v) = \Ord(g(t,v))  
\end{equation*}
if there exist a $\Rt \in (0,R)$ and a map
\begin{equation*}
\ft \in C^0\bigl([T_0,0),C^\infty(B_R(W)\times B_R(V),L(Y,Z))\bigr)
\end{equation*}
such that 
\begin{gather*}
f(t,w,v) = \tilde{f}(t,w,v)g(t,v)
\intertext{and}
|\nabla^s_{w,v}\ft(t,w,v)| \lesssim 1
\end{gather*}
for all $(t,w,v) \in  [T_0,0) \times B_{\Rt}(W) \times B_{\Rt}(V)$ and $s\geq 0$.

\subsection{Sobolev spaces}
The $W^{k,p}$, $k\in \Zbb_{\geq 0}$, norm of a section $u\in \Gamma(V)$ is defined by
\begin{equation*}
\norm{u}_{W^{k,p}} = \begin{cases} \begin{displaystyle}\biggl( \sum_{\ell=0}^k \int_{\Sigma} |\nabla^{\ell} u|^p \nu_g\biggl)^{\frac{1}{p}}  \end{displaystyle} & \text{if $1\leq p < \infty $} \\
 \begin{displaystyle} \max_{0\leq \ell \leq k}\sup_{x\in \Sigma}|\nabla^{\ell} u(x)|  \end{displaystyle} & \text{if $p=\infty$}
\end{cases},
\end{equation*}
where $\nu_g \in \Omega^n(\Sigma)$ denotes the volume form of $g$. The Sobolev space $W^{k,p}(V)$ can then be defined as the completion of the space of smooth sections $\Gamma(V)$ in the norm
$\norm{\cdot}_{W^{k,p}}$. When $V=\Sigma \times \Rbb$ or the vector bundle is clear from context, we will often write $W^{k,p}(\Sigma)$ instead. 
Furthermore, when $p=2$, we will employ the standard notation $H^k(V)=W^{k,2}(V)$, and we recall that $H^k(V)$ is a Hilbert space with the inner-product given by
\begin{equation*}
\ip{u}{v}_{H^k} = \sum_{\ell=0}^k \ip{\nabla^{\ell} u}{\nabla^{\ell} v},
\end{equation*}  
where the $L^2$ inner-product $\ip{\cdot}{\cdot}$ is defined by
\begin{equation*}
 \ip{w}{z} = \int_\Sigma \ipe{w}{z} \, \nu_g.
\end{equation*}

\section{Singular symmetric hyperbolic systems}

As discussed in the introduction, our aim is to develop an existence theory for initial value problems (IVPs) of the form
\begin{align}
B^0(t,u)\del{t}u + B^i(t,u)\nabla_{i} u  &= \frac{1}{t}\Bc(t,u)\Pbb u + F(t,u) && \text{in $[T_0,T_1)\times \Sigma$,} \label{symivp.1} \\
u &=u_0 && \text{in $\{T_0\}\times \Sigma$,} \label{symivp.2}
\end{align}
where $T_0 < T_1 \leq 0$ and the coefficients satisfy the assumptions set out in the following section. Since these assumption imply, in particular, 
that \eqref{symivp.1} is symmetric hyperbolic, we know that this evolution equation enjoys the Cauchy stability property satisfied by symmetric hyperbolic
equations. As a direct consequence, the existence of solutions to \eqref{symivp.1}-\eqref{symivp.2} when $T_1<0$ is guaranteed for sufficiently small initial data.
Thus the problem that we will focus on is the existence problem when $T_1=0$. Our main result of this article is that we establish the existence of solutions
to \eqref{symivp.1}-\eqref{symivp.2} for $T_1=0$ under a suitable smallness assumption on the initial data. The precise statement of our existence result is given in Theorem \ref{symthm}, which can viewed as a natural generalization of Theorem B.1 from \cite{Oliynyk:CMP_2016}.

\subsection{Coefficient assumptions\label{coeffassump}}
\begin{enumerate}[(i)]
\item The section $\Pbb \in \Gamma(L(V))$ is a time-independent, covariantly constant, symmetric projection
operator, that is,
\begin{equation} \label{Pbbprop}
\Pbb^2 = \Pbb,  \quad  \Pbb^{\tr} = \Pbb, \quad \del{t}\Pbb =0 \AND \nabla \Pbb =0.
\end{equation}
For use below, we define the \textit{complementary projection operator} by
\begin{equation*}
\Pbb^\perp = \id -\Pbb,
\end{equation*}
which by our above assumptions, is also a time-independent,  covariantly constant, symmetric projection operator.
\item There exist constants  $\kappa, \gamma_1, \gamma_2 >0$ such that the maps $B^0 \in 
C^1\bigl([T_0,0), C^\infty(B_R(V),L(V))\bigr)$ 
and $\Bc\in C^0\bigl([T_0,0], C^\infty(B_R(V),L(V))\bigr)$  satisfy
\begin{equation*}
\pi(B^0(t,v))=\pi(\Bc(t,v))=\pi(v), 
\end{equation*}
and
\begin{equation} \label{B0BCbnd}
\frac{1}{\gamma_1} \text{id}_{V_{\pi(v)}} \leq  B^0(t,v)\leq \frac{1}{\kappa} \Bc(t,v) \leq \gamma_2 \textrm{id}_{V_{\pi(v)}}
\end{equation}
for  all $(t,v)\in [T_0,0)\times B_{R}(V)$. Moreover\footnote{Sums such as
$a+bv+cv\otimes v$ should be interpreted as elements of the vector 
bundle
$\Rbb\oplus V \oplus V\otimes V$.},
\begin{align} 
[\Pbb(\pi(v)),\Bc(t,v)] &= 0, \label{BcPbbcom}\\
(B^0(t,v))^{\tr} &= B^0(t,v), \label{B0sym}\\
\Pbb(\pi(v)) B^0(t,v)\Pbb^\perp(\pi(v)) &= \Ord\bigl(|t|^{\frac{1}{2}}+\Pbb(\pi(v)) v\bigr), \label{B0bnd.2}
\intertext{and}
\Pbb^\perp(\pi(v)) B^0(t,v) \Pbb(\pi(v)) &= \Ord\bigl(|t|^{\frac{1}{2}}+\Pbb(\pi(v)) v\bigr),\label{B0bnd.3}
\end{align}
for all $(t,v) \in [T_0,0) \times B_{R}(V)$,
and there exist maps $\Bt^0, \tilde{\Bc} \in C^0\bigl([T_0,0], \Gamma(L(V))\bigr)$
such that
\begin{align}
[\Pbb,\tilde{\Bc}] &=0, \label{BctPbbcom}\\
B^0(t,v)-\Bt^0(t,\pi(v)) &= \Ord(v) \label{B0bnd.1} 
\intertext{and}
\Bc(t,v)-\tilde{\Bc}(t,\pi(v))&=\Ord(v) \label{Bcbnd}
\end{align}
for all $(t,v)\in  [T_0,0)\times B_R(V)$.

\bigskip

\noindent \textit{Note:} In local coordinates $(x,v)=(x^i,v^I)$ on the vector bundle $\pi \: : \: V\longrightarrow \Sigma$, $B^0$ is given
locally by a\footnote{Here, $\Sbb{N}$ denotes the subset of $\Mbb{N}$, i.e. the  $N\times N$-matrices, that are symmetric with respect to the local representation of the vector bundle metric $h$, that is, if $h_{IJ}$ is the local representation of $h$ and
$(h^{IJ}):=(h_{Ij})^{-1}$ is its inverse, then $A^I_J$ will define an element of $\Sbb{N}$ if and only if $h^{IJ}A_J^K=h^{KJ} A_J^I$.} $\Sbb{N}$-valued map $B^0(t,x,v)=\bigl((B^0)^I_J(t,x,v)\bigr)$, while $\Bc$
is given locally by a $\Mbb{N}$-valued map $\Bc(t,x,v)=\bigl(\Bc^I_J(t,x,v)\bigr)$.

\bigskip

\item The map $F\in C^0\bigl([T_0,0), C^\infty(B_R(V),V)\bigr)$ can be expanded as
\begin{equation} \label{fexp}
F(t,v) = \Ft(t,\pi(v)) + F_0(t,v) + |t|^{-\frac{1}{2}}F_1(t,v) + |t|^{-1}F_2(t,v)
\end{equation}
where $\Ft \in C^0\bigl([T_0,0], \Gamma(V)\bigr)$, and
 $F_0,F_1,F_2 \in C^0\bigl([T_0,0], C^\infty(B_R(V),V)\bigr)$
 satisfy
 \begin{equation*}
 \pi (F_a (t,v))=\pi(v), \quad a=0,1,2,
 \end{equation*}
 and
\begin{equation} 
\Pbb(\pi(v)) F_2(t,v) = 0 \label{F2vanish} 
\end{equation}
for all $(t,v)\in [T_0,0]\times B_R(V)$, and
there exist constants $\lambda_a\geq 0$, $a=1,2,3$, such that
\begin{align}
F_0(t,v) &=\Ord(v),  \label{F0bnd}\\
\Pbb(\pi(v)) F_1(t,v) &= \Ordc(\lambda_1 v), \label{F1bnd.1}\\
\Pbb^\perp(\pi(v)) F_1(t,v) &= \Ordc(\lambda_2\Pbb(\pi(v)) v) \label{F1bnd.2}
\intertext{and}
\Pbb^\perp(\pi(v)) F_2(t,v) & = \Ordc\biggl(\frac{\lambda_3}{R}\Pbb(\pi(v)) v\otimes\Pbb(\pi(v))v \biggr) \label{F2bnd.3}
\end{align}
for all $(t,v)\in  [T_0,0)\times B_R(V)$.

\bigskip

\noindent \textit{Note:} In local coordinates $(x,v)=(x^i,v^I)$ on the vector bundle $\pi \: : \: V\longrightarrow \Sigma$,
$\Ft$, $F$ and $F_a$ are given locally by
$\Rbb^{N}$-valued maps $\Ft(t,x)=(\Ft^I(t,x))$, $F(t,x,v)=(F^I(t,x,v))$ and $F_a(t,x,v)=(F_{a}^I(t,x,v))$, respectively.

\bigskip

\item The  map $B\in C^0\bigl([T_0,0), C^\infty(B_R(V),L(V)\otimes T\Sigma)\bigr)$ satisfies
\begin{equation*}
\pi(B(t,v))=\pi(v)
\end{equation*}
and
\begin{equation*}
\bigl[\sigma(\pi(v))(B(t,v))\bigr]^{\tr}=\sigma(\pi(v))(B(t,v))
\end{equation*}
for all $(t,v)\in [T_0,0)\times B_R(V)$ and $\sigma \in \mathfrak{X}^*(\Sigma)$, where
we are using the notation $\sigma(A)$ to denote the natural action of a differential 1-form $\sigma\in \mathfrak{X}^*(\Sigma)$ on an element of
$A\in \Gamma(L(V)\otimes T\Sigma)$, which if we express $A$ and $\sigma$ locally as
\begin{equation*}
A=A^{iI}_J \theta^J\otimes e_I \otimes \del{i} \AND \sigma = \sigma_i dx^i,
\end{equation*} 
respectively, is given by
\begin{equation*}
\sigma(A) = \sigma_i A^{iI}_J  \theta^J\otimes e_I.
\end{equation*}
Moreover, $B$ can be expanded as
\begin{equation} \label{Bexp}
B(t,v) = B_0(t,v) + |t|^{-\frac{1}{2}}B_1(t,v) + |t|^{-1}B_2(t,v)
\end{equation}
where  $B_0,B_1,B_2 \in C^0\bigl([T_0,0], C^\infty(B_R(V),L(V)\otimes T\Sigma)\bigr)$ satisfy
\begin{equation*}
\pi(B_a(t,v))=\pi(v), \quad a=0,1,2,
\end{equation*}
for all  $(t,v)\in [T_0,0]\times B_R(V)$, and
there exist a constant $\alpha\geq 0$ and a map $\Bt_2\in  C^0\bigl([T_0,0], \Gamma(L(V)\otimes T\Sigma)\bigr)$  such that
\begin{align} 
\Pbb(\pi(v)) B_1(t,v) \Pbb(\pi(v)) &=   \Ord(1), \label{BI1bnd.1} \\
\Pbb(\pi(v)) B_1(t,v) \Pbb^\perp(\pi(v))&=  \Ordc(\alpha), \label{BI1bnd.2}\\
 \Pbb^\perp(\pi(v)) B_1(t,v) \Pbb(\pi(v)) &= \Ordc(\alpha), \label{BI1bnd.3} \\
\Pbb^\perp(\pi(v)) B_1(t,v) \Pbb^\perp(\pi(v)) &= \Ord(\Pbb(\pi(v)) v) , \label{BI1bnd.4}\\
\Pbb(\pi(v)) B_2(t,v) \Pbb^\perp(\pi(v)) &= \Ord(\Pbb(\pi(v)) v) , \label{BI2bnd.1} \\
\Pbb^\perp(\pi(v)) B_2(t,v)\Pbb(\pi(v)) &=  \Ord(\Pbb(\pi(v)) v) , \label{BI2bnd.2}\\
\Pbb^\perp(\pi(v)) B_2(t,v) \Pbb^\perp(\pi(v))  &= \Ord\bigl(\Pbb(\pi(v)) v\otimes \Pbb(\pi(v)) v \bigr) \label{BI2bnd.4}
\intertext{and}
\Pbb(\pi(v))(B_2(t,v)-\Bt_2(t,\pi(v)))\Pbb(\pi(v)) &= \Ord(v) \label{BI2bnd.3}
\end{align}
for all $(t,v)\in  [T_0,0)\times B_R(V)$.

\bigskip

\noindent \textit{Note:} In local coordinates $(x,v)=(x^i,v^I)$ on the vector bundle $\pi \: : \: V\longrightarrow \Sigma$, the maps
$B$, $\Bt_2$ and $B_a$ can be expressed as
\begin{equation*}
B=B^i(t,x,v)\del{i}, \quad \Bt_2=\Bt^i_2(t,x)\del{i} \AND B_a=B_a^i(t,x,v)\del{i},
\end{equation*}
respectively, where $B^i(t,x,v)=((B^i)^I_J(t,x,v))$, 
$\Bt^i(t,x)=((\Bt^i_2)^I_J(t,x))$ and $B_a^i(t,x,v)=((B_a^i)^I_J(t,x,v))$ are $\Sbb{N}$-valued maps. Then expressing $\sigma \in \mathfrak{X}^*(\Sigma)$ locally
as
\begin{equation*}
\sigma = \sigma_i(x) dx^i,
\end{equation*}
we see that
\begin{equation*}
\sigma(B) = B^i(t,x,v)\sigma_i(x), \quad \sigma(\Bt_2)=\Bt_2^i(x,t)\sigma_i(x) \AND \sigma(B_a) = B_a^i(t,x,v)\sigma_i(x).
\end{equation*}
We further note that since $u(t,x)$ is a time-dependent section of the vector bundle $V$, it can be represented locally as
\begin{equation*}
u(t,x) = (x,\uh(t,x)),
\end{equation*}
where $\uh(t,x)=(\uh^J(t,x))$ is $\Rbb^N$-valued. Using this and the above local expressions for $B^0$, $\Bc$ and $B$, we can
write the local version of \eqref{symivp.1} as
\begin{align}
B^0(t,x,\uh(t,x))\del{t}\uh(t,x)+& B^i(t,x,\uh(t,x))\bigl(\del{i}\uh(t,x) + \omega_i(x)\uh(t,x)\bigr) \notag \\
&= \frac{1}{t}\Bc(t,x,\uh(t,x))\Pbb(x)\uh(t,x)
+F(t,x,\uh(t,x)), \label{symloc}
\end{align}
where the $\omega_i = (\omega_{iI}^J)$ are the connection coefficients \eqref{connect}.
From our assumptions, it is then clear that \eqref{symloc} defines a symmetric hyperbolic system in the standard sense.

\bigskip

\item  There exist constants $\theta$ and $\beta_a \geq 0$,
$a=0,1,\ldots,7$, such that the map
\begin{equation*}
\Div\! B \: : \: [T_0,0)\times B_R(V\otimes V\otimes T^*\Sigma)  \longrightarrow L(V)
\end{equation*} 
defined locally by
\begin{align} 
&\Div\!  B(t,x,v,w) = \del{t}B^0(t,x,v)+  D_v B^0(t,x,v)\cdot (B^0(t,x,v))^{-1}\Bigl[-B^i(t,x,v)\cdot w_i \notag \\
&\quad+
 \frac{1}{t}\Bc(t,x,v)\Pbb(x) v + F(t,x,v)
\Bigr]
+ \del{i}B^i(t,x,v)+ D_v B^i(t,x,v)\cdot (w_i-\omega_i(x) v) \notag \\
&\quad + \Gamma^i_{ij}(x)B^j(t,x,v)+\omega_i(x)B^i(t,x,v)-B^i(t,x,v)\omega_i(x), \label{divBdef}
\end{align}
where $v=(v^J)$, $w=(w_i)$, $w_i=(w_i^J)$, $\omega_i=(\omega_{iI}^J)$, and $B^i=(B^{iJ}_I)$,
satisfies
\begin{align}
\Pbb(\pi(v)) \Div \! B(t,v,w) \Pbb(\pi(v)) &= \Ordc\bigl(\theta+  |t|^{-\frac{1}{2}}\beta_0 + |t|^{-1}\beta_1\bigl), \label{divBbnd.1}\\
\Pbb(\pi(v))  \Div\! B(t,v,w) \Pbb^\perp(\pi(v)) &= 
\Ordc\biggl(\theta+|t|^{-\frac{1}{2}}\beta_2
+ \frac{|t|^{-1}\beta_3}{R}\Pbb(\pi(v)) v \biggr), \label{divBbnd.2}\\
\Pbb^\perp(\pi(v)) \Div\! B(t,v,w) \Pbb(\pi(v))&= 
\Ordc\biggl(\theta+|t|^{-\frac{1}{2}}\beta_4
+ \frac{|t|^{-1}\beta_5}{R}\Pbb(\pi(v)) v \biggr) \label{divBbnd.3}
\intertext{and}
\Pbb^\perp(\pi(v)) \Div\! B(t,v,w) \Pbb^\perp(\pi(v))& =\Ordc\biggl(\theta+
\frac{|t|^{-\frac{1}{2}}\beta_6}{R}\Pbb(\pi(v)) v
+ \frac{|t|^{-1}\beta_7}{R^2}\Pbb(\pi(v)) v\otimes\Pbb(\pi(v)) v  \biggr). \label{divBbnd.4}
\end{align}

\noindent \textit{Note:} It is not difficult to verify that 
\begin{equation} \label{divBid}
\Div\! B(t,u(t,x),\nabla u(t,x)) =\del{t}(B^0(t,u(t,x))+ \nabla_{i}(B^i(t,u(t,x)))
\end{equation}
for solutions $u(t,x)$ of \eqref{symivp.1}. 
\end{enumerate}

\begin{rem} \label{kappatrem}
By a straightforward calculation, it can be shown that  \eqref{B0bnd.2} and \eqref{B0bnd.3} imply that $(B^0)^{-1}$ satisfies similar relations given by
\begin{align}
\Pbb(\pi(v)) (B^0)^{-1}(t,v)\Pbb^\perp(\pi(v)) &= \Ord\bigl(|t|^{\frac{1}{2}}+\Pbb(\pi(v)) v\bigr) \label{B0invbnd.1}
\intertext{and}
\Pbb^\perp(\pi(v)) (B^0)^{-1}(t,v) \Pbb(\pi(v)) &= \Ord\bigl(|t|^{\frac{1}{2}}+\Pbb(\pi(v)) v\bigr) \label{B0invbnd.2}
\end{align}
for all $(t,v)\in  [T_0,0)\times B_R(V)$.  Moreover, it is clear from \eqref{B0BCbnd} that there exist constants $0<\gammat_1\leq \gamma_1$ and $\kappat\geq \kappa >0$
such that
\begin{equation} \label{B0BCbndt}
\frac{1}{\gammat_1} \Pbb(\pi(v)) \leq  \Pbb(\pi(v))B^0(t,v)\Pbb(\pi(v))\leq \frac{1}{\kappat} \Bc(t,v)\Pbb(\pi(v))  \leq \gamma_2 \Pbb(\pi(v))
\end{equation}
for  all $(t,v)\in [T_0,0)\times B_R(V)$.
\end{rem}

\subsection{Preliminary estimates}

In the following, for $k\in \Zbb_{n/2+1}$, we let $C_{\text{Sob}}>0 $ be the constant from Sobolev's inequality, that is,
\begin{equation} \label{Sobest}
\max\bigl\{\norm{\nabla u(t)}_{L^\infty},\norm{u(t)}_{L^\infty}\bigr\} \leq C_{\text{Sob}}\norm{u(t)}_{H^k}.
\end{equation}

\begin{prop} \label{Pperpprop}
Suppose  $k\in \Zbb_{>n/2+1}$, $u \in  B_{C_{\text{Sob}}^{-1}R}\bigl(H^{k}(V)\bigr)$, $B^0 = B^0(t,u(x))$, $B=B(t,u(x))$,
$F=F(t,u(x))$ and
\begin{equation*}
G = \Pbb^\perp (B^0)^{-1}\biggl(-B^i\nabla_i u + \frac{1}{t} \Bc \Pbb u + F\biggr). 
\end{equation*}
Then\footnote{See the expansion \eqref{fexp} for the definition
of $\Ft$.}
\begin{align*}
\norm{G}_{H^{k-1}} &\leq C\bigl(\norm{u}_{H^k}\bigr)\biggl(\norm{\Ft}_{H^{k-1}}+1+
\frac{1}{|t|^{\frac{1}{2}}}\norm{\Pbb u}_{H^{k}}-\frac{1}{t}\norm{\Pbb u}_{H^k}\norm{\Pbb u}_{H^{k-1}}\biggr),\\
\norm{|t|^{\frac{1}{2}}G}_{L^2} &\leq C\bigl(\norm{u}_{H^k}\bigr)\biggl(\norm{\Ft}_{H^{k-1}}+1 +\frac{1}{|t|^{\frac{1}{2}}}\norm{\Pbb u}_{L^2}\biggr)
\intertext{and}
\norm{|t|^{\frac{1}{2}}G}_{H^{k-1}} &\leq C\bigl(\norm{u}_{H^k}\bigr)\biggl(\norm{\Ft}_{H^{k-1}}+1 +\frac{1}{|t|^{\frac{1}{2}}}\norm{\Pbb u}_{H^{k-1}}\biggr).
\end{align*}
\end{prop}
\begin{proof}
Since $\Pbb^\perp$ satisfies the same properties \eqref{Pbbprop} and \eqref{BcPbbcom} as $\Pbb$, the expansion \eqref{Bexp} allows us to write $\Pbb^\perp (B^0)^{-1}B^i\nabla_i u$ as
\begin{align*}
\Pbb^\perp(B^0)^{-1}B^i\nabla_i u =&  \Pbb^\perp(B^0)^{-1}B^i_0\nabla_i u + |t|^{-\frac{1}{2}}\Bigl[\Pbb^\perp(B^{0})^{-1}\Pbb B_1^i\nabla_i u+
\Pbb^\perp(B^0)^{-1}\Pbb^\perp \bigl(\Pbb^\perp B_1^i \Pbb^\perp \nabla_i u  \notag \\
&+ \Pbb^\perp B_1^i \Pbb \nabla_i \Pbb u \bigr)\Bigr]+ |t|^{-1}\Bigl[\Pbb^\perp (B^0)^{-1}\Pbb\bigl(\Pbb B^i_2 \Pbb \nabla_i \Pbb u + \Pbb B^i_2\Pbb^\perp \nabla_i u \bigr) \notag \\
&+\Pbb^\perp (B^0)^{-1}\Pbb^\perp\bigl(\Pbb^\perp B^i_2 \Pbb \nabla_i \Pbb u + \Pbb^\perp B^i_2\Pbb^\perp \nabla_i u \bigr) \Bigr].
\end{align*} 
From this identity, properties \eqref{B0invbnd.1}-\eqref{B0invbnd.2} of $(B^0)^{-1}$ and \eqref{BI1bnd.1}-\eqref{BI2bnd.4} of the expansion coefficients $B^i_a$, $a=0,1,2$, it is clear that the inequalities
\begin{equation}  \label{PperppropP1a}
\norm{\Pbb^\perp(B^0)^{-1}B^i\nabla_i u}_{H^{k-1}} \leq C(\norm{u}_{H^k})\bigl(1 + |t|^{-\frac{1}{2}}\norm{\Pbb u}_{H^{k}} + |t|^{-1}\norm{\Pbb u}_{H^k}
\norm{\Pbb u}_{H^{k-1}}\bigr),
\end{equation}
\begin{equation}  \label{PperppropP1b}
\norm{|t|^{\frac{1}{2}}\Pbb^\perp(B^0)^{-1}B^i\nabla_i u}_{L^2} \leq C(\norm{u}_{H^k})\bigl(1 + |t|^{-\frac{1}{2}}\norm{\Pbb u}_{L^2} \bigr)
\end{equation}
and
\begin{equation} \label{PperppropP1c}
\norm{|t|^{\frac{1}{2}}\Pbb^\perp(B^0)^{-1}B^i\nabla_i u}_{H^{k-1}} \leq C(\norm{u}_{H^k})\bigl(1 + |t|^{-\frac{1}{2}}\norm{\Pbb u}_{H^{k-1}} \bigr)
\end{equation}
follow directly from the Sobolev, product and Moser calculus inequalities, see Theorems \ref{Sobolev}, \ref{calcpropB} and \ref{calcpropC} from Appendix \ref{calc}. 
Using similar arguments, it is also not difficult to verify that the estimates
\begin{equation}\label{PperppropP2}
\norm{|t|^{-1}\Pbb^\perp (B^0)^{-1}\Bc \Pbb u}_{H^{k-1}} +\norm{ \Pbb^\perp (B^0)^{-1}F}_{H^{k-1}} \leq C(\norm{u}_{H^{k-1}})\Bigl(\norm{\Ft}_{H^{k-1}} + 1 +|t|^{-1}\norm{\Pbb u}^2_{H^{k-1}} \Bigr),
\end{equation} 
\begin{equation}\label{PperppropP3}
\norm{|t|^{-\frac{1}{2}}\Pbb^\perp (B^0)^{-1}\Bc \Pbb u}_{L^2} +\norm{|t|^{\frac{1}{2}} \Pbb^\perp (B^0)^{-1}F}_{L^2} \leq C(\norm{u}_{H^{k-1}})\Bigl(\norm{\Ft}_{L^2} + 1 + 
|t|^{-\frac{1}{2}}\norm{\Pbb u}_{L^2} \Bigr)
\end{equation}
and
\begin{equation}\label{PperppropP3a}
\norm{|t|^{-\frac{1}{2}}\Pbb^\perp (B^0)^{-1}\Bc \Pbb u}_{H^{k-1}} +\norm{|t|^{\frac{1}{2}} \Pbb^\perp (B^0)^{-1}F}_{H^{k-1}} \leq C(\norm{u}_{H^{k-1}})
\Bigl(\norm{\Ft}_{H^{k-1}} + 1 + 
|t|^{-\frac{1}{2}}\norm{\Pbb u}_{H^{k-1}} \Bigr)
\end{equation}
follow from  \eqref{BcPbbcom}, \eqref{B0invbnd.1}-\eqref{B0invbnd.2} and the expansion \eqref{fexp} for $F$ along with properties \eqref{F2vanish}-\eqref{F2bnd.3} of the expansions
coefficients $F_a$, $a=0,1,2$. The  stated estimates for
$G = \Pbb^\perp (B^0)^{-1}\bigl(-B^i\nabla_i u + \frac{1}{t} \Bc \Pbb u + F\bigr)$ are then an immediate consequence of \eqref{PperppropP1b}-\eqref{PperppropP3a}.
\end{proof}

\begin{rem}\label{Gestrem}
The exact form of the estimate for $\norm{G}_{H^{k-1}}$ from  Proposition \ref{Pperpprop} is crucial for determining the decay rate of $\Pbb^\perp u$.
Improvements to this estimate, which will follow from more restrictive assumptions on the coefficients of \eqref{symivp.1}, can lead to corresponding improvements in the decay rate; see Remark \ref{Gdecay} below.
\end{rem}

\begin{prop} \label{FdivBprop} 
Suppose  $k\in \Zbb_{>n/2+1}$, $u \in  B_{C_{\text{Sob}}^{-1}R}\bigl(H^{k}(V)\bigr)$, $v\in L^2(V\otimes T^0_\ell(\Sigma))$, $F=F(t,u(x))$
and $\Div\! B = \Div\!B(t,u(x),\nabla u(x))$. Then
\begin{equation*}
|\ip{u}{F}|\leq \norm{u}_{L^2}\norm{\Ft}_{L^2} + C\norm{u}_{L^2}\norm{u}_{H^k}
+|t|^{-\frac{1}{2}}(\lambda_1+\lambda_2)\norm{u}_{L^2}\norm{\Pbb u}_{L^2}
+|t|^{-1}\lambda_3\norm{\Pbb u}_{L^2}^2,
\end{equation*}
\begin{equation*}
|\ip{\Pbb u}{\Pbb F}|\leq \norm{\Ft}_{L^2}\norm{\Pbb u}_{L^2} +  C \norm{u}_{L^2}\norm{\Pbb u}_{L^2}
+|t|^{-\frac{1}{2}}\lambda_1 \norm{u}_{L^2}\norm{\Pbb u}_{L^2}
\end{equation*}
and
\begin{align*}
|\ip{v}{\Div\! B v}| \leq &4 \theta \norm{v}^2_{L^2}+|t|^{-\frac{1}{2}}\biggl((\beta_0+\beta_2+\beta_4)\norm{v}_{L^2}\norm{\Pbb v}_{L^2}  
+ \frac{\beta_6}{R} \norm{|v|^2 |\Pbb u|}_{L^1}\biggr) \\
&+ |t|^{-1}\biggl(\beta_1 \norm{\Pbb v}_{L^2}^2+ \frac{\beta_3+\beta_5}{R} \norm{|v||\Pbb v| |\Pbb u|}_{L^1} + \frac{\beta_7}{R^2} \norm{|v|^2|\Pbb u|^2}_{L^1}\Bigr)
\end{align*}
where the constants $\theta$, $\lambda_a$, $a=1,2,3$, and $\beta_b$, $b=0,1,\ldots,7$, are as defined in Section \ref{coeffassump}, see
assumptions (iii) and (v).
\end{prop}
\begin{proof}
Using \eqref{fexp} to expand the inner-product $\ip{u}{F}$ as
\begin{equation*}
\ip{u}{F} =\ip{u}{\Ft(t)} +   \ip{u}{F_0(t,u)} + |t|^{-\frac{1}{2}}\ip{u}{F_1(t,u)} + |t|^{-1}\ip{u}{F_2(t,u)},
\end{equation*}
we deduce, with the help of the triangle and Cauchy-Schwartz inequalities, that
\begin{equation} 
|\ip{u}{F(t,u)}| \leq \norm{u}_{L^2}\norm{\Ft(t)}_{L^2} +   \norm{u}_{L^2}\norm{F_0(t,u)}_{L^2} + |t|^{-\frac{1}{2}}|\ip{u}{F_1(t,u)}| + |t|^{-1}|\ip{u}{F_2(t,u)}|.
\label{FdivBpropP1}
\end{equation}
We also observe that
\begin{align*}
\ip{u}{F_1(t,u)} &= \ip{u}{(\Pbb+\Pbb^\perp)F_1(t,u)}  = \ip{\Pbb u}{\Pbb F_1(t,u)} +\ip{\Pbb^\perp u}{\Pbb^\perp F_1(t,u)}.
\end{align*}
Taking the absolute value, we find using \eqref{F1bnd.1} and \eqref{F1bnd.2} along with the triangle and Cauchy-Schwartz inequalities that
\begin{equation} \label{FdivBpropP4}
|\ip{u}{F_1(t,u)}| \leq  \norm{\Pbb u}_{L^2} \lambda_1 \norm{u}_{L^2} + \norm{\Pbb^\perp u}_{L^2} \lambda_2\norm{\Pbb u}_{L^2} \leq (\lambda_1+\lambda_2)\norm{u}_{L^2}\norm{\Pbb u}_{L^2}. 
\end{equation} 

Next, from \eqref{F2bnd.3}, \eqref{Sobest}, and the Cauchy-Schwartz, triangle and Sobolev (Theorem \ref{Sobolev}) inequalities, we have
\begin{equation*}
|\ip{\Pbb^\perp u}{\Pbb^\perp F_2(t,u)}|  
\leq \int_{\Sigma} |\Pbb^\perp u||\Pbb^\perp F_2(t,u)|\, \nu_g \leq \frac{\lambda_3}{R} \norm{u}_{L^\infty} \int_{\Sigma} |\Pbb u|^2\, \nu_g \leq \lambda_3\norm{\Pbb u}_{L^2}^2.
\end{equation*}
Combining this inequality with \eqref{F2vanish} gives 
\begin{equation} \label{FdivBpropP5}
|\ip{u}{F_2(t,u)}| \leq \lambda_3 \norm{\Pbb u}_{L^2}^2.
\end{equation}
The first estimate
\begin{equation*}
|\ip{u}{F}|\leq \norm{u}_{L^2}\norm{\Ft}_{L^2} + C\norm{u}_{L^2}\norm{u}_{H^k}
+|t|^{-\frac{1}{2}}(\lambda_1+\lambda_2)\norm{u}_{L^2}\norm{\Pbb u}_{L^2}
+|t|^{-1}\lambda_3\norm{\Pbb u}_{L^2}^2
\end{equation*}
then follows from \eqref{F0bnd} and \eqref{FdivBpropP1}-\eqref{FdivBpropP5}.
The second estimate can be established in a similar manner.

To establish the third estimate, we decompose the inner-product $\ip{v}{\Div\! B v}$ as 
\begin{equation*}
\ip{v}{\Div\! B v} \leq \ip{\Pbb v}{\Pbb\Div\! B\Pbb \Pbb v}+\ip{\Pbb^\perp v}{\Pbb^\perp \Div\! B\Pbb \Pbb v}+\ip{\Pbb v}{\Pbb\Div\! B\Pbb^\perp \Pbb^\perp v}+
\ip{\Pbb^\perp v}{\Pbb^\perp \Div\! B\Pbb^\perp \Pbb^\perp v}.
\end{equation*}
Then, using similar calculations as above, it is not difficult to see that the estimate
\begin{align*}
|\ip{v}{\Div\! B v}| \leq &4 \theta \norm{v}^2_{L^2}+|t|^{-\frac{1}{2}}\biggl((\beta_0+\beta_2+\beta_4)\norm{v}_{L^2}\norm{\Pbb v}_{L^2}  
+ \frac{\beta_6}{R} \norm{|v|^2 |\Pbb u|}_{L^1}\biggr) \\
&+ |t|^{-1}\biggl(\beta_1 \norm{\Pbb v}_{L^2}^2+ \frac{\beta_3+\beta_5}{R} \norm{|v||\Pbb v| |\Pbb u|}_{L^1} + \frac{\beta_7}{R^2} \norm{|v|^2|\Pbb u|^2}_{L^1}\Bigr)
\end{align*}
is a consequence of the bounds \eqref{divBbnd.1}-\eqref{divBbnd.4} and the Cauchy-Schwartz, triangle and Sobolev inequalities.
\end{proof}

Given $\Asc \in \Gamma(L(V))$ and $v\in \Gamma(V)$, we can expand $\nabla^k(\Asc v)$ as
\begin{equation*}
\nabla^k(\Asc v) = \Asc \nabla^k v + \nabla \Asc \star \nabla^{k-1}v + \sum_{\ell=2}^k \nabla^\ell \Asc \star \nabla^{k-\ell}v
\end{equation*}
where $\nabla^\ell \Asc \star \nabla^{k-\ell}v$ involves sums of tensor products of the covariant derivatives $\nabla^\ell\Asc $ and $\nabla^{k-\ell}v$. In particular,
$\nabla \Asc \star \nabla^{k-1}v$ is given locally by 
\begin{equation*}
(\nabla\Asc \star \nabla^{k-1}v)^I_{i_k i_{k-1}\cdots i_1}
= \sum_{\ell=1}^{k} \nabla_{i_\ell}\Asc^{I}_J \nabla_{i_{k}}\cdots \nabla_{i_{\ell+1}} \nabla_{i_{\ell-1}}\cdots \nabla_{i_1}v^J.
\end{equation*}

\begin{lem} \label{commlem}
Suppose  $k\in \Zbb_{>n/2+2}$, $1\leq \ell \leq k$,  $\Asc \in H^{k}(L(V))$,  $\Bsc \in H^{k-2}(L(V))$,  $\Csc \in H^{k-1}(L(V))$,
$v \in L^{2}(V\otimes T_\ell^0(\Sigma))$ and $w\in H^{k-1}(V)$.  
Then there exists a constant $C>0$ such that
\begin{equation*}
|\ip{v}{\Bsc[\nabla^\ell,\Asc]\Csc w}|\leq \ell\norm{|\Bsc \nabla \Asc \Csc|_{\op}}_{L^\infty}\norm{v}_{L^2}\norm{w}_{H^{k-1}}
+C \norm{\Asc}_{H^{k}}\norm{\Bsc}_{H^{k-2}} \norm{\Csc}_{H^{k-1}} \norm{v}_{L^2}\norm{w}_{H^{k-2}}.
\end{equation*}
\end{lem}
\begin{proof}
Since
\begin{align*}
\Bsc [\nabla^{\ell},\Asc]\Csc w &=\Bsc \nabla^{\ell}(\Asc \Csc w) - \Bsc\Asc \nabla^\ell (\Csc w) \\
& = \Bsc \Asc \nabla^\ell( \Csc w) + \Bsc\nabla \Asc \star \nabla^{\ell-1}(\Csc w) +\Bsc \sum_{m=2}^\ell \nabla^m \Asc \star \nabla^{\ell-m}(\Csc w) -\Bsc  \Asc \nabla^\ell (\Csc w) \\
& = \Bsc\nabla \Asc \star \Csc \nabla^{\ell-1}w +  \Bsc\nabla \Asc \star[ \nabla^{\ell-1},\Csc] w + \Bsc\sum_{m=2}^\ell \nabla^m \Asc \star \nabla^{\ell-m}(\Csc w),
\end{align*}
we get that
\begin{equation} \label{commlemP1}
\ip{v}{\Bsc[\nabla^{\ell},\Asc]\Csc w} = \ip{\Bsc^{\tr}v}{\nabla \Asc \star \Csc \nabla^{\ell-1}w} + \ip{\Bsc^{\tr} v}{\nabla \Asc \star[ \nabla^{\ell-1},\Csc] w}
+ \sum_{m=2}^\ell \ip{\Bsc^{\tr} v}{ \nabla^m \Asc \star \nabla^{\ell-m}(\Csc w)}.
\end{equation}
Using the calculus inequalities (H\"{o}lder, Sobolev, product and commutator - Theorems \ref{Holder}, \ref{Sobolev} and \ref{calcpropB}), we can estimate the second and third terms in the above expression by
\begin{align} 
|\ip{\Bsc^{\tr}v}{\nabla \Asc \star[ \nabla^{\ell-1},\Csc] w}|& \lesssim  \norm{\Bsc}_{L^\infty}\norm{v}_{L^2}\norm{\nabla \Asc}_{L^\infty} \norm{[ \nabla^{\ell-1},\Csc] w}_{L^2} \notag \\
&\lesssim \norm{v}_{L^2} \norm{\Asc}_{H^{k-1}}\norm{\Bsc}_{H^{k-2}}\norm{\Csc}_{H^{k-1}}
\norm{w}_{H^{k-2}}\label{commlemP2}
\end{align}
and
\begin{align}
\biggl| \sum_{m=2}^\ell \ip{\Bsc^{\tr}v}{ \nabla^m\Asc \star \nabla^{\ell-m}(\Csc w)}\biggr| 
&\lesssim  \sum_{m=2}^\ell\norm{\Bsc}_{L^\infty} \norm{v}_{L^2}\norm{\nabla^m \Asc \nabla^{\ell-m}(\Csc w)}_{L^2} \notag \\
&\lesssim \sum_{m=2}^\ell \norm{v}_{L^2} \norm{\Bsc}_{H^{k-2}} \norm{\nabla^m\Asc}_{H^{\ell-m}}\norm{\nabla^{\ell-m}(\Csc w)}_{H^{k-2-(\ell-m)}}\notag \\
&\lesssim \norm{v}_{L^2}\norm{\Asc}_{H^{\ell}}\norm{\Bsc}_{H^{k-2}} \norm{\Csc w}_{H^{k-2}} \notag \\
&\lesssim \norm{v}_{L^2}\norm{\Asc}_{H^{k}}\norm{\Bsc}_{H^{k-2}} \norm{\Csc}_{H^{k-2}} \norm{w}_{H^{k-2}}, \label{commlemP3}
\end{align}
respectively.

Noting that we can write the first term in \eqref{commlemP1} as 
\begin{equation*}
 \ip{v}{\Bsc \nabla \Asc \star \Csc \nabla^{\ell-1} w} =\sum_{m=1}^{\ell} \int_\Sigma \ipe{S_m v}{\Bsc \nabla \Asc \Csc \nabla^{\ell-1}w}\, \nu_g,
\end{equation*}
where the $S_m$ are linear maps that rearranges the order of the covariant tensor indices, we see that
\begin{align}
 \bigl|\ip{v}{\Bsc \nabla \Asc \star \Csc \nabla^{\ell-1} w}\bigr| &\leq  \sum_{m=1}^{\ell}\int_\Sigma\bigl|\ipe{S_m v}{\Bsc \nabla \Asc \Csc \nabla^{\ell-1}w}\bigr|\, \nu_g
 && \text{(by the triangle inequality)}\notag\\
 &\leq  \sum_{m=1}^{\ell}\int_\Sigma |\Bsc \nabla \Asc \Csc|_{\op} |S_m v||\nabla^{\ell-1}w|\, \nu_g && \text{ (by definition of the norm $|\cdot|_{\op}$)}\notag \\
 &\leq  \ell \int_\Sigma |\Bsc \nabla \Asc \Csc|_{\op} |v||\nabla^{\ell-1}w|\, \nu_g && \text{ (since $|S_m v|=|v|$)}\notag \\
 &\leq \ell \||\Bsc \nabla \Asc \Csc|_{\op}\|_{L^\infty}\norm{v}_{L^2}\norm{w}_{H^{k-1}}, \label{commlemP4}
\end{align}
where in deriving the last inequality we used H\"{o}lder's inequality. The proof now follows \eqref{commlemP1}, \eqref{commlemP2}, \eqref{commlemP3}, \eqref{commlemP4}
and the triangle inequality.
\end{proof}

\begin{prop} \label{Bcommprop}
Suppose  $k\in \Zbb_{>n/2+2}$, $1\leq \ell \leq k$, $v\in L^2(V\otimes T_\ell^0(\Sigma) )$, $u \in  B_{C_{\text{Sob}}^{-1}R}\bigl(H^{k}(V)\bigr)$,
$\Bc=\Bc(t,u(x))$, $B^0=B^0(t,u(x))$ and $B=B(t,u(x))$. Then
\begin{gather*}
|\ip{v}{\Bc \nabla^\ell (\Bc^{-1}F)}| + |\ip{v}{\Bc[\nabla^\ell,\Bc^{-1}B^0](B^{0})^{-1}F}| \leq C\bigl(\norm{v}_{L^2}
\norm{\Ft}_{H^k} + \Xi\bigl), \\
|\ip{v}{\Bc[\nabla^\ell,\Bc^{-1}B^0](B^{0})^{-1} t^{-1}\Bc \Pbb u}| \leq C\bigl(|t|^{-1}\norm{\Pbb v}_{L^2}\norm{\Pbb u}_{H^{k-1}}
+\Xi\bigr),
\end{gather*}
\begin{align*}
|\ip{v}{\Bc [\nabla^\ell,\Bc^{-1}B^i]\nabla_{i}u}| + |\ip{v}{\Bc[\nabla^\ell,\Bc^{-1}B^0](B^{0})^{-1}B^i\nabla_{i} u}|
\leq& |t|^{-1}\Bigl(\ell\mathtt{b}\norm{\Pbb v}_{L^2}\norm{\Pbb u}_{H^k} \\
&+ C\norm{\Pbb v}_{L^2}\norm{\Pbb u}_{H^{k-1}}\Bigr)
+ C\Xi,
\end{align*}
and
\begin{equation*}
|\ip{v}{B^i [\nabla^\ell,\nabla_i]u}| \leq C\bigl(|t|^{-1}\norm{\Pbb v}_{L^2}\norm{\Pbb u}_{H^{k-1}}+ \Xi \bigr),
\end{equation*}
where $C=C(\norm{u}_{H^k})$,
\begin{equation*}
\Xi = \norm{v}_{L^2}\norm{u}_{H^k}+|t|^{-\frac{1}{2}}\Bigl(\norm{v}_{L^2}\norm{\Pbb u}_{H^k}
+\norm{\Pbb v}_{L^2}\norm{u}_{H^k}\Bigr)+|t|^{-1}\Bigl(\norm{v}_{L^2}\norm{\Pbb u}^2_{H^k} + 
\norm{\Pbb v}_{L^2}\norm{u}_{H^k}\norm{\Pbb u}_{H^k}\Bigl)
\end{equation*}
and
\begin{align*}
\mathtt{b} =  \sup_{T_0 \leq t < 0}\Bigl( \big\|\big|\Pbb\tilde{\Bc}(t)\nabla(\tilde{\Bc}(t)^{-1}\tilde{B}^0(t))\tilde{B}^0(t)^{-1}\Pbb \tilde{B}_2(t)\Pbb\bigl|_{\op}\bigr\|_{L^\infty}
+ \big\|\big|\Pbb\tilde{\Bc}(t)\nabla(\tilde{\Bc}(t)^{-1}\tilde{B}_2(t))\Pbb\bigl|_{\op}\bigr\|_{L^\infty} \Bigr).
\end{align*}
\end{prop}
\begin{proof}
Using \eqref{Pbbprop} and \eqref{BcPbbcom},
we calculate
\begin{align*}
\ip{v}{\Bc[\nabla^\ell,\Bc^{-1}B^0](B^0)^{-1}G} =& \ip{\Pbb v}{\Bc[\nabla^\ell,\Bc^{-1}B^0](B^0)^{-1}\Pbb G}
+ \ip{\Pbb v}{\Bc[\nabla^\ell,\Bc^{-1}B^0](B^0)^{-1}\Pbb^\perp G} \notag \\
&+\ip{\Pbb^\perp v}{\Bc[\nabla^\ell,\Bc^{-1}B^0](B^0)^{-1}\Pbb G}
+\ip{\Pbb^\perp v}{\Bc[\nabla^\ell,\Bc^{-1}B^0](B^0)^{-1}\Pbb^\perp G} \notag\\
 =& \ip{\Pbb v}{\Bc[\nabla^\ell,\Bc^{-1}B^0](B^0)^{-1}\Pbb G}
+ \ip{\Pbb v}{\Bc[\nabla^\ell,\Bc^{-1}\Pbb B^0\Pbb^\perp](B^0)^{-1}\Pbb^\perp G} \notag \\
&+ \ip{\Pbb v}{\Bc[\nabla^\ell,\Bc^{-1}\Pbb B^0]\Pbb (B^0)^{-1}\Pbb^\perp G} 
+\ip{\Pbb^\perp v}{\Bc[\nabla^\ell,\Bc^{-1}\Pbb^\perp B^0\Pbb](B^0)^{-1}\Pbb G} \notag \\
&+\ip{\Pbb^\perp v}{\Bc[\nabla^\ell,\Bc^{-1}\Pbb^\perp B^0]\Pbb^\perp (B^0)^{-1}\Pbb G}
+\ip{\Pbb^\perp v}{\Bc[\nabla^\ell,\Bc^{-1}B^0](B^0)^{-1}\Pbb^\perp G}. 
\end{align*}
From this identity, we see, with the help of the  calculus inequalities from Appendix \ref{calc} and  the properties  \eqref{B0bnd.2}-\eqref{B0bnd.3} and \eqref{B0invbnd.1}-\eqref{B0invbnd.2}
of
$B^0$, that $\ip{v}{\Bc[\nabla^\ell,\Bc^{-1}B^0](B^0)^{-1}G}$ can be estimated by
\begin{align}
\bigl|\ip{v}{\Bc[\nabla^\ell,\Bc^{-1}B^0](B^0)^{-1}G}\bigr| \leq& 
\bigl|\ip{\Pbb v}{\Bc[\nabla^\ell,\Bc^{-1}B^0](B^0)^{-1}\Pbb G}\bigr|+ C(\norm{u}_{H^k})\Bigl[\norm{\Pbb v}_{L^2}\bigl(\norm{\Pbb u}_{H^k}+|t|^{\frac{1}{2}}\bigr)
\norm{\Pbb^\perp G}_{H^{k-1}} \notag \\
&+ \norm{\Pbb^\perp v}_{L^2}\bigl(\norm{\Pbb u}_{H^k}+|t|^{\frac{1}{2}}\bigr)
\norm{\Pbb G}_{H^{k-1}}+\norm{\Pbb^\perp v}_{L^2}\norm{\Pbb^\perp G}_{H^{k-1}}
\Bigr]. \label{BcommpropP1}
\end{align}
Using this estimate and another application of the calculus inequalities, it then follows from the expansion \eqref{fexp} for $F$ and properties \eqref{F2vanish}-\eqref{F2bnd.3}
of the expansions coefficients $F_a$, $a=0,1,2$, that we can bound
 $\ip{v}{\Bc[\nabla^\ell,\Bc^{-1}B^0](B^0)^{-1}F}$  by
\begin{align}
\bigl|\ip{v}{\Bc[\nabla^\ell,\Bc^{-1}B^0](B^0)^{-1}F}\bigr| \leq& C(\norm{u}_{H^k})\Bigl[
\norm{v}_{L^2}\bigl(\norm{\Ft}_{H^k}+\norm{u}_{H^k}\bigr) \notag \\
&+ |t|^{-\frac{1}{2}}\bigl(\norm{v}_{L^2}\norm{\Pbb u}_{H^k}
+\norm{\Pbb v}_{L^2}\norm{u}_{H^k}\bigr)+|t|^{-1}\norm{v}_{L^2}\norm{\Pbb u}^2_{H^k}\Bigr]. \label{BcommpropP2}
\end{align}
We also observe, with the help of $[\Bc,\Pbb]=0$ and $\nabla\Pbb =0$, that 
\begin{align*}
\ip{v}{\Bc \nabla^\ell (\Bc^{-1}F)} =& \ip{\Pbb v}{\Bc \nabla^\ell (\Bc^{-1}\Pbb F)} + \ip{\Pbb^\perp v}{\Bc \nabla^\ell (\Bc^{-1}\Pbb^\perp F)}.
\end{align*}
Using this in conjunction with the calculus inequalities, the expansion \eqref{fexp} for $F$, and the properties \eqref{F2vanish}-\eqref{F2bnd.3}
of the expansions coefficients $F_a$, $a=0,1,2$, we then get
\begin{align}
\bigl|\ip{v}{\Bc \nabla^\ell (\Bc^{-1}F)}\bigr| \leq& C(\norm{u}_{H^k})\Bigl[
\norm{v}_{L^2}\bigl(\norm{\Ft}_{H^k}+\norm{u}_{H^k}\bigr) \notag \\
&+ |t|^{-\frac{1}{2}}\bigl(\norm{v}_{L^2}\norm{\Pbb u}_{H^k}
+\norm{\Pbb v}_{L^2}\norm{u}_{H^k}\bigr)+|t|^{-1}\norm{v}_{L^2}\norm{\Pbb u}^2_{H^k}\Bigr]. \label{BcommpropP3}
\end{align}

Appealing again to the calculus inequalities, we find from \eqref{B0bnd.1}, \eqref{Bcbnd}, 
\eqref{Bexp} - \eqref{BI2bnd.3}
that the term
$\bigl|\ip{\Pbb v}{\Bc[\nabla^\ell,\Bc^{-1}B^0](B^0)^{-1}\Pbb B^i\nabla_{i}u}\bigr|$
can be estimated by
\begin{align}
\bigl|\ip{\Pbb v}{\Bc[\nabla^\ell,\Bc^{-1}B^0]&(B^0)^{-1}\Pbb B^i \nabla_i u}\bigr| \leq
\bigl|\ip{\Pbb v}{\Bc[\nabla^\ell,\Bc^{-1}B^0](B^0)^{-1}\Pbb B^i\Pbb \nabla_i \Pbb u}\bigr|  \notag \\
& + \bigl|\ip{\Pbb v}{\Bc[\nabla^\ell,\Bc^{-1}B^0](B^0)^{-1}\Pbb B^i \Pbb^\perp 
\nabla_i u}\bigr|
\notag \\
\leq &|t|^{-1} \bigl|\ip{\Pbb v}{\tilde{\Bc}[\nabla^\ell,\tilde{\Bc}^{-1}\tilde{B}^0](\tilde{B}^0)^{-1}\Pbb \tilde{B}_2^i \nabla_i \Pbb u}\bigr| +C(\norm{u}_{H^k})\Bigl[
\norm{v}_{L^2}\norm{u}_{H^k} \notag \\
&  +
|t|^{-\frac{1}{2}}\norm{\Pbb v}_{L^2}\norm{u}_{H^k} +  |t|^{-1}
\norm{\Pbb v}_{L^2}\norm{u}_{H^k}\norm{\Pbb u}_{H^k}\Bigl] \notag \\
\leq&  |t|^{-1} \ell \big\|\big|\Pbb\tilde{\Bc}\nabla(\tilde{\Bc}^{-1}\tilde{B}^0)(\tilde{B}^0)^{-1}\Pbb \tilde{B}_2\Pbb\bigl|_{\op}\bigr\|_{L^\infty}
\norm{\Pbb v}_{L^2}\norm{\Pbb u}_{H^k} +C(\norm{u}_{H^k})\Bigl[ 
\norm{v}_{L^2}\norm{u}_{H^k}\notag \\
&  +
|t|^{-\frac{1}{2}}\norm{\Pbb v}_{L^2}\norm{u}_{H^k} +  |t|^{-1}\bigl( \norm{\Pbb v}_{L^2}\norm{\Pbb u}_{H^{k-1}}+
\norm{\Pbb v}_{L^2}\norm{u}_{H^k}\norm{\Pbb u}_{H^k}\bigr)\Bigl],\notag 
\end{align}
where in deriving the last inequality we used Lemma \ref{commlem}. Using this, the inequality
\begin{align}
\bigl|\ip{v}{\Bc[\nabla^\ell,\Bc^{-1}B^0](B^0)^{-1}B^i &\nabla_i u}\bigr| \leq  |t|^{-1} \ell \big\|\big|\Pbb\tilde{\Bc}\nabla(\tilde{\Bc}^{-1}\tilde{B}^0)(\tilde{B}^0)^{-1}\Pbb \tilde{B}_2\Pbb\bigl|_{\op}\bigr\|_{L^\infty}
\norm{\Pbb v}_{L^2}\norm{\Pbb u}_{H^k} \notag \\
&+C(\norm{u}_{H^k})\Bigl[ 
\norm{v}_{L^2}\norm{u}_{H^k} +
|t|^{-\frac{1}{2}}\bigl(\norm{\Pbb v}_{L^2}\norm{u}_{H^k}+\norm{v}_{L^2}\norm{\Pbb u}_{H^k}\bigr) \notag\\
&+  |t|^{-1}\bigl( \norm{\Pbb v}_{L^2}\norm{\Pbb u}_{H^{k-1}}+\norm{v}_{L^2}\norm{\Pbb u}^2_{H^k}+
\norm{\Pbb v}_{L^2}\norm{u}_{H^k}\norm{\Pbb u}_{H^k}\bigr)\Bigl] \label{BcommpropP4}
\end{align}
is then easily seen to be a direct consequence of the expansion \eqref{Bexp}, the properties \eqref{BI1bnd.4}-\eqref{BI2bnd.4} of $B_1$ and $B_2$, and the estimate \eqref{BcommpropP1}.
By similar arguments as above, it is also not difficult to verify the following estimates hold:
\begin{equation}
\bigl|\ip{v}{\Bc[\nabla^\ell,\Bc^{-1} B^0](B^0)^{-1} \Bc\Pbb u}\bigr| \leq C(\norm{u}_{H^k})\Bigl[\norm{\Pbb v}_{L^2}\norm{\Pbb u}_{H^{k-1}}+ |t|^{\frac{1}{2}} \norm{v}_{L^2}\norm{\Pbb u}_{H^{k}}
+\norm{v}_{L^2}\norm{\Pbb u}_{H^k}^2\Bigr] \label{BcommpropP5}
\end{equation}
and
\begin{align}
\bigl|\ip{v}{\Bc[\nabla^\ell,\Bc^{-1}B^i]\nabla_i u}\bigr|
\leq&  |t|^{-1} \ell \big\|\big|\Pbb\tilde{\Bc}\nabla(\tilde{\Bc}^{-1}\tilde{B}_2)\Pbb\bigl|_{\op}\bigr\|_{L^\infty}
\norm{\Pbb v}_{L^2}\norm{\Pbb u}_{H^k} \notag \\
&+C(\norm{u}_{H^k})\Bigl[ 
\norm{v}_{L^2}\norm{u}_{H^k} +
|t|^{-\frac{1}{2}}\bigl(\norm{\Pbb v}_{L^2}\norm{u}_{H^k}+\norm{v}_{L^2}\norm{\Pbb u}_{H^k}\bigr) \notag\\
&+  |t|^{-1}\bigl( \norm{\Pbb v}_{L^2}\norm{\Pbb u}_{H^{k-1}}+\norm{v}_{L^2}\norm{\Pbb u}^2_{H^k}+
\norm{\Pbb v}_{L^2}\norm{u}_{H^k}\norm{\Pbb u}_{H^k}\bigr)\Bigl] \label{BcommpropP6}.
\end{align}

To complete the proof, we note, since $\nabla \Pbb =0$ by assumption, that $[\nabla^\ell,\nabla_i]$ is a differential operator of order $\ell-1$ that commutes with $\Pbb$. This fact allows us
to write  $\ip{v}{B^i [\nabla^\ell,\nabla_i]u}$ as
\begin{align*}
\ip{v}{B^i [\nabla^\ell,\nabla_i]u} = &\ip{\Pbb v}{\Pbb B^i \Pbb [\nabla^\ell,\nabla_i]\Pbb u} + \ip{\Pbb v}{\Pbb B^i \Pbb^\perp [\nabla^\ell,\nabla_i]\Pbb^\perp u} \\
&
+\ip{\Pbb^\perp v}{\Pbb^\perp B^i \Pbb [\nabla^\ell,\nabla_i]\Pbb u}+\ip{\Pbb^\perp v}{\Pbb^\perp B^i \Pbb^\perp [\nabla^\ell,\nabla_i]\Pbb^\perp u}.
\end{align*}
This identity together with the expansion \eqref{Bexp} for $B^i$ and the properties \eqref{BI1bnd.4}-\eqref{BI2bnd.3} of the expansion coefficients $B_1$ and $B_2$ shows,
with the help of the calculus inequalities, that $\ip{v}{B^i [\nabla^\ell,\nabla_i]u}$ can be estimated by
\begin{align}
|\ip{v}{B^i [\nabla^\ell,\nabla_i]u}|\leq &C(\norm{u}_{H^k})\Bigl[ \norm{v}_{L^2}\norm{u}_{H^k}+|t|^{-\frac{1}{2}}\Bigl(\norm{v}_{L^2}\norm{\Pbb u}_{H^k}
+\norm{\Pbb v}_{L^2}\norm{u}_{H^k}\Bigr)\notag \\
&+|t|^{-1}\Bigl(\norm{\Pbb v}_{L^2}\norm{\Pbb u}_{H^{k-1}}+\norm{v}_{L^2}\norm{\Pbb u}^2_{H^k} + 
\norm{\Pbb v}_{L^2}\norm{u}_{H^k}\norm{\Pbb u}_{H^k}\Bigl)\Bigr]. \label{BcommpropP7}
 \end{align}
The stated estimates now follow directly from the inequalities \eqref{BcommpropP2}-\eqref{BcommpropP7}.
\end{proof}

Similar arguments can be used to establish the following variation of the above proposition.

\begin{prop} \label{BcommpropA}
Suppose  $k\in \Zbb_{>n/2+3}$, $1\leq \ell \leq k-1$, $v\in L^2(V\otimes T_\ell^0(\Sigma) )$, $u \in  B_{C_{\text{Sob}}^{-1}R}\bigl(H^{k}(V)\bigr)$,
$\Bc=\Bc(t,u(x))$, $B^0=B^0(t,u(x))$ and $B=B(t,u(x))$. Then
\begin{gather*}
|\ip{v}{\Bc\Pbb [\nabla^\ell,\Bc^{-1}\Pbb B^0 \Pbb]\Pbb (B^{0})^{-1}F}| \leq C\bigl(\norm{v}_{L^2}
\norm{\Ft}_{H^k} + \Theta\bigl), \\
|\ip{v}{\Bc\Pbb[\nabla^\ell,\Bc^{-1}\Pbb B^0\Pbb]\Pbb(B^{0})^{-1} t^{-1}\Bc \Pbb u}| \leq C\bigl(|t|^{-1}\norm{\Pbb v}_{L^2}\norm{\Pbb u}_{H^{k-2}}
+\Theta\bigr),
\end{gather*}
\begin{align*}
|\ip{v}{\Bc\Pbb [\nabla^\ell,\Bc^{-1}\Pbb B^i\Pbb]\nabla_{i}\Pbb u}| + |\ip{v}{\Bc \Pbb[\nabla^\ell,\Bc^{-1}\Pbb B^0 \Pbb]\Pbb(B^{0})^{-1}B^i\nabla_{i} u}|
\leq& |t|^{-1}\Bigl(\ell\tilde{\mathtt{b}}\norm{\Pbb v}_{L^2}\norm{\Pbb u}_{H^{k-1}} \\
&+ C\norm{\Pbb v}_{L^2}\norm{\Pbb u}_{H^{k-2}}\Bigr)
+ C\Theta,
\end{align*}
and
\begin{equation*}
|\ip{v}{\Pbb B^i \Pbb [\nabla^\ell,\nabla_i]\Pbb u}| \leq C\bigl(|t|^{-1}\norm{\Pbb v}_{L^2}\norm{\Pbb u}_{H^{k-2}}+ \Theta \bigr),
\end{equation*}
where $C=C(\norm{u}_{H^k})$,
\begin{align*}
\Theta = \norm{\Pbb v}_{L^2}\norm{u}_{H^{k-1}}+ |t|^{-\frac{1}{2}}\Bigl((\lambda_1+\alpha)\norm{\Pbb v}_{L^2}\norm{u}_{H^{k-1}}
+\norm{\Pbb v}_{L^2}\norm{\Pbb u}_{H^{k-1}}\Bigr)+|t|^{-1}\norm{\Pbb v}_{L^2}\norm{u}_{H^{k-1}}\norm{\Pbb u}_{H^{k-1}}
\end{align*}
and
\begin{align*}
\tilde{\mathtt{b}} =  \sup_{T_0 \leq t < 0}\Bigl( \big\|\big|\Pbb\tilde{\Bc}(t)\nabla(\tilde{\Bc}(t)^{-1}\Pbb \tilde{B}^0(t) \Pbb)\Pbb \tilde{B}^0(t)^{-1} \tilde{B}_2(t)\Pbb\bigl|_{\op}\bigr\|_{L^\infty}
+ \big\|\big|\Pbb\tilde{\Bc}(t)\nabla(\tilde{\Bc}(t)^{-1}\tilde{B}_2(t))\Pbb\bigl|_{\op}\bigr\|_{L^\infty} \Bigr).
\end{align*}
\end{prop}

\subsection{Global existence and asymptotics\label{global}}
The main result of this article is the following
theorem that guarantees the existence of solutions
to the IVP \eqref{symivp.1}-\eqref{symivp.2} on time intervals of the form $[T_0,0)$ under a suitable small initial data hypothesis. The theorem also yields decay estimates for the solutions; see Remark \ref{Gdecay} below for additional information on improvements to the decay rate for $\Pbb^\perp u$. Applications of this theorem can be found in Section \ref{applications}. 

\begin{thm} \label{symthm}
Suppose $k \in \Zbb_{>n/2+3}$, $\sigma>0$, $u_0\in H^k(\Sigma)$, assumptions (i)-(v) from Section \ref{coeffassump} are fulfilled,
and the constants $\kappa$, $\gamma_1$, $\lambda_3$, $\beta_0$, $\beta_1$, $\beta_3$, $\beta_5$, $\beta_7$ from Section \ref{coeffassump} satisfy
\begin{equation}
\label{eq:symhypkappa}
\kappa > \frac{1}{2}\gamma_1\max\biggl\{\sum_{a=0}^3  \beta_{2a+1}+2\lambda_3,\beta_1+ 2k(k+1)\mathtt{b}\biggr\}
\end{equation}
where
\begin{align*}
\mathtt{b} =  \sup_{T_0 \leq t < 0}\Bigl( \big\|\big|\Pbb\tilde{\Bc}(t)\nabla(\tilde{\Bc}(t)^{-1}\tilde{B}^0(t))\tilde{B}^0(t)^{-1}\Pbb \tilde{B}_2(t)\Pbb\bigl|_{\op}\bigr\|_{L^\infty}
+ \big\|\big|\Pbb\tilde{\Bc}(t)\nabla(\tilde{\Bc}(t)^{-1}\tilde{B}_2(t))\Pbb\bigl|_{\op}\bigr\|_{L^\infty} \Bigr).
\end{align*}
Then there exists
a $\delta > 0$ such that if 
\begin{equation*}
 \max\Bigr\{\norm{u_0}_{H^k},\sup_{T_0 \leq \tau < 0}\norm{\Ft(\tau)}_{H^k}\Bigr\}
 \leq \delta,
\end{equation*}
then there exists a unique solution 
\begin{equation*}
u \in C^0\bigl([T_0,0),H^k(\Sigma)\bigr)\cap L^\infty\bigl([T_0,0),H^k(\Sigma)\bigr)\cap C^1\bigl([T_0,0),H^{k-1}(\Sigma)\bigr)
\end{equation*}
of the IVP \eqref{symivp.1}-\eqref{symivp.2} with $T_1=0$ such that the limit $\lim_{t\nearrow 0} \Pbb^\perp u(t)$, denoted $\Pbb^\perp u(0)$, exists in $H^{k-1}(\Sigma)$.

\medskip

\noindent Moreover, for $T_0 \leq t < 0$,  the solution $u$ satisfies  the energy estimate
\begin{equation*}
\norm{u(t)}_{H^k}^2+ \sup_{T_0 \leq \tau < 0}\norm{\Ft(\tau)}_{H^k}^2- \int_{T_0}^t \frac{1}{\tau} \norm{\Pbb u(\tau)}_{H^k}^2\, d\tau   \leq C(\delta,\delta^{-1})\Bigl( \norm{u_0}_{H^k}^2 + \sup_{T_0 \leq \tau < 0}\norm{\Ft(\tau)}_{H^k}^2\Bigr)
\end{equation*}
and the decay estimates
\begin{align*}
\norm{\Pbb u(t)}_{H^{k-1}} &\lesssim \begin{cases}
|t|+(\lambda_1+\alpha)|t|^{\frac{1}{2}} & \text{if $\zeta > 1$} \\
|t|^{\zeta-\sigma}+(\lambda_1+\alpha)|t|^{\frac{1}{2}} & \text{if $\frac{1}{2} < \zeta \leq 1$} \\
|t|^{\zeta-\sigma}   & \text{if $0 < \zeta \leq \frac{1}{2}$}\end{cases}
\intertext{and}
\norm{\Pbb^\perp u(t) - \Pbb^\perp u(0)}_{H^{k-1}} &\lesssim
 \begin{cases}  |t|^{\frac{1}{2}}+ |t|^{\zeta-\sigma}
& \text{if $\zeta > \frac{1}{2}$} \\
 |t|^{\zeta-\sigma}  &  \text{if $\zeta \leq \frac{1}{2} $ }
 \end{cases},
\end{align*}
where\footnote{Recall from Remark \ref{kappatrem} that $\kappat \geq \kappa >0$ and  $\gamma_1\geq \gammat_1 >0$, while it is clear from the definitions of
$\mathtt{b}$ and $\tilde{\mathtt{b}}$ that
$\mathtt{b} \geq \tilde{\mathtt{b}} \geq 0$. Thus,  $\zeta = \kappat -\frac{1}{2}\gammat_1\bigl(\beta_1+ (k-1)k\tilde{\mathtt{b}}\bigr) >0$ since
 $\kappa - \frac{1}{2}\gamma_1\max\bigl\{\sum_{a=0}^3  \beta_{2a+1}+2\lambda_3,\beta_1+ 2k(k+1)\mathtt{b}\bigr\}>0$ by assumption.}
\begin{equation}
\label{eq:defzeta}
\zeta = \kappat -\frac{1}{2}\gammat_1\bigl(\beta_1+ (k-1)k\tilde{\mathtt{b}}\bigr)
\end{equation}
and
\begin{equation*}
\tilde{\mathtt{b}} =  \sup_{T_0 \leq t < 0}\Bigl( \big\|\big|\Pbb\tilde{\Bc}(t)\nabla(\tilde{\Bc}(t)^{-1}\Pbb \tilde{B}^0(t) \Pbb)\Pbb \tilde{B}^0(t)^{-1} \tilde{B}_2(t)\Pbb\bigl|_{\op}\bigr\|_{L^\infty}
+ \big\|\big|\Pbb\tilde{\Bc}(t)\nabla(\tilde{\Bc}(t)^{-1}\tilde{B}_2(t))\Pbb\bigl|_{\op}\bigr\|_{L^\infty} \Bigr).
\end{equation*}
\end{thm}
\begin{proof}
Since $k\in \Zbb_{>n/2+3}$, we know by standard local-in-time existence and uniqueness results for symmetric hyperbolic equations, e.g.
\cite[Ch.16 \S 1]{TaylorIII:1996}, that there exists a solution $u \in C^0([T_0,T^*),H^k)\cap C^1([T_0,T^*),H^{k-1})$ to
\eqref{symivp.1}-\eqref{symivp.2} for some time $T^*\in (T_0,0]$ that we can take to be maximal. Letting $R>0$ be as in the assumptions from Section \ref{coeffassump}, we choose
\begin{equation*}
\delta \in (0,\Quarter\Rc), \quad \Rc = \min\bigl\{\frac{3 R}{4 C_{\text{Sob}}}, \frac{3 R}{4}\bigr\} ,
\end{equation*}
and
initial data such that 
\begin{equation*}
\norm{u(T_0)}_{H^k} < \delta.
\end{equation*} 
Then either $\norm{u(t)}_{H^{k}} < \Rc$ for all $t\in [T_0,T^*)$ or there exists a first
time $T_*\in (T_0,T^*)$ such that 
\begin{equation*}
\norm{u(T_*)}_{H^{k}} =  \Rc \leq \begin{textstyle} \frac{3}{4} \end{textstyle} R.
\end{equation*}  
Setting $T_*=T^*$ if the first case holds, then, by Sobolev's inequality, we have
that 
\begin{equation} \label{Linfty}
\max\bigl\{\norm{\nabla u(t)}_{L^\infty},\norm{u(t)}_{L^\infty}, \norm{u(t)}_{H^k} \bigr\} 
\leq \begin{textstyle} \frac{3}{4} \end{textstyle} R, \quad T_0\leq t < T_*.
\end{equation}

Applying
the operator $\Bc \nabla^\ell \Bc^{-1}$, $0\leq \ell \leq k$, to
\eqref{symivp.1} on the left yields
\begin{equation*}
B^0\del{t} \nabla^\ell u + B^i\nabla_i \nabla^\ell u = \frac{1}{t}\Bc \nabla^\ell \Pbb u - \Bc [\nabla^\ell,\Bc^{-1} B^0]\del{t}u
-\Bc[\nabla^\ell,\Bc^{-1}B^i]\nabla_i u -B^i[\nabla^\ell,\nabla_i]u + \Bc \nabla^\ell(\Bc^{-1}F),
\end{equation*}
which we observe, with the help of \eqref{symivp.1}, can be written as
\begin{align}
B^0\del{t} \nabla^\ell u + B^i\nabla_i \nabla^\ell  u = & \frac{1}{t}\Bigl[\Bc\Pbb \nabla^\ell  u  - \Bc[\nabla^\ell,\Bc^{-1}B^0](B^0)^{-1}\Bc \Pbb u \Bigr] 
+ \Bc[\nabla^\ell ,\Bc^{-1}B^0](B^0)^{-1}B^i\nabla_i u  \notag \\
&-\Bc[\nabla^\ell ,\Bc^{-1}B^i]\nabla_i u -B^i[\nabla^\ell,\nabla_i]u 
 - \Bc[\nabla^\ell ,\Bc^{-1}B^0](B^0)^{-1}F 
 + \Bc \nabla^\ell(\Bc^{-1}F). \label{symthm3.1}
\end{align}
In the following, we will use \eqref{symthm3.1} to derive energy estimates that are well behaved in the limit $t\nearrow 0$. From these energy estimates, we will deduce the global existence of solutions and decay estimates under a small initial data assumption. 

\bigskip

\noindent \underline{$L^2$ energy estimate}

\bigskip

Setting $\ell=0$ in \eqref{symthm3.1} and using the resulting symmetric hyperbolic system to derive an energy identity in the standard fashion,
we find that
\begin{equation} \label{symthm4}
\frac{1}{2} \del{t} \ip{u}{B^0 u} = \frac{1}{t}\ip{u}{\Bc \Pbb u}+\frac{1}{2} \ip{u}{\Div \! B u} + \ip{u}{F},
\end{equation}
where 
\begin{equation*}
\Div\!B = \Div \! B(t,x,u(t,x),\nabla u(t,x))
\end{equation*}
with $\Div\! B(t,x,u,w)$ as previously defined by \eqref{divBdef}, see also \eqref{divBid}.
Defining the energy norm
\begin{equation*}
\nnorm{u}^2_s =\sum_{\ell=0}^s \ip{\nabla^\ell u}{B^0 \nabla^\ell u},
\end{equation*}
we see from \eqref{B0BCbnd} and $t<0$ that the inequalities
\begin{equation*}
\frac{2}{t}\ip{v}{\Bc v} \leq \frac{2\kappa}{t}\nnorm{v}_0^2
\AND
\norm{v}_{L^2} \leq \sqrt{\gamma_1}\nnorm{v}_{0}
\end{equation*}
hold for any $v\in L^2(V)$.
With the help of these inequalities, it is not difficult to see that the estimate\footnote{The constant $C(\norm{u}_{k})$ implicitly depends on $\theta$. However, since $\theta$ will be fixed throughout
the proof, we will not indicate the dependence of any of the constants on $\theta$. The same will be true for all the other fixed constants, e.g. $\lambda_a$, $\beta_a$, $\gamma_2$, $\kappa$, and so on.} 
\begin{align*}
\del{t}\nnorm{u}^2_0\leq \frac{\bigl(2\kappa -\gamma_1\bigl[\beta_{\mathrm{o}}+2\lambda_3\big]\big)}{t}\nnorm{\Pbb u}^2_0
+&\frac{\sqrt{\gamma_1}\bigl(\beta_{\mathrm{e}}+2(\lambda_1+\lambda_2)\bigr)}{|t|^{\frac{1}{2}}}\nnorm{\Pbb u}_0\norm{u}_{L^2}
\\
&+C(\norm{u}_{H^k})\norm{u}_{L^2}\norm{u}_{H^k}+2\norm{u}_{L^2}\norm{\Ft}_{L^2},  \quad T_0\leq t < T_*,
\end{align*}
is a direct consequence of \eqref{Linfty}, \eqref{symthm4},  Proposition \ref{FdivBprop} and Sobolev's inequality, where we have set
\begin{equation*}
\beta_{\mathrm{o}} = \sum_{a=0}^3  \beta_{2a+1} \AND \beta_{\mathrm{e}} = \sum_{a=0}^3 \beta_{2a}.
\end{equation*}
The above estimate in conjunction with Young's inequality  (i.e.  $ab \leq \frac{1}{2\epsilon}a^2 + \frac{\epsilon}{2}b^2$
for $a,b\geq 0$ and $\epsilon>0$) then gives
\begin{equation} \label{symthm5}
\del{t}\nnorm{u}^2_0\leq \frac{\rho_0}{t}\nnorm{\Pbb u}^2_0
+(1+\ep^{-1})C(\nnorm{u}_{k})\nnorm{u}_{0}\nnorm{u}_k+2\sqrt{\gamma_1}\nnorm{u}_{0}\norm{\Ft}_{L^2}, \quad T_0\leq t < T_*,
\end{equation}
for any $\epsilon>0$, where
\begin{equation*} 
\rho_0 = 2\kappa -\gamma_1\bigl[\beta_{\mathrm{o}}+2\lambda_3 + \epsilon( \beta_{\mathrm{e}}+2\lambda_1+2\lambda_2) \big].
\end{equation*}
Since $2\kappa -\gamma_1\bigl[\beta_{\mathrm{o}}+2\lambda_3\big]>0$
by assumption, we choose $\epsilon$ small enough to ensure that
\begin{equation} \label{rho0pos}
\rho_0 > 0.
\end{equation}

\bigskip

\noindent \underline{$H^k$ energy estimate}

\bigskip

Before continuing, we note that the energy norm $\nnorm{u}_k$ and the standard Sobolev norm $\norm{u}_{H^k}$ are equivalent since they satisfy
\begin{equation*}
\frac{1}{\sqrt{\gamma_1}} \norm{\cdot}_{H^s} \leq \nnorm{\cdot}_s \leq \sqrt{\gamma_{2}} \norm{\cdot}_{H^s}
\end{equation*}
by \eqref{B0BCbnd}.  We will employ this equivalence below without comment.

Applying the $L^2$ energy identity, i.e.\ \eqref{symthm4}, to \eqref{symthm3.1} gives
\begin{equation} \label{symthm6}
\frac{1}{2} \del{t} \ip{\nabla^\ell u}{B^0 \nabla^\ell u} = \frac{1}{t}\ip{\nabla^\ell u}{\Bc \Pbb \nabla^\ell u}+\frac{1}{2} \ip{\nabla^\ell u}{\Div \! B \nabla^\ell u} + \ip{\nabla^\ell u}{G_\ell}, \qquad
0\leq \ell \leq k,
\end{equation}
where
\begin{align*}
G_\ell = & |t|^{-1}\Bc[\nabla^\ell,\Bc^{-1}B^0](B^0)^{-1}\Bc \Pbb u  
+ \Bc[\nabla^\ell,\Bc^{-1}B^0](B^0)^{-1}B^i\nabla_i u \notag \\
&-\Bc[\nabla^\ell,\Bc^{-1}B^i]\nabla_i u -B^i[\nabla^\ell,\nabla_i]u
 - \Bc[\nabla^\ell ,\Bc^{-1}B^0](B^0)^{-1}F 
 + \Bc \nabla^\ell(\Bc^{-1}F). 
\end{align*}
Employing \eqref{symthm6}, we obtain, with the help of \eqref{Pbbprop}, Proposition \ref{FdivBprop}, the bound \eqref{Linfty}, and Sobolev's inequality, the estimate
\begin{align}
\del{t}\nnorm{\nabla^\ell u}^2_0\leq& \frac{2\kappa-\gamma_1\beta_1}{t}\nnorm{\nabla^\ell \Pbb u}^2_0 
-\frac{\gamma_1(\beta_3+\beta_5+\beta_7)}{t}\nnorm{\Pbb u}_k\nnorm{\Pbb u}_{k-1} \notag \\
& + \frac{\gamma_1\beta_{\mathrm{e}}}{|t|^{\frac{1}{2}}}\nnorm{\Pbb u}_k \nnorm{u}_{k}
+4\theta \gamma_1 \nnorm{\nabla^\ell u}_{L^2}^2+ 2\ip{\nabla^\ell u}{G_\ell},  \qquad T_0\leq t < T_*, \label{symthm7}
\end{align}
while it is clear from Proposition \ref{Bcommprop} that $\ip{\nabla^\ell u}{G_\ell}$ is bounded by 
\begin{align}
 \ip{\nabla^\ell u}{G_\ell}\leq & -\frac{1}{t}\Bigl[\ell \mathtt{b} \norm{\Pbb u}^2_{H^k}+ C(\norm{u}_{H^k})\bigl(\norm{\Pbb u}_{H^k}\norm{\Pbb u}_{H^{k-1}}+\norm{u}_{H^k}\norm{\Pbb u}^2_{H^k}\bigr)\Bigr]
\notag \\
&+ \frac{1}{|t|^{\frac{1}{2}}}C(\norm{u}_{H^k})\norm{u}_{H^k}\norm{\Pbb u}_{H^k} + C(\norm{u}_{H^k})\Bigl[\norm{u}^2_{H^k} + \norm{u}_{H^k}\norm{\Ft}_{H^k}\Bigr], \label{symthm8}
\end{align}
where 
\begin{equation*}
\mathtt{b} =  \sup_{T_0 \leq t < 0}\Bigl( \big\|\big|\Pbb\tilde{\Bc}(t)\nabla(\tilde{\Bc}(t)^{-1}\tilde{B}^0(t))\tilde{B}^0(t)^{-1}\Pbb \tilde{B}_2(t)\Pbb\bigl|_{\op}\bigr\|_{L^\infty}
+ \big\|\big|\Pbb\tilde{\Bc}(t)\nabla(\tilde{\Bc}(t)^{-1}\tilde{B}_2(t))\Pbb\bigl|_{\op}\bigr\|_{L^\infty} \Bigr).
\end{equation*}
Combining \eqref{symthm7} and \eqref{symthm8} yields
\begin{align*}
\del{t}\nnorm{\nabla^\ell u}^2_0\leq \frac{2\kappa-\gamma_1\beta_1}{t}\nnorm{\nabla^\ell \Pbb u}^2_0 & -\frac{1}{t}\Bigl[2\gamma_1\ell \mathtt{b} \nnorm{\Pbb u}^2_{k}+ 
C(\nnorm{u}_{k})\bigl(\nnorm{\Pbb u}_{k}\nnorm{\Pbb u}_{k-1}+\nnorm{u}_{k}\nnorm{\Pbb u}^2_{k}\bigr)\Bigr]
\notag \\
&+ \frac{1}{|t|^{\frac{1}{2}}}C(\nnorm{u}_{k})\nnorm{u}_{k}\nnorm{\Pbb u}_{k} + C(\nnorm{u}_{k})\bigl(\nnorm{u}^2_{k} +\norm{\Ft}^2_{H^k}\bigr) ,  \qquad T_0\leq t < T_*. 
\end{align*}
Summing this inequality over $\ell$ from $0$ to $k$, we find, after an application of Young's inequality and Ehrling's lemma (Lemma \ref{Ehrling}), that 
\begin{align}
\del{t}\nnorm{u}^2_k &\leq \frac{2\kappa-\gamma_1(\beta_1+2\mathtt{b}_k)-C(\nnorm{u}_{k})(\ep+\norm{u}_{k})}{t}\nnorm{\Pbb u}^2_k \notag\\
&\qquad -\frac{1}{t}
c(\nnorm{u}_{k},\epsilon^{-1})\nnorm{\Pbb u}_{0}^2  + C(\nnorm{u}_{k},\ep^{-1})\bigl(\nnorm{u}^2_{k} +\norm{\Ft}^2_{H^k}\bigr)   \label{symthm9}
\end{align}
for any $\ep > 0$, where
\begin{equation*}
\mathtt{b}_k = \mathtt{b}\sum_{\ell=1}^k \ell = \frac{1}{2}k(k+1)\mathtt{b}.
\end{equation*}

\bigskip

\noindent \underline{Global existence on $[T_0,0)\times \Sigma$}

\bigskip

Since, initially,
\begin{equation*}
\nnorm{u(T_0)}_k \leq \sqrt{\gamma_2}\norm{u(T_0)}_{H^k} < \delta \sqrt{\gamma_2},
\end{equation*}
we can, for $\delta$ satisfying
\begin{equation} \label{deltafix}
0<\delta \leq \min\biggl\{\frac{\Rc}{2\sqrt{\gamma_1\gamma_2}},\frac{\Rc}{4}\biggr\},
\end{equation}
define $T_\delta \in (T_0,T_*]$ be the first time such that $\nnorm{u(T_\delta)}_k = 2\delta \sqrt{\gamma_2}$, or if such a time does not exist,
we set $T_\delta = T^*$, the maximal time of existence. In either case, we have that
\begin{equation*}
\nnorm{u(t)}_k \leq  2\delta \sqrt{\gamma_2}, \qquad T_0\leq t < T_\delta.
\end{equation*}
Note that by this definition,
\begin{equation*}
\norm{u(t)}_{H^k} \leq \sqrt{\gamma_1} \nnorm{u(t)}_k \leq 2\delta \sqrt{\gamma_1\gamma_2} \leq \Rc, \qquad T_0 \leq t < T_\delta \leq T_*\leq T^*.
\end{equation*}
Choosing $\ep$ so that
$\ep = \delta\sqrt{\gamma_2}$
allows us to write \eqref{symthm9} as
\begin{align}\label{symthm10}
\del{t}\nnorm{u}^2_k \leq \frac{\rho_k}{t}\nnorm{\Pbb u}^2_k & -\frac{1}{t}
c(\delta,\delta^{-1})\nnorm{\Pbb u}_{0}^2  + C(\delta,\delta^{-1})\bigl(\nnorm{u}^2_{k} +\norm{\Ft}^2_{H^k}\bigr), \qquad T_0 \leq t < T_\delta,
\end{align}
where
\begin{equation*}
\rho_k = 2\kappa-\gamma_1(\beta_1+2\mathtt{b}_k)-C(\delta)\delta.
\end{equation*}
Since $\lim_{\delta \searrow 0}C(\delta)\delta = 0$ and $2\kappa-\gamma_1(\beta_1+2\mathtt{b}_k) > 0$ by assumption, it follows that we can ensure that
\begin{equation}\label{rhokpos}
\rho_k > 0
\end{equation}
by choosing $\delta>0$ small enough.
Adding 
$\rho^{-1}_0c(\delta,\delta^{-1})$
times \eqref{symthm5} to \eqref{symthm10} yields the differential energy inequality
\begin{align*}
\del{t}\bigl(\nnorm{u}^2_k 
+\rho^{-1}_0c(\delta,\delta^{-1})\nnorm{u}_0^2 \bigr)\leq \frac{\rho_k}{t}\nnorm{\Pbb u}^2_k + C(\delta,\delta^{-1})\bigl(\nnorm{u}^2_{k} +\norm{\Ft}^2_{H^k}\bigr), \qquad T_0\leq t < T_\delta,
\end{align*}
from which we get
\begin{equation*}
\del{t} E_k \leq C(\delta,\delta^{-1})E_k, \qquad T_0\leq t < T_\delta,
\end{equation*}
by setting
\begin{equation} \label{Ekdef}
E_k(t) = \nnorm{u(t)}^2_k 
+\rho^{-1}_0c(\delta,\delta^{-1})\nnorm{u(t)}_0^2 
-\int_{T_0}^t \frac{\rho_k}{\tau}\nnorm{\Pbb u(\tau)}^2_k\, d\tau + \sup_{T_0 \leq \tau < 0}\norm{\Ft(\tau)}^2_{H^k}.
\end{equation}
By Gronwall's inequality, we then conclude that
\begin{equation} \label{Eest1}
E_k(t) \leq e^{C(\delta,\delta^{-1})(t-T_0)}E_k(T_0), \qquad T_0\leq t < T_\delta.
\end{equation}

With $\delta$ 
fixed so that \eqref{rho0pos}, \eqref{deltafix}, 
and \eqref{rhokpos} all hold, we consider $\delta_0 \in (0,\delta)$ and assume that
\begin{equation} \label{delta0def}
\norm{u(T_0)}_{H^k} \leq \delta_0 \AND  \sup_{T_0 \leq \tau < 0}\norm{\Ft(\tau)}^2_{H^k} \leq \delta_0.
\end{equation}
Then by \eqref{Ekdef} and \eqref{Eest1}, we deduce that
\begin{equation} \label{Eest2}
\nnorm{u(t)}_k \leq 
e^{C(\delta,\delta^{-1})(-T_0)}\bigl(2+\rho^{-1}_0c(\delta,\delta^{-1})\bigr)\delta_0,  \qquad T_0\leq t < T_\delta.
\end{equation}
By choosing $\delta_0$ small enough so that
\begin{equation} \label{deltain}
0<e^{C(\delta,\delta^{-1})(-T_0)}\bigl(2+\rho^{-1}_0c(\delta,\delta^{-1})\bigr)\delta_0 < \delta\sqrt{\gamma_2},
\end{equation}
it is clear that \eqref{Eest2} implies the bound
\begin{equation*}
\nnorm{u(t)}_k < \delta \sqrt{\gamma_2}, \quad T_0\leq t < T_\delta.
\end{equation*}
From the definition of $T_\delta$ and the maximality of $T^*$, we conclude that $T_\delta=T_*=T^*=0$. Thus, we have established the global existence of solutions on $[T_0,0)\times \Sigma$ that are 
uniformly bounded by
\begin{equation} \label{deltahat}
\norm{u(t)}_{H^k} \leq \sqrt{\gamma_1}\nnorm{u(t)}_k \leq \hat{\delta} := 2\delta \sqrt{\gamma_1 \gamma_2}, \quad T_0\leq t < 0.
\end{equation}
Moreover, from \eqref{Ekdef}, \eqref{Eest1} and the equivalence of norms $\norm{\cdot}_{H^k}$ and $\nnorm{\cdot}_k$, we see immediately that energy estimate
\begin{equation} \label{energy}
 \norm{u(t)}^2_{H^k} +  \sup_{T_0 \leq \tau < 0}\norm{\Ft(\tau)}^2_{H^k} -\int_{T_0}^t \frac{1}{\tau}\norm{\Pbb u(\tau)}^2_{H^k}\, d\tau \leq  C(\delta)e^{C(\delta,\delta^{-1})(t-T_0)}\Bigl(\norm{u(T_0)}^2_{H^k} +  
 \sup_{T_0 \leq \tau < 0}\norm{\Ft(\tau)}^2_{H^k} \Bigr)
\end{equation}
holds for $T_0\leq t < 0$.

\bigskip

\noindent \underline{Limit of $\Pbb^\perp u$ as $t\nearrow 0$}

\bigskip

Writing \eqref{symivp.1} as
\begin{equation*}
\del{t}u =  (B^0)^{-1}\biggl(-B^i\nabla_i u + \frac{1}{t} \Bc \Pbb u + F\biggr),
\end{equation*}
we see after multiplying of the left by $\Pbb^\perp$ that 
\begin{equation}  \label{PbbperplimA}
\del{t}\Pbb^\perp u  = \Pbb^\perp (B^0)^{-1}\biggl(-B^i\nabla_i u + \frac{1}{t} \Bc \Pbb u + F\biggr).
\end{equation}
Integrating in time and applying the $H^{k-1}$ norm yields
\begin{equation}\label{PbbperplimBb}
\norm{\Pbb^\perp u(t_2) - \Pbb^\perp u(t_1)}_{H^{k-1}} \leq \int_{t_1}^{t_2} \Bigl\|  \Pbb^\perp (B^0(\tau))^{-1}\biggl(-B^i(\tau)\nabla_i u(\tau) + \frac{1}{\tau} \Bc(\tau)
 \Pbb u(\tau) + F(\tau)\biggr)\Bigr\|_{H^{k-1}}\, d\tau
\end{equation}
for any $t_1,t_2$ satisfying $T_0 \leq t_1 < t_2$. Using Proposition \ref{Pperpprop} to estimate the integrand of the right hand side of the above inequality, we have
\begin{align} 
\norm{\Pbb^\perp u(t_2) - \Pbb^\perp u(t_1)}_{H^{k-1}} \leq 
&\int_{t_1}^{t_2} C(\norm{u(\tau)}_{H^k})\biggl(\norm{\Ft(\tau)}_{H^k}+1
\notag \\
&\qquad +
\frac{1}{|\tau|^{\frac{1}{2}}}\norm{\Pbb u(\tau)}_{H^k}-\frac{1}{\tau}\norm{\Pbb u(\tau)}_{H^k}\norm{\Pbb u(\tau)}_{H^{k-1}}\biggr)
\, d\tau  \label{PbbperplimBa}
\end{align}
and hence, with the help of the bounds \eqref{delta0def}, \eqref{deltain} and \eqref{energy},  that\footnote{Note that we chose $\delta_0$ small enough to make the right hand side of \eqref{Eest2} comparable to $\delta$, which,
in turn, makes the right hand side of \eqref{energy} comparable to
$\delta$.}
\begin{equation} \label{PbbperplimB}
\norm{\Pbb^\perp u(t_2) - \Pbb^\perp u(t_1)}_{H^{k-1}} \leq C(\delta)\biggl(|t_1-t_2|-\int_{t_1}^{t_2} \frac{1}{\tau}\norm{\Pbb u(\tau)}_{H^k}^2 \, d\tau \biggr) = \textrm{o}(|t_1-t_2|).
\end{equation}
From this inequality, we conclude that the limit $\lim_{t\nearrow 0}\Pbb^\perp u(t)$ exists in $H^{k-1}(V)$ and that 
\begin{equation*}
\Pbb^\perp u \in C^0([T_0,0],H^{k-1}(V)).
\end{equation*}

\bigskip

\noindent \underline{$L^2$ decay estimate for $\Pbb u$}

\bigskip

Multiplying \eqref{symivp.1} on the left by $\Pbb$, we observe, with the help of \eqref{Pbbprop}, that $\Pbb u$ satisfies
\begin{equation} \label{Pbbueqn}
\Pbb B^0 \Pbb \del{t} \Pbb u + \Pbb B^i \Pbb \nabla_i \Pbb u = \frac{1}{t} \Bc \Pbb u + \Fc,
\end{equation}
where
\begin{equation} \label{Fcdef}
\Fc = -\Pbb B^0\Pbb^\perp  \del{t} \Pbb^\perp u -\Pbb B^i \Pbb^\perp \nabla_i u + \Pbb F. 
\end{equation}
Taking $L^2$ inner-product of \eqref{Pbbueqn} with $\Pbb u$ yields the energy identity
\begin{equation} \label{Pbbueest1}
\frac{1}{2}\del{t}\ip{\Pbb u}{B^0\Pbb u} = \frac{1}{t}\ip{\Pbb u}{\Bc \Pbb u} +\frac{1}{2} \ip{u}{\Pbb \Div\! B\, \Pbb u} + \ip{\Pbb u}{\Pbb \Fc}.
\end{equation}

Next, we observe that
\begin{equation} \label{Pbbueest2}
\frac{1}{t}\ip{\Pbb u}{\Bc \Pbb u} \leq \frac{\kappat}{t} \nnorm{\Pbb u}_0^2
\end{equation}
holds by \eqref{B0BCbndt}, 
and also that the estimate
\begin{equation} \label{FcestA}
|\ip{\Pbb u}{\Pbb B^i\Pbb^\perp \nabla_i u}|\leq C(\delta)\norm{u}_{H^k}\norm{\Pbb u}_{L^2}
+|t|^{-\frac{1}{2}}\alpha \norm{\nabla \Pbb^\perp u}_{L^2}\norm{\Pbb u}_{L^2} + |t|^{-1}C(\delta)\norm{\nabla \Pbb^\perp u}_{L^2}\norm{\Pbb u}_{L^2}^2
\end{equation}
is a straightforward consequence of the calculus inequalities, the expansion \eqref{Bexp} for $B=(B^i)$, and the properties \eqref{BI1bnd.2} and \eqref{BI2bnd.1} of the expansions coefficients 
$\Pbb B_a \Pbb^\perp$, $a=1,2$. Moreover,  by first applying the $H^{k-1}$ norm to \eqref{PbbperplimA} and then integrating in time, the same arguments that led to  \eqref{PbbperplimB}
show that
\begin{equation} \label{PbbperpestA}
\int_{T_0}^t \norm{\del{t}\Pbb^\perp u(\tau)}_{H^{k-1}} \, d\tau \leq C(\delta), \qquad T_0 \leq t < 0,
\end{equation} 
while it is clear from \eqref{deltahat}, \eqref{PbbperplimA} and Proposition \ref{Pperpprop} that
\begin{equation} \label{PbbperpestB}
\norm{|t|^{\frac{1}{2}}\del{t}\Pbb^\perp u(t)}_{L^2} \leq C(\delta)\biggl(1+\frac{1}{|t|^{\frac{1}{2}}}\norm{\Pbb u}_{L^2} \biggr), \qquad T_0 \leq t < 0.
\end{equation} 

From the energy identity \eqref{Pbbueest1}, the estimates \eqref{Pbbueest2}, \eqref{FcestA} and \eqref{PbbperpestB}, the coefficient bounds \eqref{B0bnd.2},  \eqref{F0bnd}, \eqref{F1bnd.1}, 
\eqref{divBbnd.1}, and \eqref{B0BCbndt}, the expansion \eqref{fexp} for $F$ along with \eqref{F2vanish}, and Proposition \ref{FdivBprop}, we obtain, with the help of the bound 
\eqref{deltahat} and the calculus inequalities, the differential energy inequality
\begin{align*} 
\frac{1}{2}\del{t}\nnorm{\Pbb u}^2_0 \leq \biggl[\frac{1}{t}\biggl(\kappat- \gammat_1\biggl(\frac{\beta_1}{2} +\norm{\nabla \Pbb^\perp u}_{L^2}&C(\delta)\biggr) \biggr)+C(\delta)
\biggl(1+\frac{1}{|t|^{\frac{1}{2}}}+ \norm{\del{t}\Pbb^\perp u}_{H^{k-1}}\biggr)\biggr]\nnorm{\Pbb u}^2_0 \\
&+ \biggl[C(\delta) + \frac{\bigl(\norm{u}_{L^2}\lambda_1+\norm{\nabla \Pbb^\perp u}_{L^2}\alpha\bigr)\sqrt{\gammat_1}}{|t|^{\frac{1}{2}}} \biggr] \nnorm{\Pbb u}_0
\end{align*}
that, after dividing through by $\nnorm{\Pbb u}_0$, gives
\begin{equation}  \label{Pbbueest3}
\del{t}\nnorm{\Pbb u}_0 \leq \biggl[\frac{\hat{\rho}}{t}+C(\delta)
\biggl(1+\frac{1}{|t|^{\frac{1}{2}}}+ \norm{\del{t}\Pbb^\perp u}_{H^{k-1}}\biggr)\biggr]\nnorm{\Pbb u}_0 
+ C(\delta) \biggl[1 + \frac{(\lambda_1+\alpha)}{|t|^{\frac{1}{2}}} \biggr]
\end{equation}
where
\begin{equation*} 
\hat{\rho} = \kappat- \gammat_1\biggl(\frac{\beta_1}{2} +C(\delta)\delta\biggr).
\end{equation*}
Since $\kappat- \frac{\gammat_1 \beta_1}{2} >0$ by assumption, we see, by shrinking $\delta$ if necessary, that we can arrange that
\begin{equation*}
\hat{\rho} >0.
\end{equation*}
The $L^2$ decay estimate
\begin{equation}\label{L2decayA}
\norm{\Pbb u(t)}_{L^2} \lesssim \begin{cases}
|t|+(\lambda_1+\alpha)|t|^{\frac{1}{2}} & \text{if $\hat{\rho} > 1$} \\
-|t| \ln\bigl(\frac{t}{T_0}\bigr)+(\lambda_1+\alpha)|t|^{\frac{1}{2}} & \text{if $\hat{\rho} = 1$ } \\
|t|^{\hat{\rho}}+(\lambda_1+\alpha)|t|^{\frac{1}{2}} & \text{if $\frac{1}{2} < \hat{\rho} < 1$} \\
|t|^{\frac{1}{2}}-(\lambda_1+\alpha)|t|^{\frac{1}{2}} \ln\bigl(\frac{t}{T_0}\bigr) & \text{if $\hat{\rho}= \frac{1}{2}$}\\
|t|^{\hat{\rho}}   & \text{if $0 < \hat{\rho} < \frac{1}{2}$}\end{cases}, \qquad T_0 \leq t < 0,
\end{equation}
then follows directly from \eqref{PbbperpestA}, \eqref{Pbbueest3}, Gronwall's inequality\footnote{Here, we are using the following form of Gronwall's inequality: if $x(t)$ satisfies
$x'(t) \leq a(t)x(t) +h(t)$, $t\geq T_0$, then $x(t)\leq x(T_0)e^{A(t)}+ \int_{T_0}^t e^{A(t)-A(\tau)}h(\tau) \, d\tau$ where
$A(t)= \int_{T_0}^t a(\tau)\, d\tau$. In particular, we observe from this that if, $x(T_0)\geq 0$ and  $a(t)=\frac{\lambda}{t}+b(t)$, where  $\lambda \in \Rbb$ and $|\int_{T_0}^t b(\tau) \, d\tau |\leq r$, then
\eqn{gronwall}{
x(t) \leq e^{r}x(T_0)\left(\frac{t}{T_0}\right)^\lambda +   e^{2r}(-t)^\lambda \int_{T_0}^t \frac{|h(\tau)|}{(-\tau)^\lambda}\, d\tau
}
for $T_0 \leq t < 0$.
}
and the integral formula
\begin{equation} \label{intform}
(-t)^\lambda \int_{T_0}^t \frac{1}{(-\tau)^{\lambda+\mu}}\, d\tau = \begin{cases} \frac{1}{\lambda+\mu -1}|t|^{1-\mu} + \frac{|T_0|^{1-(\lambda+\mu)}}{1-(\lambda+\mu)}|t|^\lambda &
\text{if $\lambda+\mu \neq 1$} \\
-|t|^\lambda\ln\bigl(\frac{t}{T_0} \bigr) & \text{if $\lambda+\mu =1$}\end{cases}.
\end{equation}

\bigskip

\noindent \underline{$H^{k-1}$ decay estimate for $\Pbb u$}

\bigskip

Applying $\Bc \Pbb \nabla^\ell \Bc^{-1} \Pbb $, $0\leq \ell \leq k-1$, to \eqref{Pbbueqn} on the left yields
\begin{align*}
\Pbb B^0\Pbb\del{t} \nabla^\ell\Pbb u + &\Pbb B^i \Pbb \nabla_i \nabla^\ell \Pbb u = \frac{1}{t}\Bc \nabla^\ell \Pbb u - \Bc\Pbb [\nabla^\ell,\Bc^{-1}\Pbb B^0 \Pbb ]\Pbb\del{t}u \notag \\
&-\Bc\Pbb [\nabla^\ell,\Bc^{-1}\Pbb B^i \Pbb]\nabla_i \Pbb u -\Pbb B^i \Pbb [\nabla^\ell,\nabla_i]\Pbb u + \Bc \Pbb \nabla^\ell(\Bc^{-1}\Pbb \Fc).
\end{align*}
Using \eqref{symivp.1}, we can write this
\begin{align}
\Pbb B^0\Pbb \del{t} \nabla^\ell\Pbb  u + \Pbb B^i\Pbb \nabla_i \nabla^\ell \Pbb  u &=  \frac{1}{t}\Bigl[\Bc\Pbb \nabla^\ell \Pbb  u  - 
\Bc\Pbb [\nabla^\ell,\Bc^{-1}\Pbb B^0\Pbb ]\Pbb (B^0)^{-1}\Bc \Pbb u \Bigr] 
\notag \\
&+ \Bc\Pbb [\nabla^\ell ,\Bc^{-1}\Pbb B^0\Pbb ]\Pbb (B^0)^{-1}B^i\nabla_i u  
-\Bc\Pbb [\nabla^\ell ,\Bc^{-1}\Pbb B^i\Pbb ]\nabla_i \Pbb u \notag \\
& -\Pbb B^i\Pbb [\nabla^\ell,\nabla_i]\Pbb u 
- \Bc\Pbb [\nabla^\ell ,\Bc^{-1}\Pbb B^0\Pbb ]\Pbb (B^0)^{-1}F 
 + \Bc\Pbb  \nabla^\ell(\Bc^{-1}\Fc). \label{Hkdec1}
\end{align}
Taking the $L^2$ inner-product of \eqref{Hkdec1} with $\nabla^\ell \Pbb u$ yields the energy identity
\begin{equation} \label{Hkdec2}
\frac{1}{2}\del{t}\ip{\nabla^\ell \Pbb u}{B^0 \nabla^\ell \Pbb u} = \frac{1}{t}\ip{\nabla^\ell \Pbb u}{\Bc \nabla^\ell \Pbb u} +\frac{1}{2} \ip{\nabla^\ell \Pbb u}{\Pbb \Div\! B\, \Pbb
\nabla^\ell \Pbb u} + \ip{\nabla^\ell \Pbb u}{ \Gc_\ell} + \ip{\nabla^\ell \Pbb u}{\Bc\Pbb  \nabla^\ell(\Bc^{-1}\Fc)}
\end{equation}
where
\begin{align*}
\Gc_\ell =&    |t|^{-1}\Bc\Pbb [\nabla^\ell,\Bc^{-1}\Pbb B^0\Pbb ]\Pbb (B^0)^{-1}\Bc \Pbb u 
+ \Bc\Pbb [\nabla^\ell ,\Bc^{-1}\Pbb B^0\Pbb ]\Pbb (B^0)^{-1}B^i\nabla_i u  
\\
&-\Bc\Pbb [\nabla^\ell ,\Bc^{-1}\Pbb B^i\Pbb ]\nabla_i \Pbb u 
 -\Pbb B^i\Pbb [\nabla^\ell,\nabla_i]\Pbb u 
- \Bc\Pbb [\nabla^\ell ,\Bc^{-1}\Pbb B^0\Pbb ]\Pbb (B^0)^{-1}F.
\end{align*}

Since
\begin{equation*}
\Pbb \nabla^\ell \Pbb u = \nabla^\ell \Pbb u,
\end{equation*}
it follows from \eqref{B0BCbndt} that
\begin{equation} \label{Hkdec3}
\frac{1}{t}\ip{ \nabla^\ell \Pbb u}{\Bc  \nabla^\ell\Pbb u} \leq \frac{\kappat}{t} \nnorm{\nabla^\ell \Pbb u}_0^2.
\end{equation}
Moreover, using \eqref{B0bnd.2}, \eqref{fexp}-
\eqref{F1bnd.1}, 
\eqref{Bexp}, \eqref{BI1bnd.2} and \eqref{BI2bnd.1}, it is not difficult from the definition of
$\Fc$, see \eqref{Fcdef}, to verify, with the help of the calculus
inequalities, that the last term in \eqref{Hkdec2} can be estimated as
\begin{align*}
|\ip{\nabla^\ell \Pbb u}{\Bc\Pbb  \nabla^\ell(\Bc^{-1}\Fc)}| \leq& C(\norm{u}_{H^{k}})\norm{\nabla^\ell \Pbb u}_{L^2}\Bigl[1+\norm{\Ft}_{H^{k-1}}+ \norm{t^{\frac{1}{2}}\del{t}\Pbb^\perp u}_{H^{k-1}}
+\norm{\Pbb u}_{H^{k-1}}\norm{\del{t}\Pbb^{\perp}u}_{H^{k-1}} \notag  \\
&+ |t|^{-\frac{1}{2}}\bigl(\alpha \norm{\Pbb^\perp u}_{H^k} + \lambda_1 \norm{u}_{H^{k-1}}\bigr) +|t|^{-1}C\norm{\Pbb^\perp u}_{H^k}\norm{\Pbb u}_{H^{k-1}}
\Bigr]. 
\end{align*}
Recalling \eqref{PbbperplimA}, the above estimate in conjunction with Proposition \ref{Pperpprop} and the bounds \eqref{delta0def} and \eqref{deltain}-\eqref{deltahat} implies that
\begin{equation} \label{Hkdec5}
|\ip{\nabla^\ell \Pbb u}{\Bc\Pbb  \nabla^\ell(\Bc^{-1}\Fc)}| \leq C(\delta)\norm{\nabla^\ell \Pbb u}_{L^2}\biggl[ \biggl( -\frac{\delta}{t} + \frac{1}{|t|^{\frac{1}{2}}} +
\norm{\del{t}\Pbb^{\perp}u}_{H^{k-1}}
\biggr)\norm{\Pbb u}_{H^{k-1}}+ 1 + \frac{\lambda_1+\alpha}{|t|^{\frac{1}{2}}}
\biggr].
\end{equation}
We further observe that the estimate 
\begin{align*}
|\ip{\nabla^\ell \Pbb u}{ \Gc_\ell}| \leq& \norm{\nabla^\ell \Pbb u}_{L^2}\biggl[ -\frac{\ell\tilde{\mathtt{b}}}{t}\norm{\Pbb u}_{H^{k-1}} + C(\norm{u}_{H^k})\biggl( -\frac{1}{t} \Bigl(\norm{\Pbb u}_{H^{k-2}}
+\norm{u}_{H^{k-1}}\norm{\Pbb u}_{H^{k-1}} \Bigr) 
\notag \\
&+\norm{\Ft}_{H^{k-1}}+\norm{u}_{H^{k-1}} + \frac{1}{|t|^{\frac{1}{2}}}\Bigl((\lambda_1+\alpha)\norm{u}_{H^{k-1}}+\norm{\Pbb u}_{H^{k-1}}\Bigr) \biggr) \biggr], 
\end{align*}
where
\begin{align*}
\tilde{\mathtt{b}} =  \sup_{T_0 \leq t < 0}\Bigl( \big\|\big|\Pbb\tilde{\Bc}(t)\nabla(\tilde{\Bc}(t)^{-1}\Pbb \tilde{B}^0(t) \Pbb)\Pbb \tilde{B}^0(t)^{-1} \tilde{B}_2(t)\Pbb\bigl|_{\op}\bigr\|_{L^\infty}
+ \big\|\big|\Pbb\tilde{\Bc}(t)\nabla(\tilde{\Bc}(t)^{-1}\tilde{B}_2(t))\Pbb\bigl|_{\op}\bigr\|_{L^\infty} \Bigr),
\end{align*}
for the second to last term in \eqref{Hkdec2} is a direct consequence of Proposition \ref{BcommpropA}, which, given the bounds \eqref{delta0def}, \eqref{deltain} and \eqref{deltahat}, implies via
an application of Ehrling's lemma (Lemma \ref{Ehrling}) that
\begin{align}
|\ip{\nabla^\ell \Pbb u}{ \Gc_\ell}| \leq &\norm{\nabla^\ell \Pbb u}_{L^2}\biggl[ \biggl(-\frac{\ell \tilde{\mathtt{b}} +
C(\delta)(\mu+\delta)}{t}+ \frac{C(\delta)}{|t|^{\frac{1}{2}}}\biggr)\norm{\Pbb u}_{H^{k-1}} \notag \\
&\hspace{2.5cm}+ C(\delta)\biggl(1+\frac{\lambda_1+\alpha}{|t|^{\frac{1}{2}}}\biggr) - \frac{C(\delta,\mu^{-1})}{t}\norm{\Pbb u}_{L^2}\biggr] \label{Hkdec7}
\end{align}
for any $\mu >0$. Choosing $\mu=\delta$, we deduce, from \eqref{divBbnd.1}, \eqref{Hkdec2}, and the inequalities  \eqref{B0BCbndt}, \eqref{Hkdec3}, \eqref{Hkdec5} and \eqref{Hkdec7} 
that
\begin{align*}
\frac{1}{2}\del{t}&\nnorm{\nabla^\ell \Pbb u}_0^2 \leq \biggl[\biggl( \frac{1}{t}\biggl(\kappat -\frac{\gammat_1\beta_1}{2}\biggr)+ \frac{\gammat_1\beta_0}{2|t|^{\frac{1}{2}}} 
\biggr)\nnorm{\nabla^\ell \Pbb u}_0
+\biggl(-\frac{\gammat_1 \ell \tilde{\mathtt{b}}+C(\delta)\delta}{t}  \\
&+  C(\delta)\biggl(1+\frac{1}{|t|^{\frac{1}{2}}} + \norm{\del{t}\Pbb^{\perp}u}_{H^{k-1}}\biggr)\biggl)\nnorm{\Pbb u}_{k-1}
+ C(\delta)\biggl(1+\frac{\lambda_1+\alpha}{|t|^{\frac{1}{2}}}\biggr)-  \frac{C(\delta,\delta^{-1})}{t}\norm{\Pbb u}_{L^2} \biggr] \nnorm{\nabla^\ell \Pbb u}_0.
\end{align*}
Dividing the above expression by $\nnorm{\nabla^\ell \Pbb u}_0$ and summing over $\ell$ from 1 to $k-1$, we see that $\nnorm{\Pbb u}_{k-1}$
satisfies  
\begin{equation} \label{Hkdec8}
\del{t}\nnorm{\Pbb u}_{k-1} \leq  \biggl[\frac{\tilde{\rho}}{t}+ C(\delta)\biggl(1+\frac{1}{|t|^{\frac{1}{2}}} + \norm{\del{t}\Pbb^{\perp}u}_{H^{k-1}}\biggr)\biggr] \nnorm{\Pbb u}_{k-1}
+ C(\delta)\biggl(1+\frac{\lambda_1+\alpha}{|t|^{\frac{1}{2}}}\biggr)-  \frac{C(\delta,\delta^{-1})}{t}\norm{\Pbb u}_{L^2},
\end{equation}
where
\begin{equation*} 
\tilde{\rho} = \kappat -\gammat_1\biggl(\frac{\beta_1}{2}+\tilde{\mathtt{b}}_k\biggr) - C(\delta)\delta
\end{equation*}
and
\begin{equation*}
\tilde{\mathtt{b}}_k =\tilde{\mathtt{b}}\sum_{\ell = 1}^{k-1}\ell = \frac{(k-1)k}{2}\tilde{\mathtt{b}}.
\end{equation*}

Since $\kappat -\gammat_1\bigl(\frac{\beta_1}{2}+\tilde{\mathtt{b}}_k\bigr) >0$ by assumption, we can arrange, by choosing $\delta>0$ small enough, that
\begin{equation*}
\rhot  > 0.
\end{equation*}
We then obtain from  \eqref{PbbperpestA}, \eqref{L2decayA}, \eqref{intform},  \eqref{Hkdec8}  and Gronwall's inequality the decay estimate
\begin{equation*} 
\norm{\Pbb u(t)}_{H^{k-1}} \lesssim \begin{cases}
|t|+(\lambda_1+\alpha)|t|^{\frac{1}{2}} & \text{if $\tilde{\rho} > 1$} \\
-|t| \ln\bigl(\frac{t}{T_0}\bigr)+(\lambda_1+\alpha)|t|^{\frac{1}{2}} & \text{if $\tilde{\rho} = 1$ } \\
|t|^{\tilde{\rho}}+(\lambda_1+\alpha)|t|^{\frac{1}{2}} & \text{if $\frac{1}{2} < \tilde{\rho} < 1$} \\
|t|^{\frac{1}{2}}-(\lambda_1+\alpha)|t|^{\frac{1}{2}} \ln\bigl(\frac{t}{T_0}\bigr) & \text{if $\tilde{\rho}= \frac{1}{2}$}\\
|t|^{\tilde{\rho}}   & \text{if $0 < \tilde{\rho} < \frac{1}{2}$}\end{cases}, \qquad T_0 \leq t < 0.
\end{equation*}
Given any $\sigma > 0$, it follows from this inequality 
that we can choose $\delta>0$ small enough so that 
\begin{equation} \label{HkdecayB}
\norm{\Pbb u(t)}_{H^{k-1}} \lesssim \begin{cases}
|t|+(\lambda_1+\alpha)|t|^{\frac{1}{2}} & \text{if $\zeta > 1$} \\
|t|^{\zeta-\sigma}+(\lambda_1+\alpha)|t|^{\frac{1}{2}} & \text{if $\frac{1}{2} < \zeta \leq 1$} \\
|t|^{\zeta-\sigma}   & \text{if $0 < \zeta \leq \frac{1}{2}$}\end{cases}, \qquad T_0 \leq t < 0,
\end{equation}
where
\begin{equation*}
\zeta = \kappat -\gammat_1\biggl(\frac{\beta_1}{2}+\tilde{\mathtt{b}}_k\biggr).
\end{equation*}

\bigskip

\noindent \underline{$H^{k-1}$ decay estimate for $\Pbb^\perp u-\Pbb^\perp u|_{t=0}$}

\bigskip

To complete the proof, we note, via the Cauchy Schwartz inequality, that
\begin{equation*} \label{decayCS}
\int_{t_1}^{t_2} -\frac{1}{\tau}\norm{\Pbb u(\tau)}_{H^{k-1}}
\norm{\Pbb u(\tau)}_{H^{k}} \, d \tau \leq \biggl(\int_{t_1}^{t_2} -\frac{1}{\tau}\norm{\Pbb u(\tau)}_{H^{k-1}}^2\, d\tau\biggr)^{\frac{1}{2}} 
\biggl(\int_{t_1}^{t_2}- \frac{1}{\tau}\norm{\Pbb u(\tau)}_{H^{k}}^2\, d\tau 
\biggr)^{\frac{1}{2}}
\end{equation*}
for all $T_0\leq t_1 < t_2 < 0$. Using this in conjunction with the inequalities
\eqref{energy}, \eqref{PbbperplimBa} and \eqref{HkdecayB} yields
the estimate
\begin{equation*}
\norm{\Pbb^\perp u(t_2) - \Pbb^\perp u(t_1)}_{H^{k-1}} \lesssim
\begin{cases}  
\Bigl(-((-t_2)^{2(\zeta-\sigma)}-t_2\bigr)+
\bigl((-t_1)^{2(\zeta-\sigma)}-t_1)\bigr)\Bigr)^{\frac{1}{2}} 
& \text{if $\zeta > \frac{1}{2}$ } \\
   \Bigl(-(-t_2)^{2(\zeta-\sigma)}+(-t_1)^{2(\zeta-\sigma)}\Bigr)^{\frac{1}{2}}  &  \text{if $0<\zeta \leq \frac{1}{2}$ }
 \end{cases}. 
\end{equation*}
Sending $t_2\nearrow 0$ then shows that $\Pbb^\perp u$ satisfies the decay
estimate
\begin{equation*}
\norm{\Pbb^\perp u(t) - \Pbb^\perp u(0)}_{H^{k-1}} \lesssim
 \begin{cases}  |t|^{\frac{1}{2}}+ |t|^{\zeta-\sigma}
& \text{if $\zeta > \frac{1}{2}$} \\
 |t|^{\zeta-\sigma}  &  \text{if $\zeta \leq \frac{1}{2} $ }
 \end{cases}, \qquad T_0 \leq t < 0.
\end{equation*}
\end{proof}

\begin{rem} \label{divBassumptions}
Applying the covariant derivative $\nabla_i$ to \eqref{symivp.1} will, by \eqref{symthm3.1} with $\ell=1$, yield an equation for $\nabla_i u$ that is of the same form as \eqref{symivp.1}.
This equation together with \eqref{symivp.1} defines a system of evolution equations for $(u,\nabla_i u)$. It is then not difficult to check that the all the assumptions needed to apply
Theorem \eqref{symthm} to the evolution system for $(u,\nabla_i u)$ will hold provided the assumptions (i)-(iv) from Section \ref{coeffassump} are satisfied 
and the following generalization of assumption (v) holds:  
\begin{align}
\Pbb(\pi(v)) \Div \! B(t,v,w) \Pbb(\pi(v)) &= \Ordc\bigl(\theta+  |t|^{-\frac{1}{2}}\beta_0 + |t|^{-1}\beta_1\bigl), \label{divBbndg.1}\\
\Pbb(\pi(v))  \Div\! B(t,v,w) \Pbb^\perp(\pi(v)) &= 
\Ordc\biggl(\theta+|t|^{-\frac{1}{2}}\beta_2
+ \frac{|t|^{-1}\beta_3}{R}\bigl(\Pbb(\pi(v))v+  \Pbb(\pi(v)) w \bigr)\biggr), \label{divBbndg.2}\\
\Pbb^\perp(\pi(v)) \Div\! B(t,v,w) \Pbb(\pi(v))&= 
\Ordc\biggl(\theta+|t|^{-\frac{1}{2}}\beta_4
+ \frac{|t|^{-1}\beta_5}{R}\bigl(\Pbb(\pi(v))v+  \Pbb(\pi(v)) w \bigr) \biggr) \label{divBbndg.3}
\intertext{and}
\Pbb^\perp(\pi(v)) \Div\! B(t,v,w) \Pbb^\perp(\pi(v))& =\Ordc\biggl(\theta+
\frac{|t|^{-\frac{1}{2}}\beta_6}{R}\bigl(\Pbb(\pi(v))v+  \Pbb(\pi(v)) w \bigr)
\notag \\
& \hspace{1.0cm}+ \frac{|t|^{-1}\beta_7}{R^2}\bigl(\Pbb(\pi(v)) v\otimes\Pbb(\pi(v)) v + \Pbb(\pi(v)) w\otimes\Pbb(\pi(v)) w\bigr) \biggr). \label{divBbndg.4}
\end{align}
In this way, it is possible to establish the global existence and decay of solutions to \eqref{symthm3.1} from Theorem \ref{symthm} under the less restrictive assumptions \eqref{divBbndg.1}-\eqref{divBbndg.4}
for the term $ \Div \! B(t,v,w)$. 
\end{rem}

\begin{rem} \label{Gdecay}$\;$

\begin{enumerate}[(a)]
\item As discussed above in Remark \ref{Gestrem}, improvements to the decay estimate
for $\Pbb^\perp u$ can be made by imposing more restrictive assumptions on
the coefficients of \eqref{symivp.1}. For example:
\begin{enumerate}[(i)]
\item
If we assume that
\begin{gather} \label{improvedecayA}
\Pbb^\perp B_1=\Pbb^\perp B_2=0, \quad
\Pbb^\perp F_1=\Pbb^\perp F_2=0, \AND
\Pbb^\perp (B^0)^{-1}\Pbb =0, 
\end{gather}
then it is not difficult to verify from the proof of Proposition \ref{Pperpprop}
that $G$, as defined there, satisfies the estimate
\begin{equation*}
\norm{G}_{H^{k-1}}\leq C(\norm{u}_{H^k})\bigl(\norm{\Ft}_{H^{k-1}}+1\bigr).
\end{equation*}
Taken together, this estimate and the inequalities
\eqref{energy} and \eqref{PbbperplimBb} imply that
$\Pbb^\perp u$ satisfies
\begin{equation*}
\norm{\Pbb^\perp u(t_2) - \Pbb^\perp u(t_1)}_{H^{k-1}} \lesssim -t_1+t_2
\end{equation*}
for $T_0\leq t_1 < t_2<0$,
which, after sending $t_2\nearrow 0$, yields the improved decay estimate
\begin{equation}\label{improvedecayB}
\norm{\Pbb^\perp u(t)-\Pbb^\perp u(0)}_{H^{k-1}} \lesssim |t|,
\quad T_0\leq t < 0.
\end{equation}
\item On the other hand, if we make the less restrictive assumption
\begin{equation*} 
\Pbb^\perp B_1=\Pbb^\perp B_2=0     
\end{equation*}
instead of \eqref{improvedecayA}, then it can be seen from the proof of Proposition \ref{Pperpprop}
that $G$ will satisfy the estimate
\begin{equation*}
\norm{G}_{H^{k-1}}\leq C(\norm{u}_{H^k})\biggl(\norm{\Ft}_{H^{k-1}}+1
-\frac{1}{t}\norm{\Pbb u}_{H^{k-1}}^2\biggr).
\end{equation*}
Using this estimate in conjunction with \eqref{energy}, \eqref{PbbperplimBb}
and \eqref{HkdecayB},
it is not difficult to verify that $\Pbb^\perp u$ satisfies
\begin{equation*}
\norm{\Pbb^\perp u(t_2) - \Pbb^\perp u(t_1)}_{H^{k-1}} \lesssim
\begin{cases}  |t_2-t_1|
& \text{if $\zeta > 1$ } \\
  |t_2-t_1| +  (-t_1)^{2(\zeta-\sigma)}-(-t_2)^{2(\zeta-\sigma)}  &  \text{if $\zeta \leq 1$ }
 \end{cases}
\end{equation*}
for $T_0\leq t_1 < t_2 <0$.
Sending $t_2\nearrow 0$ then yields the decay estimate
\begin{equation*}
\norm{\Pbb^\perp u(t)-\Pbb^\perp u(0)}_{H^{k-1}} \lesssim 
\begin{cases}  |t|
& \text{if $\zeta > 1$ } \\
  |t| +  |t|^{2(\zeta-\sigma)}  &  \text{if $\zeta \leq 1$ }
 \end{cases} , 
\quad T_0\leq t < 0,
\end{equation*}
which is an improvement over the one from Theorem \ref{symthm}, but
not as good as \eqref{improvedecayB} for $\zeta\leq 1$.
\end{enumerate}
\item It is also possible to achieve better decay rates for $\Pbb^\perp u$ under less restrictive assumptions if the norm used to measure the decay is of lower differentiability. For example, under no additional assumptions beyond those of Theorem \ref{symthm}, we have
by Proposition \ref{Pperpprop} that
\begin{equation*}
\norm{G}_{H^{k-2}}\leq C(\norm{u}_{H^{k-1}})\biggl(\norm{\Ft}_{H^{k-2}}+1
+\frac{1}{|t|^{\frac{1}{2}}}\norm{\Pbb u}_{H^{k-1}}-\frac{1}{t}
\norm{\Pbb u}_{H^{k-2}}\norm{\Pbb u}_{H^{k-1}}\biggr)
\end{equation*}
from which we obtain
\begin{equation*}
\norm{G}_{H^{k-2}}\leq C(\norm{u}_{H^k})\biggl(\norm{\Ft}_{H^{k-2}}+1
-\frac{1}{t}
\norm{\Pbb u}_{H^{k-1}}^2\biggr)
\end{equation*}
via an application of Young's inequality. Using this inequality
together with the version of the inequality \eqref{PbbperplimBb} where $k$ has been replaced
by $k-1$ gives
\begin{align*} 
\norm{\Pbb^\perp u(t_2) - \Pbb^\perp u(t_1)}_{H^{k-1}} \leq 
&\int_{t_1}^{t_2} C(\norm{u(\tau)}_{H^k})\biggl(\norm{\Ft(\tau)}_{H^k}+1-
\frac{1}{\tau}\norm{\Pbb u(\tau)}_{H^{k-1}}^2\biggr)
\, d\tau.  \label{PbbperplimBa}
\end{align*}
From this inequality and the estimates \eqref{energy}
and \eqref{HkdecayB}, we deduce that
 $\Pbb^\perp u$ satisfies
\begin{equation*}
\norm{\Pbb^\perp u(t_2) - \Pbb^\perp u(t_1)}_{H^{k-2}} \lesssim
\begin{cases}  |t_2-t_1|
& \text{if $\zeta > 1$ } \\
  |t_2-t_1| +  (-t_1)^{2(\zeta-\sigma)}-(-t_2)^{2(\zeta-\sigma)}  &  \text{if $\zeta \leq 1$ }
 \end{cases}
\end{equation*}
for $T_0\leq t_1 < t_2 <0$.
Sending $t_2\nearrow 0$ then yields decay estimate
\begin{equation*} 
\norm{\Pbb^\perp u(t)-\Pbb^\perp u(0)}_{H^{k-2}} \lesssim 
\begin{cases}  |t|
& \text{if $\zeta > 1$ } \\
  |t| +  |t|^{2(\zeta-\sigma)}  &  \text{if $\zeta \leq 1$ }
 \end{cases}, 
\quad T_0\leq t < 0,
\end{equation*}
which gives better decay rates compared to those of Theorem \ref{symthm} under the same assumption, but of course, measured with respect to the weaker norm $H^{k-2}$.
\end{enumerate}
\end{rem}

\subsection{Time transformations}
\label{sec:timetrafos}

By employing a time transformation of the form
\begin{equation}\label{tptrans}
\tau = -(-t)^p, \qquad 0< p \leq 1,
\end{equation}
it is possible to handle coefficients with more singular behavior than allowed by the assumptions from Section \ref{coeffassump}.
To see this, assume that the coefficients of \eqref{symivp.1} satisfy the following:
\begin{enumerate}[(i*)]
\item The section $\Pbb \in \Gamma(L(V))$ satisfies
\begin{equation*} 
\Pbb^2 = \Pbb,  \quad  \Pbb^{\tr} = \Pbb, \quad \del{t}\Pbb =0 \AND \nabla \Pbb =0.
\end{equation*}
\item There exist constants  $\kappa, \gamma_1, \gamma_2 >0$ such that the maps $B^0 \in 
C^1\bigl([T_0,0), C^\infty(B_R(V),L(V))\bigr)$ 
and $\Bc\in C^0\bigl([T_0,0], C^\infty(B_R(V),L(V))\bigr)$  satisfy
\begin{equation*}
\pi(B^0(t,v))=\pi(\Bc(t,v))=\pi(v), 
\end{equation*}
and
\begin{equation} \label{B0BCbndtrafo}
\frac{1}{\gamma_1} \text{id}_{V_{\pi(v)}} \leq  B^0(t,v)\leq \frac{1}{\kappa} \Bc(t,v) \leq \gamma_2 \textrm{id}_{V_{\pi(v)}}
\end{equation}
for  all $(t,v)\in [T_0,0)\times B_{R}(V)$. Moreover,
\begin{align} 
[\Pbb(\pi(v)),\Bc(t,v)] &= 0, \label{BcPbbcomtrafo}\\
(B^0(t,v))^{\tr} &= B^0(t,v), \\ 
\Pbb(\pi(v)) B^0(t,v)\Pbb^\perp(\pi(v)) &= \Ord\bigl(|t|^{\frac{p}{2}}+\Pbb(\pi(v)) v\bigr), 
\intertext{and}
\Pbb^\perp(\pi(v)) B^0(t,v) \Pbb(\pi(v)) &= \Ord\bigl(|t|^{\frac{p}{2}}+\Pbb(\pi(v)) v\bigr)
\end{align}
for all $(t,v) \in [T_0,0) \times B_{R}(V)$,
and there exist maps $\Bt^0, \tilde{\Bc} \in C^0\bigl([T_0,0], \Gamma(L(V))\bigr)$
such that
\begin{align*}
[\Pbb,\tilde{\Bc}] &=0, \\
B^0(t,v)-\Bt^0(t,\pi(v)) &= \Ord(v), 
\intertext{and}
\Bc(t,v)-\tilde{\Bc}(t,\pi(v))&=\Ord(v) 
\end{align*}
for all $(t,v)\in  [T_0,0)\times B_R(V)$.

\item The map $F\in C^0\bigl([T_0,0), C^\infty(B_R(V),V)\bigr)$ can be expanded as
\begin{equation}\label{fexptrafo}
F(t,v) = |t|^{-(1-p)} \Ft(t,\pi(v)) + |t|^{-(1-p)}F_0(t,v) + |t|^{-(1-\frac{p}{2})}F_1(t,v) + |t|^{-1}F_2(t,v)
\end{equation}
where $\Ft \in C^0\bigl([T_0,0], \Gamma(V)\bigr)$, the maps
 $F_0,F_1,F_2 \in C^0\bigl([T_0,0], C^\infty(B_R(V),V)\bigr)$
 satisfy
 \begin{equation*}
 \pi (F_a (t,v))=\pi(v), \quad a=0,1,2,
 \end{equation*}
 and
\begin{equation*} 
\Pbb(\pi(v)) F_2(t,v) = 0 
\end{equation*}
for all $(t,v)\in [T_0,0]\times B_R(V)$, and
there exist constants $\lambda_a\geq 0$, $a=1,2,3$, such that
\begin{align*}
F_0(t,v) &=\Ord(v),\\
\Pbb(\pi(v)) F_1(t,v) &= \Ordc(\lambda_1 v),\\
\Pbb^\perp(\pi(v)) F_1(t,v) &= \Ordc(\lambda_2\Pbb(\pi(v)) v) 
\intertext{and}
\Pbb^\perp(\pi(v)) F_2(t,v) & = \Ordc\biggl(\frac{\lambda_3}{R}\Pbb(\pi(v)) v\otimes\Pbb(\pi(v))v \biggr)
\end{align*}
for all $(t,v)\in  [T_0,0)\times B_R(V)$.

\item The  map $B\in C^0\bigl([T_0,0), C^\infty(B_R(V),L(V)\otimes T\Sigma)\bigr)$ satisfies
\begin{equation*}
\pi(B(t,v))=\pi(v)
\end{equation*}
and
\begin{equation*}
\bigl[\sigma(\pi(v))(B(t,v))\bigr]^{\tr}=\sigma(\pi(v))(B(t,v))
\end{equation*}
for all $(t,v)\in [T_0,0)\times B_R(V)$ and $\sigma \in \mathfrak{X}^*(\Sigma)$. 
Moreover, $B$ can be expanded as
\begin{equation*} 
B(t,v) = |t|^{-(1-p)}B_0(t,v) + |t|^{-(1-\frac{p}{2})}B_1(t,v) + |t|^{-1}B_2(t,v)
\end{equation*}
where  $B_0,B_1,B_2 \in C^0\bigl([T_0,0], C^\infty(B_R(V),L(V)\otimes T^*\Sigma)\bigr)$ satisfy
\begin{equation*}
\pi(B_a(t,v))=\pi(v)
\end{equation*}
for all  $(t,v)\in [T_0,0]\times B_R(V)$, and
there exist a constant $\alpha\geq 0$ and a map $\Bt_2\in  C^0\bigl([T_0,0], \Gamma(L(V)\otimes T^*\Sigma)\bigr)$  such that
\begin{align*} 
\Pbb(\pi(v)) B_1(t,v) \Pbb(\pi(v)) &=   \Ord(1), \\
\Pbb(\pi(v)) B_1(t,v) \Pbb^\perp(\pi(v))&=  \Ordc(\alpha),\\
 \Pbb^\perp(\pi(v)) B_1(t,v) \Pbb(\pi(v)) &= \Ordc(\alpha), \\
\Pbb^\perp(\pi(v)) B_1(t,v) \Pbb^\perp(\pi(v)) &= \Ord(\Pbb(\pi(v)) v),\\
\Pbb(\pi(v)) B_2(t,v) \Pbb^\perp(\pi(v)) &= \Ord(\Pbb(\pi(v)) v),\\
\Pbb^\perp(\pi(v)) B_2(t,v)\Pbb(\pi(v)) &=  \Ord(\Pbb(\pi(v)) v),\\
\Pbb^\perp(\pi(v)) B_2(t,v) \Pbb^\perp(\pi(v))  &= \Ord\bigl(\Pbb(\pi(v)) v\otimes \Pbb(\pi(v)) v \bigr)\\
\intertext{and}
\Pbb(\pi(v))(B_2(t,v)-\Bt_2(t,\pi(v)))\Pbb(\pi(v)) &= \Ord(v)
\end{align*}
for all $(t,v)\in  [T_0,0)\times B_R(V)$.

\item  There exist constants $\theta\geq 0$, and $\beta_a \geq 0$, $a=0,1,\ldots,7$, such that the map
\begin{equation*}
\Div\! B \: : \: [T_0,0)\times B_R(V\otimes V\otimes T^*\Sigma)  \longrightarrow L(V)
\end{equation*} 
defined locally by
\begin{align*} 
&\Div\!  B(t,x,v,w) = \del{t}B^0(t,x,v)+  D_v B^0(t,x,v)\cdot (B^0(t,x,v))^{-1}\Bigl[-D_v B^i(t,x,v)\cdot w_i \notag \\
&\quad+
 \frac{1}{t}\Bc(t,x,v)\Pbb(x) v + F(t,x,v)
\Bigr]
+ \del{i}B^i(t,x,v)+ D_v B^i(t,x,v)\cdot (w_i-\omega_i(x) v) \notag \\
&\quad + \Gamma^i_{ij}(x)B^j(t,x,v)+\omega_i(x)B^i(t,x,v)-B^i(t,x,v)\omega_i(x),
\end{align*}
where $v=(v^J)$, $w=(w_i)$, $w_i=(w_i^J)$, $\omega_i=(\omega_{iI}^J)$, and $B^i=(B^{iJ}_I)$,
satisfies
\begin{align*}
\Pbb(\pi(v)) \Div \! B(t,v,w) \Pbb(\pi(v)) &= 
\Ordc\bigl(|t|^{-(1-p)}\theta+  |t|^{-(1-\frac{p}{2})}\beta_0 + |t|^{-1}\beta_1\bigl), \\
\Pbb(\pi(v))  \Div\! B(t,v,w) \Pbb^\perp(\pi(v)) &= 
\Ordc\biggl(|t|^{-(1-p)}\theta+|t|^{-(1-\frac{p}{2})}\beta_2
+ \frac{|t|^{-1}\beta_3}{R}\Pbb(\pi(v)) v\biggr), \\
\Pbb^\perp(\pi(v)) \Div\! B(t,v,w) \Pbb(\pi(v))&= 
\Ordc\biggl(|t|^{-(1-p)}\theta+|t|^{-(1-\frac{p}{2})}\beta_4
+ \frac{|t|^{-1}\beta_5}{R}\Pbb(\pi(v)) v\biggr) 
\intertext{and}
\Pbb^\perp(\pi(v)) \Div\! B(t,v,w) \Pbb^\perp(\pi(v))& =\Ordc\biggl(|t|^{-(1-p)}\theta+
\frac{|t|^{-(1-\frac{p}{2})}\beta_6}{R}\Pbb(\pi(v)) v
+ \frac{|t|^{-1}\beta_7}{R^2}\Pbb(\pi(v)) v\otimes\Pbb(\pi(v)) v \biggr). 
\end{align*}
\end{enumerate}

Now, a straightforward computation shows that under the time transformation \eqref{tptrans} equation \eqref{symivp.1}
transforms as
\begin{equation}\label{symivpb}
\Bb^0(\tau,\ub)\del{\tau} \ub + \Bb^i(\tau,\ub)\nabla_i \ub = \frac{1}{\tau} \bar{\Bc}(\tau,\ub)\Pbb \ub + \Fb(\tau,\ub)
\end{equation}
where
\begin{align*}
\ub(\tau,x) &= u\bigl(-(-\tau)^\frac{1}{p},x\bigr), \\
\Bb^0(\tau,\ub) &= B^0\bigl(-(-\tau)^\frac{1}{p},\ub\bigr),\\
\Bb^i(\tau,\ub) &= \frac{(-\tau)^{\frac{1-p}{p}}}{p}B^i\bigl(-(-\tau)^\frac{1}{p},\ub\bigr), \\ 
\bar{\Bc}(\tau,\ub) &= \frac{1}{p} \Bc\bigl(-(-\tau)^\frac{1}{p},\ub\bigr)
\intertext{and}
\Fb(\tau,\ub) &= \frac{(-\tau)^{\frac{1-p}{p}}}{p}F\bigl(-(-\tau)^\frac{1}{p},\ub\bigr),
\end{align*}
while $\Div\! B$ transforms as
\begin{equation*}
 \frac{(-\tau)^{\frac{1-p}{p}}}{p}\Div\! B\bigl(-(-\tau)^\frac{1}{p},x,v,w\bigr)
 = \Div\! \Bb(\tau,x,v,w).
 \end{equation*}
Using these formulas, it is then not difficult to see that the assumptions (i*)-(v*) imply that \eqref{symivpb}
 satisfies assumption (i)-(v) for an appropriate choice of constants. Thus, under the assumptions (i*)-(v*), we can apply 
 Theorem \ref{symthm} to \eqref{symivpb} to obtain the existence of solutions to \eqref{symivp.1} all the way up to the singular time 
 $\tau=0$ for a suitably small choice of initial data and under the condition $\sup_{T_0 \leq \tau < 0}\norm{\Ft(\tau)}_{H^k}\le\delta$.
 In fact, we can show that if the constants found from assumptions (i*)-(v*) for the original system \eqref{symivp.1} satisfy \eqref{eq:symhypkappa}, then the corresponding constants for the transformed system \eqref{symivpb} also satisfy \eqref{eq:symhypkappa}. Moreover
the decay estimates implied by the theorem for solutions of \eqref{symivpb} take the following form for the corresponding solutions of the original system
\begin{align}
\label{eq:asymptoticstrafo1}
\norm{\Pbb u(t)}_{H^{k-1}} &\lesssim \begin{cases}
|t|^p+(\lambda_1+\alpha)|t|^{\frac{p}{2}} & \text{if $\zeta > p$} \\
|t|^{\zeta-\sigma}+(\lambda_1+\alpha)|t|^{\frac{p}{2}} & \text{if $\frac{p}{2} < \zeta \leq p$} \\
|t|^{\zeta-\sigma}   & \text{if $0 < \zeta \leq \frac{p}{2}$}\end{cases}
\intertext{and}
\label{eq:asymptoticstrafo2}
\norm{\Pbb^\perp u(t) - \Pbb^\perp u(0)}_{H^{k-1}} &\lesssim
 \begin{cases}  |t|^{p/2}+ |t|^{\zeta-\sigma}
& \text{if $\zeta > p/2$} \\
 |t|^{\zeta-\sigma}  &  \text{if $\zeta \leq p/2 $ }
 \end{cases},
\end{align}
where the constants $\lambda_1$ and $\alpha$ are found as in assumptions (i*)-(v*) for the original system \eqref{symivp.1}, and $\zeta$ and $\sigma$ are defined as in Theorem~\ref{symthm} in terms of the constants found for the original system. Remark~\ref{Gdecay} can be easily modified to apply here as well.

\subsection{Fuchsian heuristics}
The point of the following brief discussion is to develop some intuition for Fuchsian systems and what to expect from Theorem~\ref{symthm}.
This \emph{Fuchsian heuristics}  is largely based on the idea that, under suitable assumptions, the asymptotics of systems like \eqref{symivp.1} should be determined by an associated linear ODE of the form
  \begin{equation}
    \label{eq:heuristicsexamplerprobl1}
    \del{t}u  = \frac{1}{t}\tilde{\Bc} u + F
  \end{equation}
where
\begin{equation*}
  F=|t|^{-(1-p)} \Ft(t), \quad 0<p\leq 1,
\end{equation*}
and $\Ft\in C^0([-1,0])$. By assumption, there are no spatial derivatives, and consequently, we can assume for the purpose of this argument that all the fields are constant on the spatial manifold $\Sigma$ (i.e., in local coordinates, they are independent of the spatial coordinates $x$). 

In order to obtain a useful example that will illustrate the general situation, we restrict ourselves, in the following,
to unknowns $u$
in \eqref{eq:heuristicsexamplerprobl1} with two components. We will be interested in the situation where one component has a limit, possibly non-zero, as $t\nearrow 0$ while the other decays to $0$ at some fixed rate. To this end, we further assume that
\begin{equation}
  \label{eq:heuristicsexamplerprobl3b}
\tilde{\Bc}=
\begin{pmatrix}
  0 & 0\\
  0 & a
\end{pmatrix}
\end{equation}
for some $a>0$. Then, fixing initial data
\begin{equation} \label{eq:heuristicsexamplerprobl3a}
u(-1)=(u_*,u_{**})^{\tr},
\end{equation}
it is straightforward to integrate 
\eqref{eq:heuristicsexamplerprobl1} to obtain the
(unique) solution given by
\begin{equation}
  \label{eq:heuristicsexsol}
  u(t)=\begin{pmatrix}
    u^1(t)\\
    u^2(t)
  \end{pmatrix}=
  \begin{pmatrix}
    u_*+\int_{-1}^t |s|^{-1+p} \Ft^{1}(s) ds\\
    (-t)^a\bigl(u_{**}+\int_{-1}^t |s|^{-1+p-a} \Ft^{2}(s) ds\bigr)
  \end{pmatrix}, \quad t\in [-1,0).
\end{equation}
From this formula, we see immediately that $\lim_{t\nearrow 0} u^1(t)$, denoted $u^1(0)$, exists (since $p>0$) and
is given by
\begin{equation}
  \label{eq:heuristicsexestA}
  u^1(0)=u_*+\int_{-1}^0 |s|^{-1+p} \Ft^{1}(s) ds,
\end{equation}
and the decay estimates
\begin{equation}
  \label{eq:heuristicsexestB}
  |u^1(t)-u^1(0)|\lesssim |t|^p \AND |u^2(t)|\lesssim |t|^p+|t|^a
\end{equation}
hold. This shows that if the source term $F$ is not too singular (i.e., if $p\ge a$), then \emph{the solutions of \eqref{eq:heuristicsexamplerprobl1} behave like powers of $|t|$ where these powers are the eigenvalues of the {principal matrix} \eqref{eq:heuristicsexamplerprobl3b}}. On the other hand, if the source term is very singular  (i.e., if $p< a$), then there are $|t|^p$-''corrections'' to these decay rates. 

We now claim that the optimal decay rates \eqref{eq:heuristicsexestB} of the above solutions can, up to an arbitrarily small loss, be inferred from 
Theorem~\ref{symthm}. To see this, we set
\begin{equation}
  \label{eq:heuristicsexamplerprobl4}
\Bc=
\begin{pmatrix}
  a & 0\\
  0 & a
\end{pmatrix}\AND \Pbb=\begin{pmatrix}
  0 & 0\\
  0 & 1
\end{pmatrix}.
\end{equation}
We then have
\begin{equation}
  \label{eq:heuristicsexamplerprobl3}
\tilde{\Bc}= \Bc\Pbb,
\end{equation}
which we can use to write
\eqref{eq:heuristicsexamplerprobl1} as
\begin{equation}
    \label{eq:heuristicsexamplerprobl5}
    \del{t}u  = \frac{1}{t}\Bc\Pbb u + F.
\end{equation}
With a view to applying Theorem \ref{symthm} to this equation,
we observe from \eqref{eq:heuristicsexamplerprobl4} that the
optimal (i.e. largest) choice for $\kappa$, see \eqref{B0BCbndtrafo}, is\footnote{Note here that $B^0=\id$.} $\kappa=a$. We further observe we can choose $\gamma_1=\gamma_2=1$ in \eqref{B0BCbndtrafo}, and zero for all other constants found from assumptions (i*)-(v*) in Section~\ref{sec:timetrafos}. With these choices, \eqref{eq:symhypkappa} is satisfied, and so, as discussed in Section~\ref{sec:timetrafos}, all the hypothesises of Theorem~\ref{symthm} have been satisfied. We can, therefore, conclude from Theorem \ref{symthm} the existence
of a unique solution $u(t)$, $t\in [-1,0)$, that satisfies
the initial conditions \eqref{eq:heuristicsexamplerprobl3a}
and for which the limit $\lim_{t\nearrow 0} \Pbb^\perp u(t)$, denoted $\Pbb^\perp u(0)$, exists. By uniqueness, this solution must coincide with one
given above by
\eqref{eq:heuristicsexsol}. By Remark~\ref{Gdecay} (case (a)(i)), we also conclude from Theorem \ref{symthm} that
the decay rates of the solution $u(t)$ are given by
\begin{equation*}
|\Pbb u(t)| \lesssim \begin{cases}
|t|^p & \text{if $a > p$} \\
|t|^{a-\sigma} & \text{if $0 < a \leq p$} 
\end{cases} \AND
|\Pbb^\perp u(t) - \Pbb^\perp u(0)| \lesssim
   |t|^p.
\end{equation*}
Up to the (arbitrarily) small loss given by the constant $\sigma>0$, these decay rates are fully consistent with the optimal rates given by \eqref{eq:heuristicsexestB}. 

As a final remark, there are many ways we could have chosen $\Bc$ and $\Pbb$ to get \eqref{eq:heuristicsexamplerprobl3}. 
The above choice for $\Pbb$ is motivated by the expectation from the decay estimates \eqref{eq:asymptoticstrafo1}-\eqref{eq:asymptoticstrafo2} and \eqref{eq:heuristicsexestA}-\eqref{eq:heuristicsexestB} that $\Pbb u=u^2$ should decay to zero and $\Pbb^\perp u=u^1$ should have a non-zero limit at $t=0$. Given this, a general constant ansatz for $\Bc$ is then reduced to diagonal form by \eqref{eq:heuristicsexamplerprobl3} and \eqref{BcPbbcomtrafo}.

\section{Applications\label{applications}}

We now turn to three applications of Theorem \ref{symthm}. The first two involve systems of semilinear wave equations near spatial infinity on Minkowski and Schwarzschild spacetimes, while
the third involves perfect fluids with Gowdy symmetry on Kasner spacetimes near the big bang singularity. As discussed in the introduction, a special case of Theorem \ref{symthm}, i.e. \cite[Theorem B.1]{Oliynyk:CMP_2016}, has already been used to establish the global existence of solutions to the Einstein-Euler equations that represent non-linear perturbations of FLRW spacetimes and also the global existence of
solutions to the Euler equations on FLRW spacetimes
\cite{LeFlochWei:2015,LiuOliynyk:2018b,LiuOliynyk:2018a,LiuWei:2019,Oliynyk:CMP_2016,Wei:2018}.

\subsection{Wave equations on Minkowski spacetime near spatial infinity}

\subsubsection{The cylinder at spatial infinity}
The \textit{cylinder at spatial infinity} construction was first introduced by Friedrich  in \cite{Friedrich:JGP_1998} as a tool to
analyze the behavior of his conformal  version of the Einstein field equations, see \cite{Friedrich1981,Friedrich:1986,Friedrich:1991}, near spatial infinity.  
In this section, we use Friedrich's cylinder at spatial infinity construction in conjunction with Theorem \ref{symthm} to obtain the existence of solutions to a class of semilinear wave equations near spatial infinity in Minkowski spacetime. For other applications of the cylinder at spatial
infinity construction to linear wave and spin-2 equations on Minkowski spacetime, see \cite{Beyer_et_al:2014,DoulisFrauendiener:2013,FrauendienerHennig:2014,Friedrich:2002,MacedoKroon:2018}.

Following \cite{FrauendienerHennig:2014}, we employ the cylinder at spatial infinity representation of Minkowski spacetime, that is the spacetime consisting of the manifold
\begin{equation*}
M = (0,2)\times (0,\infty) \times \mathbb{S}^2
\end{equation*}
and the (conformal) Lorentzian metric
\begin{equation} \label{Mgdef}
g= -dt\otimes dt + \frac{1-t}{r}(dt\otimes dr + dr\otimes dt) +\frac{(2-t)t}{r^2}dr\otimes dr +\gsl
\end{equation}
where $(t,r)\in (0,2)\times (0,\infty)$ and
\begin{equation}\label{gsldef}
\gsl = d\theta \otimes d\theta + \sin^2(\theta) d\phi \otimes d\phi
\end{equation}
is the canonical metric on the $2$-sphere $\mathbb{S}^2$. 
Here, we use\footnote{In this and the next section, we will employ lower case Greek letters, e.g. $\mu$, $\nu$, $\gamma$, to label spacetime 
coordinate indices that run from $0$ to $3$, while upper case Greek letter, e.g. $\Lambda$, $\Sigma$, $\Gamma$, will be reserved to label spherical
coordinate indices that run from 2 to 3, i.e. $(x^\Lambda)=(x^2,x^3) :=(\theta,\phi)$.}
\begin{equation*}
(x^\mu) = (x^0,x^1,x^2,x^3) = (t,r,\theta,\phi)
\end{equation*}
to label the coordinates. The importance of the metric \eqref{Mgdef} is that
\begin{equation} \label{Mctrans}
\gt = \Omega^2 g,
\end{equation}
where
\begin{equation}\label{MOmegadef}
\Omega = \frac{1}{r(2-t)t},
\end{equation}
defines a flat Lorentzian metric on $M$, i.e., $(M,\gt)$ is a part of Minkowski space. The flatness of $\gt$ can be seen either by a direct calculation using the formulas \eqref{Mgdef}, \eqref{Mctrans} and \eqref{MOmegadef}, or  by letting
\begin{equation*}
(\xb^\mu) = (\xb^0,\xb^1,\xb^2,\xb^3)=(\tb,\rb,\theta,\phi)
\end{equation*}
denote spherical coordinates on $\Rbb^4$ and noting that
\begin{equation} \label{psibdef}
\psib \: :\: \Mb  \longrightarrow M \: : \: (\xb^\mu) = (\tb,\rb,\theta,\phi) \longmapsto (x^\mu) = \biggl(1-\frac{\tb}{\rb},\frac{\rb}{-\tb^2+\rb^2},\theta,\phi\biggr)
\end{equation}
defines a diffeomorphism from the manifold\footnote{Geometrically, $\Mb$ is the interior of the spacelike cone with vertex at the origin in $\Rbb^4$.}
\begin{equation*}
\Mb = \{ \, (\tb,\rb,(\theta,\phi)) \in (-\infty,\infty)\times (0,\infty)\times \mathbb{S}^2 \: |\: -\tb^2+\rb^2 > 0 \, \}
\end{equation*}
to $M$ that satisfies
\begin{equation*}
\psib_* \gb = \gt
\end{equation*}
where
\begin{equation*}
\gb =  -d\tb\otimes d\tb + d\rb \otimes d\rb + \rb^2\gsl
\end{equation*}
is the Minkowski metric on $\Rbb^4$ in spherical coordinates.  

The diffeomorphism \eqref{psibdef} defines a compactification of null rays in $\Mb$. Decomposing the boundary of $M$ as
\begin{equation*}
\del{}M = \Isc^+ \cup i^0 \cup \Isc^{-},
\end{equation*}
where
\begin{equation*}
\Isc^{+} = \{0\}\times (0,\infty) \times \mathbb{S}^2, \quad  \Isc^{-} = \{2\}\times (0,\infty) \times \mathbb{S}^2 \AND i^0 = [0,2]\times \{0\} \times \mathbb{S}^2,
\end{equation*}
the compatification defined by \eqref{psibdef} leads to the interpretation of $\Isc^{\pm}$ as portions of ($+$) future and ($-$) past null-infinity, respectively, and $i^0$
as spatial infinity.
We further note that  the spacelike hypersurface 
\begin{equation*}
\Sigma = \{1\} \times (0,\infty) \times \mathbb{S}^2 \subset M
\end{equation*}
corresponds to the constant time hypersurface
\begin{equation*}
\Sigmab = \{ \, (\xb^{\mu})\in \Rbb^4 \: |\: \gb_{\mu\nu}\xb^{\mu}\xb^{\nu} > 0,\; \xb^0=0\, \}\subset \Mb \subset \Rbb^4
\end{equation*}
in Minkowski spacetime.

For use below, we observe that the Ricci scalar curvature of $\gt$ satisfies
\begin{equation} \label{MRtdef}
\Rt =0
\end{equation}
by virture of $\gt$ being flat. The same is also true for the Ricci scalar curvature of the metric $g$, that is, 
\begin{equation}\label{MRdef}
R =0,
\end{equation}
as can be verified via a straightforward calculation using \eqref{Mgdef}.

\subsubsection{Semilinear wave equations}
The class of semilinear wave equations on Minkowski spacetime that we consider are systems of $N$-coupled scalar wave equations
of the form
\begin{equation}  
\gb^{\alpha\beta}\nablab_\alpha \nablab_\beta \ub^K = q_{IJ}^K(\ub^L)\gb^{\alpha\beta}\nablab_\alpha \ub^I \nablab_\beta \ub^J \label{Mbwave} 
\end{equation}
where $q_{IJ}^K \in C^\infty(\Rbb^N)$, $\nablab$ is the Levi-Civita connection of
$\gb$, and in line with our notation from Section \ref{Vbundle},
upper case Latin indices, e.g. $I,J,K$, run from $1$ to $N$.
We will be interested in solving this type of system of wave equations on domains of the form
\begin{equation} \label{Mbrb0def}
\Mb_{\rb_0} = \{ \, (\tb,\rb) \in (0,\infty)\times (0,\infty) \: |\: -\tb + \rb > \rb_0 \, \} \times \mathbb{S}^2,\qquad \rb_0>0.
\end{equation}

\begin{rem}$\;$
\begin{enumerate}[(i)]
\item Because the nonlinear terms $q_{IJ}^K(\ub^L)\gb^{\alpha\beta}\nablab_\alpha \ub^I \nablab_\beta \ub^J$
satisfy the null condition of Klainerman, global
existence results, under a small initial data condition, for systems of wave equation of the form
\eqref{Mbwave} have been known since the pioneering  work of Klainerman \cite{Klainerman:1980} and Christodoulou \cite{Christodoulou:1986}.
In this sense, the results that we present are not new. However, the method we use to establish global existence on regions of the form \eqref{Mbrb0def} is new and we 
believe it brings a valuable 
new perspective to global existence problems for systems of nonlinear wave equations. The analysis carried out here is also of interest because it
demonstrates the utility of Friedrich's cylinder at spatial infinity construction for solving nonlinear wave equations near spatial infinity. 
\item The analysis carried out here on Minkowski spacetimes can be generalized to systems of wave equations satisfying
the weak-null condition. We will report on this work in a separate article \cite{OliynykOlvera:2020}. 
\end{enumerate}
\end{rem}

Rather than attempt to solve \eqref{Mbwave} directly, we use the diffeomorphism \eqref{psibdef} to push \eqref{Mbwave} forward to obtain
the system
\begin{equation}  
\gt^{\alpha\beta}\nablat_\alpha \nablat_\beta \ut^K = q^K_{IJ}(\ut^L)\gt^{\alpha\beta}\nablat_\alpha \ut^I \nablat_\beta \ut^J \label{Mtwave}
\end{equation}
on the domain
\begin{equation} \label{Mr0def}
M_{r_0} :=\psib(\Mb_{\rb_0}) = \biggl\{ \, (t,r) \in (0,1)\times (0,r_0) \: \biggl|\: t >2- \frac{r_0}{r} \, \biggr\} \times \mathbb{S}^2, \qquad r_0 = \frac{1}{\rb_0},
\end{equation}
where  $\nablat$ is the Levi-Civita connection of $\gt$ and $\ut^K$ is related to $\ub^K$ via
\begin{equation} \label{ub2ut}
\ut^K = \psib_* \ub^K.
\end{equation}
Due to the invertibility of the diffeomorphism \eqref{psibdef}, the relation \eqref{ub2ut} establishes an equivalence between the wave equations  \eqref{Mbwave} on $\Mb_{\rb_0}$ and \eqref{Mtwave}
on $M_{r_0}$.
We are therefore free to restrict our attention to wave equations of the form \eqref{Mtwave} on $M_{r_0}$.
We further note that the constant time hypersurface 
\begin{equation} \label{Sigmabrb0def}
\Sigmab_{\rb_0} = \{0\}\times (\rb_0,\infty)\times \mathbb{S}^2,
\end{equation}
which forms the bottom of $\Mb_{\rb_0}$, gets mapped via the diffeomorphism \eqref{psibdef} to the constant time hypersurface
\begin{equation*}
\Sigma_{r_0} := \psib(\Sigmab_{\rb_0})= \{1\}\times (0,r_0)\times \mathbb{S}^2
\end{equation*}
that forms the top of $M_{r_0}$.

Using 
\begin{equation*}
\ft^K =  q^K_{IJ}(\ut^L)\gt^{\alpha\beta}\nablat_\alpha \ut^I \nablat_\beta \ut^J
\end{equation*}
to denote the nonlinear terms that appear in \eqref{Mtwave}, it follows immediately from
\eqref{MRtdef}-\eqref{MRdef} and the formulas \eqref{utrans}-\eqref{wavetransC}  and \eqref{ftransB}-\eqref{ftransC}, with $n=4$, from Appendix \ref{ctrans}
that the wave equations \eqref{Mtwave} transform, under the conformal transformation \eqref{Mctrans}, into
\begin{equation} \label{MwaveA}
g^{\alpha\beta}\nabla_\alpha \nabla_\beta u^K = f^K
\end{equation}
where $\nabla$ is the Levi-Civita connection of $g$,
\begin{equation*} 
\ut^K = rt(2-t)u^K
\end{equation*}
and
\begin{align*}
f^K = &q_{IJ}^K\bigl(rt(2-t)u^L\bigr)\biggl( rt(2-t)g^{\mu\nu}\nabla_\mu u^I\nabla_\nu u^J \\
&+2g^{\mu\nu}\nabla_\mu\bigl( rt(2-t)\bigr) \nabla_\nu u^{(I} u^{J)}
+ \frac{1}{rt(2-t)}g^{\mu\nu}\nabla_\mu\bigl(rt(2-t)\bigr) \nabla_\nu(rt(2-t)) u^I u^J\biggr).
\end{align*}
A routine computation using \eqref{Mgdef} then shows that conformal wave equations \eqref{MwaveA} can be expressed as
\begin{equation} \label{MwaveB}
(-2+t)t\del{t}^2 u^K + r^2 \del{r}^2 u^K + 2r(1-t)\del{r}\del{t}u^K + \gsl^{\Lambda\Sigma}\nablasl{\Lambda}\nablasl{\Sigma} u^K + 2(t-1)\del{t}u^K = f^K
\end{equation}
where $\nablasl{\Lambda}$ is the Levi-Civita connection of the metric \eqref{gsldef} on $\mathbb{S}^2$ and 
\begin{align} 
f^K = &q^K_{IJ}\bigl(rt(2-t)u^L\bigr)\biggl(2r(2-2t+t^2) u^{(I} r\del{r}u^{J)}-r(2-t)^2 t^2\del{t}u^I\del{t}u^J+r(2-t)t r\del{r}u^I r\del{r}u^J \notag \\
&+2 r(2-3t+t^2)t(\del{t}u^{(I}r\del{r}u^{J)}-\del{t}u^{(I}u^{J)})+r(2-t)t \gsl{}^{\Lambda\Sigma}\nablasl{\Lambda} u^I \nablasl{\Sigma} u^J + r(2-t)t u^I u^J\biggl). \label{MfA}
\end{align}

To proceed, we define new variables $U_0^J$, $U_1^J$, $U_2^J$, $U_3^J$ and $U_4^J$ by setting
\begin{align*}
U_0^J = t^{\lambda + \frac{1}{2}} \del{t}u^J, \quad 
U_1^J = t^{\lambda} r\del{r}u^J, \quad 
U_\Lambda^J  = t^{\lambda} \nablasl{\Lambda} u^J 
\AND
U_4^J &= t^{\lambda-\frac{1}{2}}u^J
\end{align*}
where $\lambda\in \Rbb$ is a constant to be fixed later.
Using these variables, we can, with the help of \eqref{MfA}, write the conformal system of wave equations \eqref{MwaveB}
in first order form as follows:
\begin{equation} \label{MwaveC}
B^0 \del{t}U +  B^1 r\del{r}U + B^\Gamma \nablasl{\Gamma} U = \frac{1}{t}\Bc U + F  
\end{equation} 
where
\begin{align}
U &= \begin{pmatrix} U^J_0\\ U^J_1 \\ U^J_\Sigma \\ U^J_4\end{pmatrix}, \label{MUdef}\\[0.3cm]
B^0 &= \begin{pmatrix}
(2-t)\delta^K_J & 0 & 0 & 0\\
0 & \delta^K_J & 0 & 0 \\
0 & 0 & \delta^{\Sigma}_\Lambda\delta^K_J & 0\\
0 & 0 & 0 & \delta^K_J
\end{pmatrix}, \label{MB0def}\\[0.3cm]
B^1 &= \begin{pmatrix}
\begin{displaystyle}\frac{2(t-1)}{t}\delta^K_J \end{displaystyle} & -\begin{displaystyle}\frac{1}{t^{\frac{1}{2}}}\delta^K_J\end{displaystyle} & 0 & 0\\
 -\begin{displaystyle}\frac{1}{t^{\frac{1}{2}}}\delta^K_J\end{displaystyle} & 0 & 0 & 0 \\
0 & 0 & 0 & 0\\
0 & 0 & 0 & 0
\end{pmatrix}, \label{MB1def}\\[0.3cm]
B^\Gamma &= \begin{pmatrix}
0 & 0 & -\begin{displaystyle}\frac{1}{t^{\frac{1}{2}}}\end{displaystyle}  \gsl^{\Gamma\Sigma}\delta^K_J & 0\\
0 & 0 & 0 & 0 \\
-\begin{displaystyle}\frac{1}{t^{\frac{1}{2}}}\end{displaystyle}  \delta^\Gamma_\Lambda \delta^K_J & 0 & 0 & 0\\
0 & 0 & 0 & 0
\end{pmatrix}, \label{MBKdef}\\[0.3cm]
\Bc &= \begin{pmatrix}
\begin{displaystyle}2\biggl(\lambda-\frac{1}{2}\biggr)\delta^K_J \end{displaystyle}& 0 & 0 & 0 \\ 
0 &\lambda\delta^K_J & 0 & 0 \\
0 & 0 & \lambda \delta^\Sigma_\Lambda\delta^K_J & 0 \\
\delta^K_J & 0 & 0 &\begin{displaystyle} \biggl(\lambda-\frac{1}{2}\biggr) \delta^K_J\end{displaystyle}
\end{pmatrix}\label{MBcdef}
\end{align}
and
\begin{equation}\label{MFdef}
F = \begin{pmatrix} F_0^K \\ 0 \\ 0 \\ 0\end{pmatrix}
\end{equation}
with
\begin{align*}
F_0^K = \biggl(\frac{3}{2}-\lambda \biggr)U^K_0 +\frac{1}{t}\biggl[&- t^{\frac{1}{2}} U^K_1 + q^K_{IJ}\bigl(r(2-t)t^{\frac{3}{2}-\lambda}U^L_4\bigr)\Bigl(
-2r(2-2t+t^2)t^{1-\lambda}U^{(I}_4 U_1^{J)} + \\
 r(t-2)^2& t^{\frac{3}{2}-\lambda}U_0^I U_0^J+r(t-2) t^{\frac{3}{2}-\lambda}U_1^I U_1^J
+2r(2-3t+t^2)\bigl( t^{\frac{3}{2}-\lambda}U_4^{(I} U_0^{J)} \\
&- t^{1-\lambda}U_0^{(I} U_1^{J)}\bigr)+r(t-2)t^{\frac{5}{2}-\lambda}
U_4^I U_4^J
+ r(t-2)t^{\frac{3}{2}-\lambda}\gsl^{\Lambda\Sigma}U_\Lambda^I U_\Sigma^J
\Bigr)\biggr].
\end{align*}
We further note that \eqref{MFdef} can be expressed as
\begin{equation}\label{MFrep}
F = \begin{pmatrix}\biggl(\frac{3}{2}-\lambda \biggr)U^K_0  \\ 0 \\ 0 \\ 0\end{pmatrix}
+\frac{1}{t}\bigl(t^{\frac{1}{2}}\Csc+\Bsc\bigr)U
\end{equation}
where 
\begin{equation*}
\Csc = \begin{pmatrix}0 & -\delta^K_J & 0& 0\\
0 & 0& 0& 0\\
0 & 0& 0& 0\\
0 & 0& 0& 0
\end{pmatrix}
\end{equation*}
and
\begin{equation*}
\Bsc=  \begin{pmatrix}
\Bsc_{J}^{0K} & \Bsc_{J}^{1K}   & \Bsc^{\Sigma K}_J & \Bsc_{J}^{4K}   \\ 
0 & 0 & 0 & 0 \\
0 & 0 & 0 & 0 \\
0 & 0 & 0 & 0
\end{pmatrix}
\end{equation*}
with
\begin{align*}
\Bsc_{J}^{0K} &= q^K_{PQ}\bigl(r(2-t)t^{\frac{3}{2}-\lambda}U_4^L\bigr)\delta_I^{(P}\delta_J^{Q)}\bigl[ r(t-2)^2 t^{\frac{3}{2}-\lambda}U_0^I-2r(2-3t+t^2) t^{1-\lambda}U_1^I\bigr],\\
\Bsc_{J}^{1K} &=   q^K_{PQ}\bigl(r(2-t)t^{\frac{3}{2}-\lambda}U_4^L\bigr)\delta_I^{(P}\delta_J^{Q)}\bigl[-2r(2-2t+t^2)t^{1-\lambda}U_4^I+r(t-2) t^{\frac{3}{2}-\lambda}U_1^I\bigr],\\
\Bsc^{\Sigma K}_J &=  q^K_{PQ}\bigl(r(2-t)t^{\frac{3}{2}-\lambda}U_4^L\bigr)\delta_I^{(P}\delta_J^{Q)} r(t-2)t^{\frac{3}{2}-\lambda}\gsl^{\Lambda\Sigma}U^I_\Lambda
\intertext{and}
\Bsc_{J}^{4K} &= q^K_{PQ}\bigl(r(2-t)t^{\frac{3}{2}-\lambda}U_4^L\bigr)\delta_I^{(P}\delta_J^{Q)}\bigl[r(t-2)t^{\frac{5}{2}-\lambda}U^I_4 +2r(2-3t+t^2)t^{\frac{3}{2}-\lambda} U^I_0 \bigr].
\end{align*}
From these formulas, it is clear that
\begin{equation} \label{Bsczero}
\Bsc = \Ord(U) \quad \text{for $\lambda \leq 1$.}
\end{equation}

Letting
\begin{equation} \label{Vdef}
V = \begin{pmatrix} V^J_0\\ V^J_1 \\ V^J_\Sigma \\ V^J_4\end{pmatrix},
\end{equation}
we define a positive definite inner-product $h$ by
\begin{equation}\label{Mhdef}
h(U,V) = \delta_{IJ}\bigl(U^I_0V^J_0 + U^I_1 V^J_1 + \gsl^{\Lambda\Sigma}U^I_\Lambda V^J_\Sigma + U^I_4 V^J_4\bigr).
\end{equation}
It is then not difficult to see that $B^0$, $B^1$ and $B^\Lambda$ are all symmetric with respect to this inner-product and that
\begin{equation*}
h(V,B^0 V) \geq h(V,V)
\end{equation*} 
as long as $t\in [0,1]$. This implies that the system of equations \eqref{MwaveC} is symmetric hyperbolic. 

To proceed, we introduce a change of radial coordinate via
\begin{equation} \label{Mrhodef}
r = \rho^m, \quad m\in \Zbb_{\geq 1}.
\end{equation}
Using the transformation law
\begin{equation*}
r\del{r} = r \frac{d\rho}{dr} \del{\rho} = \frac{\rho}{m}\del{\rho},
\end{equation*}
we can express \eqref{MwaveC} as
\begin{equation} \label{MwaveD}
B^0 \del{t}U +  \frac{\rho}{m}B^1 \del{\rho}U + B^\Gamma \nablasl{\Gamma} U = \frac{1}{t}\Bc U + F,
\end{equation}
where now any $r$ appearing in $F$  is replaced using \eqref{Mrhodef}.  We further observe that the spacetime region \eqref{Mr0def} can be expressed
in terms of the radial coordinate $\rho$ as
\begin{equation} \label{Mr02rho}
M_{r_0} = \biggl\{ \, (t,\rho) \in (0,1)\times (0,\rho_0) \: \biggl|\: t >2- \frac{\rho_0^m}{\rho^m} \, \biggr\} \times \mathbb{S}^2, \quad \rho_0 = (r_0)^{\frac{1}{m}}.
\end{equation}

\subsubsection{The extended system}
Next, we let $\hat{\chi}(\rho)$ denote a smooth cut-off function that satisfies
\begin{equation*}
\hat{\chi} \geq 0, \quad \hat{\chi}|_{[-1,1]} = 1 \AND \supp(\hat{\chi}) \subset (-2,2),
\end{equation*}
and we define
\begin{equation}\label{chirho0def}
\chi(\rho) = \hat{\chi}\biggl(\frac{\rho}{\rho_0}\biggr),
\end{equation}
which is easily seen to satisfy
\begin{equation*}
\chi \geq 0, \quad \chi|_{[-\rho_0,\rho_0]} = 1 \AND \supp(\chi) \subset (-2\rho_0,2\rho_0).
\end{equation*}
We then consider an extended version of \eqref{MwaveD} given by
\begin{equation} \label{MwaveE}
B^0 \del{t} U +  \frac{\chi\rho }{m} B^1\del{\rho} U + B^\Gamma \nablasl{\Gamma} U = \frac{1}{t}\Bc U +\chi F  
\end{equation} 
that is well-defined on the extended spacetime region
\begin{equation}\label{MMscdef}
(0,1) \times  T^{1}_{3\rho_0}\times \mathbb{S}^2
\end{equation} 
where $T^{1}_{3\rho_0} \cong \mathbb{S}^1$ is the 1-dimensional torus obtained from identifying the end points of the interval $[-3\rho_0,3\rho_0]$.
By construction, \eqref{MwaveE} agrees with \eqref{MwaveD} when restricted to \eqref{Mr02rho}. 
Noting that the boundary of the region \eqref{Mr02rho} can be decomposed as
\begin{equation*}
\del{}M_{r_0} = \Sigma_{r_0}\cup \Sigma^+_{r_0} \cup \Gamma^{-} \cup \Gamma^{+}
\end{equation*}
where
\begin{gather*}
\Sigma_{r_0} =  \{1\}\times (0,\rho_0) \times \mathbb{S}^2, \quad
\Sigma^+_{r_0} = \{0\}\times \Bigl(0,\frac{\rho_0}{2^{\frac{1}{m}}}\Bigr)\times \mathbb{S}^2, \\
\Gamma^{-} = [0,1] \times \{0\} \times \mathbb{S}^2
\AND
\Gamma^{+} = \biggl\{ \, (t,r) \in [0,1]\times (0,\rho_0] \: \biggl|\: t = 2- \frac{\rho_0^m}{\rho^m} \, \biggr\} \times \mathbb{S}^2,
\end{gather*} 
we see immediately that
\begin{equation*}
n^{-}= -d\rho \AND n^{+} = -dt + m\frac{\rho_0^m}{\rho^{m+1}} d\rho
\end{equation*}
define outward pointing co-normals to $\Gamma^{-}$ and $\Gamma^{+}$, respectively. Furthermore, from \eqref{MB0def}-\eqref{MBKdef}, we get that
\begin{align*}
\Bigl(n^{-}_{0} B^0 +n^{-}_1 \frac{\chi\rho}{m} B^1 + n^{-}_\Gamma B^\Gamma\Bigr) \Bigl|_{\Gamma^{-}} &= 0
\intertext{and}
 \Bigl(n^{+}_{0} B^0 +n^{+}_1 \frac{\chi\rho}{m} B^1 + n^{+}_\Gamma B^\Gamma\Bigr)\Bigl|_{\Gamma^{+}} & = \left. \begin{pmatrix}
\begin{displaystyle}\biggl(-(2-t)+\frac{\rho_0^m}{\rho^m}\frac{2(t-1)}{t}\biggr) \delta^K_J\end{displaystyle} & -\begin{displaystyle}\frac{\rho_0^m}{\rho^m t^{\frac{1}{2}}}\delta^K_J\end{displaystyle} & 0 & 0\\
 -\begin{displaystyle}\frac{\rho_0^m}{\rho^m t^{\frac{1}{2}}}\delta^K_J\end{displaystyle} & -\delta^K_J & 0 & 0 \\
0 & 0 & -\delta^\Sigma_\Lambda\delta^K_J & 0\\
0 & 0 & 0 & -\delta^K_J
\end{pmatrix}\right|_{\Gamma^+},
\end{align*}
where we note that
\begin{align*}
\biggl(-(2-t)+\frac{\rho_0^m}{\rho^m}\frac{2(t-1)}{t}\biggr)\biggl|_{\Gamma^+} &= 
 \biggl(-(2-t)+ \frac{\rho_0^m}{\rho^m}+\frac{\rho_0^m}{\rho^m}\frac{(t-2)}{t}\biggr)\biggl|_{\Gamma^+} = -\frac{\rho_0^{2m}}{\rho^{2m}t}\biggl|_{\Gamma^+}
\end{align*}
since $2-t=\frac{\rho_0^m}{\rho^m}$ on $\Gamma^+$.
From these expressions and the definition of the inner-product \eqref{Mhdef}, we deduce that
\begin{equation*}
h\Bigl(V, \Bigl(n^{\pm}_{0} B^0 +n^{-}_1 \frac{\rho}{m} B^1 + n^{-}_\Gamma B^\Gamma\Bigr) V\Bigr)\Bigl|_{\Gamma^{-}} = 0
\end{equation*}
and
\begin{equation*}
h\Bigl(V, \Bigl(n^{+}_{0} B^0 +n^{+}_1 \frac{\rho}{m} B^1 + n^{+}_\Gamma B^\Gamma\Bigr) V\Bigr)\Bigl|_{\Gamma^{+}} = -\delta_{IJ}\biggl(\frac{\rho_0^m V^I_0}{\rho^m t^{\frac{1}{2}}} + V^I_1\biggl)\biggl(\frac{\rho_0^m V^J_0}{\rho^m t^{\frac{1}{2}}} + V^J_1\biggl)-\delta_{IJ}\gsl^{\Lambda\Sigma}V^I_\Lambda V^J_\Sigma - \delta_{IJ}V_4^IV_4^J\leq 0.
\end{equation*}
By definition, see \cite[\S 4.3]{Lax:2006}, this implies that the surfaces $\Gamma^{\pm}$ are weakly spacelike for the symmetric hyperbolic system \eqref{MwaveE}. The importance of this
fact is that it will guarantee that any solution of the extended system \eqref{MwaveE} on the extended spacetime \eqref{MMscdef} will also yield
by restriction a solution of the original system 
\eqref{MwaveD} on the region \eqref{Mr02rho} that is uniquely determined by the restriction of the initial data to $\{1\}\times (0,\rho_0)\times \mathbb{S}^2$. From this property and the above arguments, we conclude that the existence 
of solutions to the system of semilinear wave equations \eqref{Mbwave} on the regions of the form \eqref{Mbrb0def} in Minkowski spacetime can
be obtained from solving the initial value problem 
\begin{align} 
B^0 \del{t} U + \frac{\chi\rho}{m} B^1 \del{\rho} U + B^\Gamma \nablasl{\Gamma} U &= \frac{1}{t}\Bc U +\chi F  && \text{in $(0,1) \times T^{1}_{3\rho_0}\times \mathbb{S}^2$,}  \label{MwaveF.1}\\
U &= \mathring{U} && \text{in $\{1\} \times T^{1}_{3\rho_0}\times \mathbb{S}^2$,}  \label{MwaveF.2}
\end{align} 
where the solutions generated this way are independent of the particular form of the initial data $\mathring{U}$ on $(\{1\}\times T^{1}_{3\rho_0}\times\mathbb{S}^2)\setminus (\{1\}\times (0,\rho_0)\times\mathbb{S}^2)$.

We are now almost in a position to apply Theorem \ref{symthm} to establish the existence of solutions to \eqref{MwaveF.1}-\eqref{MwaveF.2}. However, to do so, we must first verify two structural conditions. The first is to show that the inequality 
\begin{equation} 
h(V,\Bc V) \geq \kappa h(V,B^0 V)
\label{MBcbndA} 
\end{equation}
holds where $\lambda$ is chosen so that
\begin{equation}\label{Mkappadef}
\kappa = \lambda -\frac{1}{4}(2+\sqrt{2})>0,
\end{equation}
which we note is compatible with the condition $\lambda \leq 1$ that is needed to ensure that \eqref{Bsczero}
holds.
To see the validity of \eqref{MBcbndA}, we note, with the help of the inequality $|\delta_{IJ}V_0^I V_4^J| \leq \frac{\ep^2}{2} \delta_{IJ}V^I_0 V_0^J + \frac{1}{2\ep^2} \delta_{IJ}V_4^I V_4^J$, $\ep >0$, that

\begin{align}
h(V,\Bc V) &=\delta_{IJ} \biggl(2\Bigl(\lambda-\frac{1}{2}\Bigr)V_0^IV_0^J + V_0^IV_4^J +\Bigl(\lambda-\frac{1}{2}\Bigr)V_4^I V_4^J + \lambda V_1^I V_1^J+ \lambda \gsl^{\Lambda\Sigma}V^I_\Lambda V^J_\Sigma \biggr)
&& \text{(by \eqref{MBcdef}  \&  \eqref{Mhdef})}\notag  \\
&\geq \delta_{IJ}\biggl(\Bigl(2\Bigl(\lambda-\frac{1}{2}\Bigr)-
\frac{\ep^2}{2}\Bigr) V_0^I V_0^J  +\Bigl(\lambda-\frac{1}{2}-\frac{1}{2\ep^2}\Bigr)V_4^I V_4^J + \lambda V_1^I V_1^J+ \lambda \gsl^{\Lambda\Sigma}V^I_\Lambda V^J_\Sigma\biggr).
\label{MBcbndB} 
\end{align}
We then fix $\ep$ by demanding that 
$\frac{1}{2}\Bigl(2\Bigl(\lambda-\frac{1}{2}\bigr)-\frac{\ep^2}{2}\Bigr) = \Bigl(\lambda-\frac{1}{2}-\frac{1}{2\ep^2}\Bigr)$.
Solving this yields
\begin{equation}\label{Mepfix}
\ep^2 = \sqrt{2},
\end{equation}
which in turns gives
\begin{equation} \label{Mepfix2}
\frac{1}{2}\Bigl(2\Bigl(\lambda-\frac{1}{2}\bigr)-\frac{\ep^2}{2}\Bigr) = \Bigl(\lambda-\frac{1}{2}-\frac{1}{2\ep^2}\Bigr) = \lambda -\frac{1}{4}(2+\sqrt{2})
=\kappa.
\end{equation}
Assuming that \eqref{Mkappadef} holds,
we see, after substituting \eqref{Mepfix} and
\eqref{Mepfix2} into \eqref{MBcbndB} and recalling \eqref{MB0def} and \eqref{Mhdef}, that the inequality 
\begin{equation*}
h(V,\Bc V) \geq \kappa \delta_{IJ}\bigl( 2 V_0^I V_0^J  + V_1^I V_1^J+ \gsl^{\Lambda\Sigma}V^I_\Lambda V^J_\Sigma + V_4^I V_4^J\bigr)\geq \kappa h(V,B^0 V)
\end{equation*}
holds for $t\in [0,1]$.

The second structural condition, which is related to the constants
$\mathtt{b}$ and $\tilde{\mathtt{b}}$, involves bounding the size of the matrix\footnote{Note that the spatial derivatives (e.g. $\nablasl{\Lambda} \Bc$, $\del{\rho}\Bc$, $\nablasl{\Lambda}\Csc$, $\del{\rho}\Csc$,
$\nablasl{\Lambda}\bigl( \frac{\rho\chi}{m}B^1\bigr)$, $\nablasl{\Lambda}B^\ell$ and $\del{\rho}B^\ell$ for $\ell=0,2,3$) of all the other $U$-independent matrices appearing in \eqref{MwaveF.1} (see also \eqref{MFrep}) vanish.}
$\del{\rho}\bigl(t\frac{\rho\chi}{m}B^1\bigr)$. 
From the bound
\begin{equation*}
\Bigl|\del{\rho}\Bigl(t\frac{\rho\chi}{m}B^1\Bigr)\Bigr|_{\op}\leq \max_{0\leq t \leq 1}|t B^1(t)|_{\op} \norm{\del{\rho}(\rho\chi))}_{L^\infty\bigl(T^1_{3\rho_0}\bigr)} \frac{1}{m},
\end{equation*}
it is clear, given any $\sigma > 0$, that there exists a positive integer $m=m(\sigma)$ such that
\begin{equation} \label{MB1bnd}
\Bigl|\del{\rho}\Bigl(t\frac{\rho\chi}{m}B^1\Bigr)\Bigr|_{\op} < \sigma \quad \text{ in $(0,1) \times  T^{1}_{3\rho_0}\times \mathbb{S}^2$.}
\end{equation}

\subsubsection{Global existence}
Having established that \eqref{MwaveF.1} is symmetric hyperbolic, we can appeal to the Cauchy stability
property satisfied by symmetric hyperbolic systems to
conclude, for any given $t_0\in (0,1)$, the existence of a unique solution 
\begin{equation*}
U\in C^0\bigl((t_0,1],H^k( T^{1}_{3\rho_0}\times \mathbb{S}^2)\bigr)\cap L^\infty\bigl((t_0,1],H^k( T^{1}_{3\rho_0}\times \mathbb{S}^2)\bigr)\cap C^1\bigl((t_0,1],H^{k-1}(T^{1}_{3\rho_0}\times \mathbb{S}^2)\bigr)
\end{equation*}
to \eqref{MwaveF.1}  provided that
 $k\in \Zbb_{>3/2}$ and the initial
data $U|_{t=1} = \mathring{U} \in H^k( T^{1}_{3\rho_0}\times \mathbb{S}^2)$ is chosen small enough. Furthermore, by
standard results, this solution will satisfy an energy estimate
of the form
\begin{equation*}
\norm{U(t)}_{H^k}^2+ \int_{t}^1 \frac{1}{\tau} \norm{U(\tau)}_{H^k}^2\, d\tau   \leq C\bigl( \norm{U}_{L^\infty((t_0,0],H^k)}\bigr) \norm{\mathring{U}}_{H^k}^2, \qquad t_0<t\leq 1,
\end{equation*}
and the norm $\norm{U(t_0)}_{H^k}$
can be made as small as we like by choosing the initial
data $\mathring{U}$ at $t=1$
suitably small.

To continue this solution from $t=t_0$ to
$t=0$, we
now assume that $k\in \Zbb_{>9/2}$ and choose the initial data
 $\mathring{U}$ at $t=1$ small enough so that $\norm{U(t_0)}_{H^k}$ and $t_0$ can be taken to be sufficiently small. If we further assume that
 \begin{equation} \label{lambdafix}
\lambda \in \Bigl(\frac{1}{4}(2+\sqrt{2}),1\Bigr],
\end{equation}
 then from \eqref{MB0def}-\eqref{MBcdef}, \eqref{MFrep}, \eqref{Bsczero}, \eqref{MBcbndA}, \eqref{MB1bnd}  and the simple time transformation $t \mapsto -t$, it is not difficult to verify
that \eqref{MwaveF.1} will satisfy all the assumptions from Section \ref{coeffassump} with $\Pbb = \id$ on the region $(-t_0,0)\times T^1_{3\rho_0}\times\mathbb{S}^2$. Moreover, by \eqref{MB1bnd}, we can always choose the integer $m$ large enough
so that the constants $\beta_1$ and  $\mathtt{b} = \tilde{\mathtt{b}}$ from Theorem \ref{symthm} can be made as small as we like while the constants $\lambda_1$ and $\alpha$
vanish.
Thus, for any $\upsilon>0$, we can apply Theorem \ref{symthm} to
\eqref{MwaveF.1} with initial data given by $U(t_0)$ at
$t=t_0$ to
obtain the existence of a unique solution
\begin{equation*}
U\in C^0\bigl((0,t_0],H^k( T^{1}_{3\rho_0}\times \mathbb{S}^2)\bigr)\cap L^\infty\bigl((0,t_0],H^k( T^{1}_{3\rho_0}\times \mathbb{S}^2)\bigr)\cap C^1\bigl((0,t_0],H^{k-1}(T^{1}_{3\rho_0}\times \mathbb{S}^2)\bigr)
\end{equation*}
that satisfies the energy
\begin{equation*}
\norm{U(t)}_{H^k}^2+ \int_{t}^1 \frac{1}{\tau} \norm{U(\tau)}_{H^k}^2\, d\tau   \leq C\bigl( \norm{U}_{L^\infty((1,t_0],H^k)}\bigr) \norm{U(t_0)}_{H^k}^2
\end{equation*}
and decay
\begin{equation*}
\norm{U(t)}_{H^{k-1}} \lesssim 
t^{\kappa-\upsilon}
\end{equation*}
estimates for $0<t\leq t_0$. This establishes the existence of a unique solution 
of the IVP \eqref{MwaveF.1}-\eqref{MwaveF.2}, 
and completes the proof of the following theorem.

\begin{thm} \label{MGlobalthm}
Suppose $k\in \Zbb_{>9/2}$, $\rho_0>0$, $\upsilon>0$, $\lambda \in \bigl(\frac{1}{4}(2+\sqrt{2}),1\bigr]$ and $\kappa =  \lambda -\frac{1}{4}(2+\sqrt{2})$. Then there exist  $m\in \Zbb_{\geq 1}$ and $\delta >0$
such that if $\mathring{U} \in H^k(T^{1}_{3\rho_0}\times \mathbb{S}^2)$ is chosen so that
$\norm{\mathring{U}}_{H^k} < \delta$,
then there exists a unique solution 
\begin{equation*}
U\in C^0\bigl((0,1],H^k( T^{1}_{3\rho_0}\times \mathbb{S}^2)\bigr)\cap L^\infty\bigl((0,1],H^k( T^{1}_{3\rho_0}\times \mathbb{S}^2)\bigr)\cap C^1\bigl((0,1],H^{k-1}(T^{1}_{3\rho_0}\times \mathbb{S}^2)\bigr)
\end{equation*}
of the IVP \eqref{MwaveF.1}-\eqref{MwaveF.2} that satisfies the energy
\begin{equation*}
\norm{U(t)}_{H^k}^2+ \int_{t}^1 \frac{1}{\tau} \norm{U(\tau)}_{H^k}^2\, d\tau   \leq C\bigl( \norm{U}_{L^\infty((1,0],H^k)}\bigr) \norm{\mathring{U}}_{H^k}^2
\end{equation*}
and decay
\begin{equation*}
\norm{U(t)}_{H^{k-1}} \lesssim 
t^{\kappa-\upsilon}
\end{equation*}
estimates for $0<t\leq 1$.
\end{thm}

\begin{rem} \label{regrem}
$\;$
\begin{enumerate}[(i)]
\item The regularity assumption on the initial data
in Theorem \ref{MGlobalthm} is not optimal since the proof does not take into account the semilinearity of the wave equations \eqref{MwaveF.1}. Although we will not do this here, regularity improvements can be obtained by establishing a version of Theorem \ref{symthm} that is adapted to semilinear equations.
\item Due to the appearance of the derivative $\rho\del{\rho}$ in the evolution equation \eqref{MwaveF.1}, it is not difficult to adapt the proof of Theorem \ref{MGlobalthm} to establish the existence and uniqueness of solutions in the weighted spaces $\Hsc^\alpha_k$ defined in \cite{ChruscielLengard:2005}. This type of generalization allows for initial data and solutions that have polyhomogeneous behavior in $\rho$ as $\rho\searrow 0$. A proof of this adaptation of Theorem \ref{MGlobalthm} will be given in a separate article.
\end{enumerate}
\end{rem}

\subsection{Wave equations near spatial infinity on Schwarzschild spacetimes}

\subsubsection{The cylinder at spatial infinity}
In this section, we generalize the global existence results for semilinear
wave equations on Miknowksi spacetime from the previous section to Schwarzschild spacetimes. We do so by employing the cylinder at spatial infinity
representation\footnote{More precisely, the future half of the cylinder at
infinity representation.} for a Schwarzschild spacetime of mass $\mu>0$ from \cite{FrauendienerHennig:2017}
that is given by the spacetime consisting of the manifold
\begin{equation*}
M = (0,1)\times (0,1)\times \mathbb{S}^2
\end{equation*}
and the (conformal) Lorentzian metric
\begin{equation}
g=\frac{1}{r}A (dr\otimes dt+dt\otimes dr) +\frac{t}{r^2}A\left(2-tA \right)dr\otimes dr+\gsl\label{met}
\end{equation}
where $(t,r)\in (0,1)\times (0,1)$,
\begin{equation}\label{defA}
A=\frac{(1+r)^3(1-rt)}{(1-r)(1+rt)^3},
\end{equation}
and as above, $\gsl$ is the canonical
metric on the 2-sphere $\mathbb{S}^2$, see \eqref{gsldef}.
As shown in \cite{FrauendienerHennig:2017}, the metric \eqref{met} is
conformally\footnote{This is easily seen using the change of coordinates
\begin{equation*}
 (\tilde{t},\tilde{r},\theta,\phi ) \longmapsto (x^{\mu}) =\left(\frac{\mu\left(1-t-r^2t(1-t)\right)}{2rt}-4\mu\ln\left(\frac{t^{\frac{1}{2}}(1-r)}{1-rt} \right),\frac{\mu}{2rt},\theta,\phi  \right)
 \end{equation*}
that will bring the metric $\gt$, defined by \eqref{Schgtdef}, into the
easily recognized form 
\begin{equation*}
\tilde{g}=-\left(\frac{1-\frac{\mu}{2\tilde{r}}}{1+\frac{\mu}{2\tilde{r}}}\right)^2d\tilde{t}\otimes d\tilde{t} + \left(1+\frac{\mu}{2\tilde{r}}\right)^4\left(d\tilde{r} \otimes  d\tilde{r} + \gsl\right)
\end{equation*}
of the Schwarzschild metric in isotropic coordinates.}
related to the Schwarzschild metric $\gt$ of mass $\mu>0$ according to 
\begin{equation} \label{Schgtdef}
 \tilde{g}=\Omega^2g,
 \end{equation}
where 
\begin{equation}\label{omegaA}
 \Omega=\frac{\mu(1+rt)^2}{2rt},
\end{equation}
while the boundary components 
\begin{equation*} 
i^0 = [0,1]\times \{0\} \times \mathbb{S}^2 \AND
\Isc^{+} = \{0\}\times (0,1) \times \mathbb{S}^2
\end{equation*}
of $M$ correspond to spatial infinity and future null-infinity, respectively,
and the boundary component 
\begin{equation*}
\Sigma = \{1\}\times (0,1)\times \mathbb{S}^2
\end{equation*}
defines a spacelike hypersurface in $M$. For other applications
of the cylinder at spatial infinity construction in Schwartzschild spacetimes
to linear wave equations,
see the articles \cite{FrauendienerHennig:2017,FrauendienerHennig:2018,ValienteKroon:2007,ValienteKroon:2009}.

For use below, we note that the Ricci scalar of $\gt$
vanishes, that is,
\begin{equation} \label{SchRtdef}
\Rt = 0.
\end{equation}
This follows because $\gt$ coincides with a Schwarzschild metric and hence, has vanishing Ricci curvature.
It is also not difficult to verify via a straightforward calculation that the Ricci scalar
of the metric $g$, defined by \eqref{met}, is given by
\begin{equation} \label{SchRdef}
R = \frac{24 rt}{(1+rt)^2}.
\end{equation}

\subsubsection{Semilinear wave equations}
In the following, we consider the same class of semilinear wave
equations as we did in the previous section, namely systems of wave equations of the
type
\begin{equation} \label{SchMtwave}
\gt^{\alpha\beta}\nablat_\alpha \nablat_\beta \ut^K  = q^K_{IJ}(\ut^L)\gt^{\alpha\beta}
\nablat_\alpha \ut^I \nablat_\beta \ut^J 
\end{equation}
where $q^K_{IJ}\in C^\infty(\Rbb^N)$, but $\gt$ is now the Schwarzschild metric defined by \eqref{Schgtdef} and $\nablat$ is the Levi-Civita connection of $\gt$. We will solve these type of wave equations on 
domains of the form
\begin{equation} \label{SchMr0def} 
\Mcal_{r_0} = (0,1)\times (0,r_0) \times \mathbb{S}^2, \qquad 0<r_0<1,
\end{equation}
that define a ``neighborhood'' of spatial infinity in the Schwarzschild 
spacetime of mass $\mu>0$. 

\begin{rem}$\;$
\begin{enumerate}[(i)]
\item The existence results we establish in this section are not completely new in the sense that global existence results, under a small initial
data assumption, for solutions to scalar semilinear wave equations
of the form \eqref{SchMtwave} on Schwarzschild
spacetimes, and more generally, on Kerr spacetimes, has been previously established in \cite{Luk:2013}, see, in particular, Theorem 1.5 from \cite{Luk:2013}. With that said, the method we use is new, and even in the scalar case, we
believe it brings a valuable 
new perspective to global existence problems for systems of nonlinear wave equations on Schwarzschild spacetimes. Indeed, comparing the analysis below
to that carried out above for wave equations on Minkowski space, it is apparent that establishing global existence on Schwarzschild spacetimes using the Fuchsian approach is no more difficult than on Minkowski spacetime, and in either case, the global existence proofs are relatively straightforward.
Furthermore, the results of this section are also of interest because they
demonstrates the utility of Friedrich's cylinder at spatial infinity construction for solving nonlinear wave equations near spatial infinity on Schwarzschild spacetimes. 
\item The analysis carried out here for wave equations of the form \eqref{SchMtwave} on Schwarzschild spacetimes can be generalized to systems of quasilinear wave equations satisfying
conditions that can be viewed as a natural adaptation of the null condition to Schwarzschild spacetimes. We will report on this work in a separate article. 
\end{enumerate}
\end{rem}

Using \eqref{SchRtdef}, we can write the sytem of wave equations \eqref{SchMtwave}
as 
\begin{equation}
\tilde{g}^{\alpha \beta}\tilde{\nabla}_{\alpha}\tilde{\nabla}_{ \beta}\tilde{u}^K-\frac{\tilde{R}}{6}\tilde{u}^K=\tilde{f}^K\label{weq0}
\end{equation}
where 
\begin{equation*}
\tilde{f}^K=q^K_{IJ}(\tilde{u}^L)\gt^{\alpha\beta}\tilde{\nabla}_{\alpha} \tilde{u}^I\tilde{\nabla}_{\beta}\tilde{u}^J.
\end{equation*}
From \eqref{omegaA}, \eqref{SchRdef} and the formulas \eqref{gtrans}-\eqref{wavetransC} and \eqref{ftransB}-\eqref{ftransC} from Appendix \ref{ctrans}, we then see that under the conformal transformation \eqref{Schgtdef} the
wave equations \eqref{weq0} transforms into
\begin{equation}
g^{\alpha \beta}\nabla_{\alpha}\nabla_{ \beta}u^K-\frac{R}{6}u^K=f^K  \label{weq}
\end{equation}
where $\nabla$ is the Levi-Civita connection of the metric $g$, now
defined by \eqref{met},
\begin{equation*}\label{raltionuutilde}
\tilde{u}^K=\frac{2rt}{\mu(1+rt)^2}u^K
\end{equation*}
and
\begin{align*}
f^K &= q^K_{IJ}\biggl(\frac{2rt}{\mu(1+rt)^2}u^L\biggr)\biggl(
\frac{2rt}{\mu(1+rt)^2} g^{\alpha\beta}\nabla_\alpha  u^I \nabla_\beta u^J 
 \\
 &\qquad +2 g^{\alpha\beta}\nabla_\alpha \biggl(\frac{2rt}{\mu(1+rt)^2}\biggr)
\nabla_\beta u^{(I} u^{J)}
+ \frac{\mu(1+rt)^2}{2rt}g^{\alpha\beta}\nabla_\alpha
\biggl(\frac{2rt}{\mu(1+rt)^2}\biggr)
\nabla_\beta\biggl(\frac{2rt}{\mu(1+rt)^2}\biggr) u^I u^J \biggr).
\end{align*}
It is then not difficult to verify, via a straightforward calculation using
the metric \eqref{met}, that the system of wave equations \eqref{weq}, after multiplication
by $A$, can be expressed as
\begin{equation}
-t(2-tA)\partial_t^2u+2r\partial_{r}\partial_tu^K -2\mathscr{A}\partial_tu^K +A\gsl ^{\Lambda \Sigma} \nablasl{\Lambda} \nablasl{\Sigma}u^K -\frac{4r tA}{(1+r t)^2}u^K= A f^K \label{weq2}
\end{equation}
where
\begin{equation} 
\mathscr{A}=1-\frac{1-2rt}{1-\left(rt \right)^2 } tA,\label{newA}
\end{equation}
\begin{equation*}
\begin{split}
Af=&q^K_{IJ}\left(\frac{2rt}{\mu(1+rt)^2}u^L\right)\left(-\frac{2t^2r(2-tA)}{\mu (1+rt )^2}\del{t}u^I\del{t}u^J+ \frac{4rt}{\mu (1+rt )^2}r\partial_{r }u^{(I}\partial_tu^{J)}+ \frac{2rt A}{\mu (1+rt )^2}\gsl^{\Lambda\Sigma}\nablasl{\Lambda} u^I \nablasl{\Sigma}u^J 
+\right.
\\
&\left.\frac{4r At(t-1)\left[1-r-t(r+6r^2+r^3)+t^2(r^4-r^3)\right]}{\mu(1+r)^3(1+rt)^3} u^{(I} \partial_t u^{J)} +\frac{ 4r(1-rt )}{\mu(1+rt)^3}u^{(I}r\partial_{r}u^{J)}+\frac{2t r (1- rt )^2A}{\mu (1+rt )^4}u^I u^J\right),
\end{split}
\end{equation*}
and, as before, $\nablasl{}$ is the Levi-Civita connection of the metric \eqref{gsldef}
on $\mathbb{S}^2$.

Next, we introduce new variables $U^J_0$, $U^J_1$, $U^J_2$, $U^J_3$ and $U^J_4$ via the relations
\begin{equation*} 
\frac{1}{2} t^{-\lambda-\frac{1}{2}}U^J_0+t^{-\lambda}U^J_1=\partial_tu^J,\quad
t^{-\lambda}U^J_1=r\partial_ru^J, \quad
t^{-\lambda}U^J_\Lambda=\nablasl{\Lambda}u^J \AND
t^{-\lambda+\frac{1}{2}}U^J_4=u^J,
\end{equation*}
where $\lambda\in \Rbb$ is a constant that will be fixed below.
A routine calculation shows that we can use these variables 
to write the 
wave equation \eqref{weq2} in first order form as
\begin{equation}
B^0\partial_t U+B^1 r \partial_{r}U +B^{\Gamma}\nablasl{\Gamma}U=\frac{1}{t}\mathcal{B}U+F
\label{SchwaveA}
\end{equation} 
where
\begin{align}
U&=\begin{pmatrix}
U^J_0\\
U^J_1\\
U^J_\Sigma\\
U^J_4
\end{pmatrix}, \label{SchwaveAvars}\\[0.3cm]
B^0&=\begin{pmatrix}
\frac{1}{2}\delta^K_J &0&0&0\\
0&-\begin{displaystyle}\frac{2 t-2-t A}{\frac{1}{2}(2-tA)}\delta^K_J\end{displaystyle}&0&0\\
0&0&\begin{displaystyle}\frac{A\delta^\Sigma_\Lambda}{\frac{1}{2}(2-tA)}\delta^K_J
\end{displaystyle}&0\\
0&0&0&\delta^K_J
\end{pmatrix}, \label{SchwaveAB0}\\[0.3cm]
B^{1}&=\begin{pmatrix}
\begin{displaystyle}\frac{2t-2- t^2 A}{2t(2-tA)}\delta^K_J
\end{displaystyle}&\begin{displaystyle}\frac{2 t-2-t^2 A}{t^{\frac{1}{2}}(2-tA)}\delta^K_J\end{displaystyle}&0&0\\
\begin{displaystyle}\frac{2 t-2-t^2 A}{t^{\frac{1}{2}}(2-tA)}
\delta^K_J\end{displaystyle}&\begin{displaystyle}\frac{2 t-2-t^2 A}{\frac{1}{2}(2-tA)}\delta^K_J\end{displaystyle}&0&0\\
0&0&-\begin{displaystyle}\frac{ A\delta^\Sigma_\Lambda}{\frac{1}{2} (2-tA)}\delta^K_J\end{displaystyle}&0\\
0&0&0&0
\end{pmatrix}, \label{SchwaveAB1}\\[0.3cm]
B^\Gamma&=\begin{pmatrix}
0&0&-\begin{displaystyle}\frac{A\gsl^{\Gamma\Sigma}}{t^{\frac{1}{2}}(2-tA)}\delta^K_J\end{displaystyle}&0\\
0&0&0&0\\
-\begin{displaystyle}\frac{A\delta^\Gamma_\Lambda}{t^{\frac{1}{2}}(2-tA)}\delta^K_J\end{displaystyle}&0&0&0\\
0&0&0&0\\
\end{pmatrix},\label{SchwaveABGamma}\\[0.3cm]
\mathcal{B}&=\begin{pmatrix}
\begin{displaystyle}\frac{\frac{1}{2}\left(\lambda+\frac{1}{2}\right)(2-tA)-
\mathscr{A} }{2-tA}\delta^K_J\end{displaystyle}&0&0&0\\
0&-\begin{displaystyle}\lambda\frac{2 t-2- t A}{\frac{1}{2} (2-tA)}\delta^K_J\end{displaystyle}&0&0\\
0&0&\begin{displaystyle}\lambda \frac{A\delta^\Sigma_\Lambda}{\frac{1}{2} (2-tA)}\delta^K_J\end{displaystyle}&0\\
\begin{displaystyle}
\frac{1}{2}\delta^K_J
\end{displaystyle}&0&0&\biggl(\lambda-\frac{1}{2}\biggr)\delta^K_J\\
\end{pmatrix}, \label{SchwaveABcal}
\intertext{and}
F&=\begin{pmatrix}
-\begin{displaystyle}\frac{t^{\lambda-\frac{1}{2}}A}{(2-tA)}f^K-\frac{2 \mathscr{A}  U^K_1}{t^{\frac{1}{2}}(2-tA)}-\frac{4rtAU^K_4}{(2-tA)(1+rt)^2}\end{displaystyle}\\
0\\
0\\
0\\
\begin{displaystyle}\frac{U_1 }{t^{\frac{1}{2}}}\end{displaystyle} \end{pmatrix}.
\label{SchwaveAF}
\end{align}
We further observe that $F$ can be expressed as
\begin{equation} \label{SchwaveAFa}
F=\begin{pmatrix}\begin{displaystyle} - \frac{4rtA}{(2-tA)(1+rt)^2}U^K_4\end{displaystyle}\\ 0\\0\\0\end{pmatrix}+\frac{1}{t}( t^{\frac{1}{2}}\Csc+\Bsc)U
\end{equation}
where 
\begin{equation} \label{SchwaveACsc}
\Csc = \begin{pmatrix}0 & \begin{displaystyle}-\frac{2 \mathscr{A} }{(2-tA)}\delta^K_J\end{displaystyle} & 0& 0\\
0 & 0& 0& 0\\
0 & 0& 0& 0\\
0 & \delta^K_J & 0& 0
\end{pmatrix}
\end{equation}
and
\begin{equation*} 
\mathscr{B}= \begin{pmatrix}
\mathscr{B}^{0K}_J & \mathscr{B}^{1K}_J & \mathscr{B}^{\Sigma K}_J & \mathscr{B}^{4K}_J\\
0&0&0&0\\
0&0&0&0\\
0&0&0&0\\
\end{pmatrix}
\end{equation*}
with
\begin{align*}
\mathscr{B}^{0K}_J=&q^K_{PQ}\left(\frac{2rt^{\frac{3}{2}-\lambda}U^L_4}{\mu(1+rt)^2} \right)\delta^{(P}_I\delta^{Q)}_J \left(-\frac{\frac{1}{2}rt^{-\lambda+\frac{3}{2}}}{\mu(1+rt)^2}U^I_0-\frac{2rt^{-\lambda+1}(1-2t-t^2A)}{\mu(1+rt)^2(2-tA)^2}U^I_1\right),\\
\mathscr{B}^{1K}_J=&q^K_{PQ}\left(\frac{2rt^{\frac{3}{2}-\lambda}U^L_4}{\mu(1+rt)^2} \right)\delta^{(P}_I\delta^{Q)}_J\left( - \frac{2t^{\frac{3}{2}-\lambda}r(2-2t+t^2A)}{\mu(2-tA)(1+tr)^2}U^I_1\right.\\ 
&\hspace{0.5cm}\left. 
-\frac{4rt^{1-\lambda}[(1-rt)(1+r)^3+At(t-1)(1-r-t(r+6r^2+r^3)+t^2(r^4-r^3))]}{\mu(2-tA)(1+r)^3(1+tr)^3}U^I_4
\right),\\
\mathscr{B}^{\Sigma K}_J=&q^K_{PQ}\left(\frac{2rt^{\frac{3}{2}-\lambda}U^L_4}{\mu(1+rt)^2} \right)\delta^{(P}_I\delta^{Q)}_J\left(-\frac{2rAt^{\frac{3}{2}-\lambda}}{\mu(2-tA)(1+rt)^2}\gsl^{\Sigma\Lambda}U^I_\Lambda \right)
\intertext{and}
\mathscr{B}^{4K}_J=&q^K_{PQ}\left(\frac{2rt^{\frac{3}{2}-\lambda}U^L_4}{\mu(1+rt)^2} \right)\delta^{(P}_I\delta^{Q)}_J\left(-\frac{2t^{\frac{5}{2}-\lambda}r(1-tr)^2A}{\mu(2-tA)(1+rt)^4}U^I_4 \right. \\
&\hspace{3cm}\left. 
-\frac{4rAt^{\frac{3}{2}-\lambda}(t-1)[1-r-t(r+6r^2+r^3)+t^2(r^4-r^3)]}{\mu(2-tA)(1+r)^3(1+rt)^3}U^I_0\right).
\end{align*}

From \eqref{defA} and the above formulas, it is clear that
\begin{equation} \label{SchwBsczero}
\Bsc = \Ord(U) \quad \text{for $\lambda \leq 1$.}
\end{equation}
Moreover, we see, using the change of radial coordinate \eqref{Mrhodef}, that
we can express \eqref{SchwaveA} as
\begin{equation}
B^0\partial_t U+ \frac{\rho}{m}B^1 \partial_{\rho}U +B^{\Gamma}\nablasl{\Gamma}U=\frac{1}{t}\mathcal{B}U+F
\label{SchwaveB}
\end{equation} 
where now any $r$ appearing in the coefficients is
replaced using \eqref{Mrhodef}.

\subsubsection{The extended system}
Proceeding
in a similar fashion as above for wave equations
on Minkowski spacetime, we consider an extended version
of \eqref{SchwaveB} given by
\begin{equation}
\Bt^0\partial_t U+ \frac{\chi\rho}{m}\Bt^1 \partial_{\rho}U +\Bt^{\Gamma}\nablasl{\Gamma}U=\frac{1}{t}\tilde{\Bc}U+
\chi F
\label{SchwaveC}
\end{equation}
where $\chi$ is the cut-off function defined above by
\eqref{chirho0def},
\begin{align}
\Bt^\mu &= B^\mu_*+\chi(B^\mu-B_*^\mu), \qquad \mu=0,1,2,3,4,
\label{Btdef}\\
\tilde{\Bc}&= \Bc_*+ \chi(\Bc-\Bc_*),
\label{Bctdef}
\end{align}
and we are employing the notation
\begin{equation} \label{*def}
(\cdot)_* = (\cdot)|_{\rho=0}.
\end{equation}
Assuming that
\begin{equation*}
0< \rho_0< \frac{1}{3},
\end{equation*}
$|\rho|<3\rho_0$ implies, via \eqref{Mrhodef}, that
$|r|<1$ and we see from the definitions \eqref{Btdef}, \eqref{Bctdef},
and the formulas \eqref{chirho0def}, \eqref{defA}, \eqref{newA},
and \eqref{SchwaveAB0}-\eqref{SchwaveAF} that the
extended system \eqref{SchwaveC} is well-defined on the extended spacetime $(0,1)\times T^1_{3\rho_0} \times \mathbb{S}^2$ (see \eqref{MMscdef}) and agrees with the original system \eqref{SchwaveB} when restricted to the region
$\Mcal_{\rho_0}=(0,1)\times (0,\rho_0)\times \mathbb{S}^2$.

Assuming that
\begin{equation*}
m\in \Nbb_{\geq 2} \AND 0<\eta < 1,
\end{equation*}
it then follows from \eqref{Mrhodef}, \eqref{defA},
\eqref{newA}, \eqref{SchwaveAB0}, \eqref{SchwaveABcal}
and Taylor's Theorem that there exists a constant
\begin{equation} \label{Kdef}
C=C(m,\eta)>0,
\end{equation}
such that
\begin{align}
|B^0-B_*^0|&\leq C|\rho|^2, \label{pertB0}\\
|\del{\rho}B^0|&\leq C|\rho|, \label{delB0} \\
|\Bc-\Bc_*|&\leq C|\rho|^2 \label{pertBc}
\intertext{and}
|\del{\rho}\Bc|&\leq C|\rho| \label{delBc}
\end{align}
for all $(t,\rho,x^\Lambda)\in (0,1)\times (-\eta,\eta)\times \mathbb{S}^2$. 
Fixing
\begin{equation*}
\sigma >0,
\end{equation*}
we can, by \eqref{Btdef}, \eqref{Bctdef}, \eqref{pertB0} and \eqref{pertBc}, ensure that
\begin{equation} \label{pertBt0}
|\Bt^0-B^0_*|\leq |\chi||B^0-B^0_*|\leq |B^0-B^0_*| < \sigma |\rho|
< \sigma 
\end{equation}
and
\begin{equation}\label{pertBct}
|\tilde{\Bc}-\Bc_*| \leq |\chi| |\Bc-\Bc_*| \leq |\Bc-\Bc_*|< \sigma |\rho| < \sigma
\end{equation}
for all $(t,\rho,x^\Lambda)\in (0,1) \times (-3\rho_0,3\rho_0)\times \mathbb{S}^2$
by choosing $\rho_0$ so that
\begin{equation} \label{rho0fix}
    0<\rho_0 < \min\biggl\{\frac{\eta}{3},\frac{\sigma}{3 C(m,\eta)}\biggr\}.
\end{equation}
Moreover, evaluating evaluating \eqref{SchwaveAB0} and \eqref{SchwaveABcal}
at $\rho=0$ yields, we find, with the help of \eqref{Mrhodef}, \eqref{defA} and \eqref{newA}, that
\begin{align}
B^0_*&=\begin{pmatrix}
\begin{displaystyle}\frac{1}{2}\delta^K_J\end{displaystyle} &0&0&0\\
0&2\delta^K_J&0&0\\
0&0&\begin{displaystyle}\frac{2}{2-t}\delta^\Sigma_\Lambda\delta^K_J
\end{displaystyle}&0\\
0&0&0&\delta^K_J
\end{pmatrix} \label{SchwaveAB0*}
\intertext{and}
\mathcal{B}_*&=\begin{pmatrix}
\begin{displaystyle}\frac{\frac{1}{2}\left(\lambda+\frac{1}{2}\right)(2-t)-
(1-t)}{2-t}\end{displaystyle}\delta^K_J&0&0&0\\
0&2\lambda\delta^K_J&0&0\\
0&0&\lambda\begin{displaystyle} \frac{2}{2-t}\delta^\Sigma_\Lambda\delta^K_J\end{displaystyle}&0\\
\begin{displaystyle}
\frac{1}{2}\delta^K_J
\end{displaystyle}&0&0&\biggl(\lambda-\frac{1}{2}\biggr)\delta^K_J\\
\end{pmatrix}. \label{SchwaveABcal*}
\end{align}

From \eqref{Mhdef} and \eqref{SchwaveAB0*}, we then have
\begin{equation*}
h(V,B^0_*V)\geq \frac{1}{2} h(V,V)
\end{equation*}
for all $V$ of the form \eqref{Vdef}. By choosing $\sigma>0$
sufficiently small, we deduce from the above inequality and the estimate \eqref{pertBct} that
\begin{equation} \label{Bt0bnd}
h(V,\Bt^0V)\geq \frac{1}{4} h(V,V)
\end{equation}
on $(0,1) \times T^1_{3\rho_0}\times \mathbb{S}^2$. From 
this inequality and the obvious symmetry of the $\Bt^\mu$, $\mu=0,1,2,3,4$, with respect to the inner-product \eqref{Mhdef}, we conclude that 
the extended system \eqref{SchwaveC} is symmetric hyperbolic. Furthermore,
decomposing the boundary of $\Mcal_{\rho_0}=(0,1)\times (0,\rho_0)\times \mathbb{S}^2$ as
\begin{equation*}
    \Mcal_{\rho_0} = \Sigma^- \cup \Sigma^+ \cup \Gamma^{-} \cup \Gamma^{+},
\end{equation*}
where
\begin{gather*}
\Sigma^{-} = \{1\}  \times (0,\rho_0)\times \mathbb{S}^2,\quad
\Sigma^{+} = \{0\}  \times (0,\rho_0)\times \mathbb{S}^2,\\
\Gamma^{-} = [0,1]  \times \{0\}\times \mathbb{S}^2
\AND
\Gamma^{+} = [0,1]  \times \{\rho_0\} \times \mathbb{S}^2,
\end{gather*}
it is clear that
\begin{equation*}
    n^{\pm} = \pm d\rho
\end{equation*}
define outward pointing co-normals to $\Gamma^{\pm}$,
\begin{align*}
\Bigl(n^{-}_0 \Bt^0+n_1^{-}\frac{\chi\rho}{m}\Bt^1 +
n_\Gamma^{-}\Bt^\Gamma \Bigr)\Bigl|_{\Gamma^{-}}&= 0
\intertext{and}
\Bigl(n^{+}_0 \Bt^0+n_1^{+}\frac{\chi\rho}{m}\Bt^1 +
n_\Gamma^{+}\Bt^\Gamma \Bigr)\Bigl|_{\Gamma^{+}}&\oset{\eqref{Btdef}}{=}
\frac{\rho_0}{m}B^1\bigl|_{\Gamma^{+}}.
\end{align*}
From these expression and \eqref{SchwaveAB1}, we deduce that
\begin{equation*}
h\Bigl(V,\Bigl(n^{-}_0 \Bt^0+n_1^{-}\frac{\chi\rho}{m}\Bt^i +
n_\Gamma^{-}\Bt^\Gamma \Bigr)V\Bigr)\Bigl|_{\Gamma^{-}} = 0
\end{equation*}
and
\begin{align*}
h\Bigl(V,\Bigl(n^{+}_0
&\Bt^0+n_1^{+}\frac{\chi\rho}{m}\Bt^1 +
n_\Gamma^{+}\Bt^\Gamma \Bigr)V\Bigr)\Bigl|_{\Gamma^{+}}  =
\frac{\rho_0}{m}h(V,B^1 V)\bigl|_{\Gamma^{+}}  \\
&=\frac{\rho_0}{m}\biggl[ \frac{2t-2-t^2A}{2-tA}\delta_{IJ}\biggl(\frac{1}{\sqrt{2}} \frac{V^I_0}{t^{\frac{1}{2}}}+\sqrt{2}V^I_1\biggr)
\biggl(\frac{1}{\sqrt{2}} \frac{V^J_0}{t^{\frac{1}{2}}}+\sqrt{2}V^J_1\biggr)- \frac{2 A}{2-t A}\gsl^{\Lambda\Sigma}
\delta_{IJ}V^I_\Lambda V^J_\Sigma\biggr]\biggl|_{\Gamma^{+}}  \leq 0.
\end{align*}
By definition, see \cite[\S 4.3]{Lax:2006}, this shows that the surfaces $\Gamma^{\pm}$ are weakly spacelike for the extended system \eqref{SchwaveC}, and hence that any solution of \eqref{SchwaveC} on  the extended
spacetime $(0,1)\times T^1_{3\rho_0}\times \mathbb{S}^2$ will determine a
solution of the original system 
\eqref{SchwaveB} on the region $(0,1)\times (0,\rho_0)\times \mathbb{S}^2$ via restriction that is uniquely determined by initial data on $\{1\}\times (0,\rho_0)\times \mathbb{S}^2$. We therefore conclude that solutions to the system of semilinear wave equations \eqref{SchMtwave} on the regions of the form \eqref{SchMr0def} in a Schwarzschild spacetime can
be obtained from solving the initial value problem 
\begin{align} 
\Bt^0 \del{t} U + \frac{\chi\rho}{m} \Bt^1 \del{\rho} U + \Bt^\Gamma \nablasl{\Gamma} U &= \frac{1}{t}\tilde{\Bc} U +\chi F  && \text{in $(0,1) \times T^{1}_{3\rho_0}\times \mathbb{S}^2$,}  \label{SchwaveD.1}\\
U &= \mathring{U} && \text{in $\{1\} \times T^{1}_{3\rho_0}\times \mathbb{S}^2$,}  \label{SchwaveD.2}
\end{align} 
where the solution to \eqref{SchMtwave} generated this way are independent of the particular form of the initial data $\mathring{U}$ on the region $(\{1\}\times T^{1}_{3\rho_0}\times\mathbb{S}^2)\setminus (\{1\}\times (0,\rho_0)\times\mathbb{S}^2)$.

We now want to conclude existence of solutions to the IVP
\eqref{SchwaveD.1}-\eqref{SchwaveD.2} via an application of Theorem \ref{symthm}.
However, in order to do this, we must first verify a number of structural conditions. We proceed by noting from
\eqref{SchwaveABcal*} that
\begin{align*}
h(V,\Bc_*V) =\delta_{IJ}\biggl(\biggl(\frac{1}{2}\biggl(\lambda+\frac{1}{2}\biggr)-\frac{1-t}{2-t}\biggr)V^I_0 V^J_0 + \frac{1}{2}V^I_0V^J_4
+2\lambda V^I_1 V^J_1+\frac{2\lambda}{2-t}\gsl^{\Lambda\Sigma}V^I_\Lambda V^J_\Sigma 
+\biggl(\lambda-\frac{1}{2}\biggr)V_4^I V_4^J\biggr).
\end{align*}
From this and the inequality $\bigl|\frac{1}{2}\delta_{IJ}V^I_0 V^J_4\bigr|\leq \frac{\ep^2}{2}\delta_{IJ}\frac{V^I_0}{2}\frac{V^J_0}{2}+\frac{1}{2\ep^2}\delta_{IJ}V^I_4 V^J_4$, $\ep>0$, 
we obtain
\begin{align*}
h(V,\Bc_*V) \geq \delta_{IJ}\biggl(\biggl(\frac{1}{2}\biggl(\lambda+\frac{1}{2}\biggr)-\frac{1-t}{2-t}-\frac{\ep^2}{8}\biggr)V^I_0V^J_0
+2\lambda V^I_1V^J_1+\frac{2\lambda}{2-t}\gsl^{\Lambda\Sigma}V^I_\Lambda V^J_\Sigma 
+\biggl(\lambda-\frac{1}{2}-\frac{1}{2\ep^2}\biggr)V^I_4V^J_4
\biggr),
\end{align*}
and hence, by setting $\ep=2^{\frac{1}{4}}$, that
\begin{align*}
h(V,\Bc_*V) \geq \delta_{IJ}\biggl(\frac{1}{2}\biggl(\kappa+1-\frac{2(1-t)}{2-t}\biggr)V^I_0V^J_0
+2\lambda V^I_1V^J_1+\frac{2\lambda}{2-t}\gsl^{\Lambda\Sigma}V^I_\Lambda V^J_\Sigma 
+\kappa V^I_4V^J_4 \biggr),
\end{align*}
where $\kappa$ is as defined previously by \eqref{Mkappadef}.
But $\frac{2(1-t)}{2-t}\leq 1$ for $0\leq t \leq 1$ and $\lambda >\kappa$, and so we conclude from the above inequality, \eqref{Mhdef} and \eqref{SchwaveAB0*} that
\begin{equation*}
h(V,\Bc_*V) \geq \kappa h(V,B^0_* V)
\end{equation*}
on $(0,1)\times (-1,1)\times \mathbb{S}^3$. Fixing $\sigmat>0$, it follows
from this inequality and
the estimates
\eqref{pertBt0} and  \eqref{pertBct} that we can guarantee that
\begin{equation} \label{Bctbnd}
h(V,\tilde{\Bc}V) \geq (\kappa-\sigmat) h(V,\Bt^0 V)
\end{equation}
on $(0,1)\times T^1_{3\rho_0}\times \mathbb{S}^2$ by choosing $\sigma>0$
small enough.

Next, setting $\mu=0$ in \eqref{Btdef} and differentiating with respect to $\rho$ shows,
with the help of \eqref{chirho0def}, that
\begin{equation*}
\del{\rho}\Bt^0 = \hat{\chi}'\biggl(\frac{\rho}{\rho_0}\biggr)
\frac{B^0-B^0_*}{\rho_0} + \chi \del{\rho}B^0.
\end{equation*}
Using \eqref{delB0} and \eqref{pertBt0}, 
we obtain from the above expression the estimate 
\begin{equation} \label{delBt0}
|\del{\rho}\Bt^0| = \norm{\hat{\chi}'}_{L^\infty(\Rbb)} \frac{|B^0-B^0_*|}{\rho_0}
+|\chi||\del{\rho}B^0| < \bigl(3\norm{\hat{\chi}'}_{L^\infty(\Rbb)}+ 1\bigr)\sigma
\end{equation}
that holds for all $(t,\rho,x^\Lambda)\in (0,1)\times (-3\rho_0,3\rho_0)\times \mathbb{S}^2$.
Additionally, we find, using similar arguments this time
starting from
the estimates \eqref{delBc} and \eqref{pertBct}, that
\begin{equation} \label{delBct}
|\del{\rho}\tilde{\Bc}|  < \bigl(3\norm{\hat{\chi}'}_{L^\infty(\Rbb)}+ 1\bigr)\sigma
\end{equation}
for all $(t,\rho,x^\Lambda)\in (0,1)\times (-3\rho_0,3\rho_0)\times \mathbb{S}^2$.
Appealing again to Taylor's Theorem, it is not difficult to verify from \eqref{Mrhodef}, \eqref{defA},
\eqref{newA}, \eqref{SchwaveAB1} and \eqref{SchwaveABGamma}
that 
\begin{align*}
|tB^i-tB_*^i|&\leq C|\rho|^2,\qquad i=1,2,3, 
\intertext{and}
|\del{\rho}(tB^i)|&\leq C|\rho| 
\end{align*}
for all $(t,\rho,x^\Lambda)\in (0,1)\times (-\eta,\eta)\times \mathbb{S}^2$
where we can take $C$ to be the same constant as above, see \eqref{Kdef}. 
Using these estimates, the same arguments that lead to the
estimates \eqref{pertBt0}, \eqref{pertBct}, \eqref{delBt0} and
\eqref{delBct} can be used to show that
\begin{gather}
|t\Bt^i-tB^i_*| \leq |tB^i-tB^i_*|<\sigma|\rho|<\sigma  \label{pertBti}
\intertext{and}
|\del{\rho}(t\Bt^i)| < \bigl(3\norm{\hat{\chi}'}_{L^\infty(\Rbb)}+ 1\bigr)\sigma \label{delBti}
\end{gather}
for all $(t,\rho,x^\Lambda)\in (0,1) \times (-3\rho_0,3\rho_0)\times \mathbb{S}^2$ provided that $\rho_0$ satisfies \eqref{rho0fix}.
Finally, differentiating $\frac{\chi\rho}{m}t\Bt^1$ with respect to 
$\rho$ gives
\begin{equation*}
\del{\rho}\biggl(t\frac{\chi\rho}{m}\Bt^1\biggr)
= \frac{1}{m}\biggl[\biggl(\hat{\chi}'\biggl(\frac{\rho}{\rho_0}\biggr)
\frac{\rho}{\rho_0}+ \chi\biggr)\bigl(tB^1_*+(t\Bt^1-tB^1_*)\bigr)+ \chi\rho t\del{\rho}\Bt^1\biggr]
\end{equation*}
from which we see, with the help of \eqref{pertBti} and
\eqref{delBti} for $i=1$, that
\begin{equation*}
\biggl|\del{\rho}\biggl(t\frac{\chi\rho}{m}\Bt^1\biggr)\biggr|
< \frac{1}{m}\bigl(3\norm{\hat{\chi}'}_{L^\infty(\Rbb)}+ 1\bigr)
\Bigl(\sup_{0<t<1}|tB^1_*(t)|+ 2\sigma\Bigr)
\end{equation*}
for all $(t,\rho,x^\Lambda)\in (0,1) \times (-3\rho_0,3\rho_0)\times \mathbb{S}^2$. Choosing $m\geq 2$ large enough
so that
\begin{equation*}
\frac{1}{m}\sup_{0<t<1}|tB^1_*(t)|  < \sigma,
\end{equation*}
the above estimate implies that the inequality
\begin{equation} \label{delBt1}
\biggl|\del{\rho}\biggl(t\frac{\chi\rho}{m}\Bt^1\biggr)\biggr|
< 3\bigl(3\norm{\hat{\chi}'}_{L^\infty(\Rbb)}+ 1\bigr)\sigma
\end{equation}
also hold for all $(t,\rho,x^\Lambda)\in (0,1) \times (-3\rho_0,3\rho_0)\times \mathbb{S}^2$.

\subsubsection{Global existence}
In the following, we choose $\sigma>0$ small enough and
$m\in \Zbb_{\geq 2}$ large enough so that the inequalities
\eqref{Bt0bnd}, \eqref{Bctbnd}, \eqref{delBt0}, \eqref{delBct},
\eqref{delBti} and \eqref{delBt1} all hold on $(0,1)\times T_{3\rho_0}\times \mathbb{S}^2$
for $\rho_0$ satisfying \eqref{rho0fix} and $\lambda$ satisfying \eqref{lambdafix} (so that $\kappa >0$), and the constant
$\upsilon$ is chosen to lie in the interval $(0,\kappa)$.
Then, from the Cauchy stability
property satisfied by symmetric hyperbolic systems,
we deduce, for any given $t_0\in (0,1)$, the existence of a unique solution 
\begin{equation*}
U\in C^0\bigl((t_0,1],H^k( T^{1}_{3\rho_0}\times \mathbb{S}^2)\bigr)\cap L^\infty\bigl((t_0,1],H^k( T^{1}_{3\rho_0}\times \mathbb{S}^2)\bigr)\cap C^1\bigl((t_0,1],H^{k-1}(T^{1}_{3\rho_0}\times \mathbb{S}^2)\bigr)
\end{equation*}
to \eqref{SchwaveD.1} provided that
 $k\in \Zbb_{>3/2}$ and the initial
data $U|_{t=1} = \mathring{U} \in H^k( T^{1}_{3\rho_0}\times \mathbb{S}^2)$ is chosen small enough. 
Moreover, by
standard results, this solution will satisfy an energy estimate
of the form
\begin{equation*}
\norm{U(t)}_{H^k}^2+ \int_{t}^1 \frac{1}{\tau} \norm{U(\tau)}_{H^k}^2\, d\tau   \leq C\bigl( \norm{U}_{L^\infty((t_0,0],H^k)}\bigr) \norm{\mathring{U}}_{H^k}^2, \qquad t_0<t\leq 1,
\end{equation*}
and the norm $\norm{U(t_0)}_{H^k}$
can be made as small as we like by choosing the initial
data $\mathring{U}$ at $t=1$
suitably small.

To continue this solution from $t=t_0$ to
$t=0$, we
now assume that $k\in \Zbb_{>9/2}$ and choose the initial data
$\mathring{U}$ at $t=1$ small enough so that $\norm{U(t_0)}_{H^k}$ and $t_0$ can be taken to be sufficiently small. Then after performing the simple time transformation $t \mapsto -t$, it is not
difficult to verify from \eqref{SchwaveAB0}-\eqref{SchwaveABcal}, \eqref{SchwaveAFa}-\eqref{SchwBsczero},
\eqref{Btdef}-\eqref{*def}, \eqref{Bt0bnd}, \eqref{Bctbnd}
-\eqref{delBct},
\eqref{delBti} and \eqref{delBt1} that the extended system\footnote{Note that the angular derivatives  $\nablasl{\Lambda} \tilde{\Bc}$, $\nablasl{\Lambda}\Csc$, 
$\nablasl{\Lambda}\bigl( \frac{\rho\chi}{m}\Bt^1\bigr)$, and $\nablasl{\Lambda}\Bt^\ell$, $\ell=0,2,3$, of  all
the $U$-independent matrices appearing in \eqref{SchwaveD.1} (see also \eqref{SchwaveAF} and \eqref{SchwaveACsc}) vanish.} \eqref{SchwaveD.1}, which we know is symmetric hyperbolic, satisfies all the assumptions from Section \ref{coeffassump} with $\Pbb = \id$, where, by choosing $\sigma>0$ sufficiently small and $m\geq 2$
sufficiently large, 
the constants $\beta_1$ and  $\mathtt{b} = \tilde{\mathtt{b}}$ from Theorem \ref{symthm} can be made as small as we like while the constants $\lambda_1$ and $\alpha$
vanish. This allows us to apply Theorem \ref{symthm} to
\eqref{SchwaveD.1} with initial data given by $U(t_0)$ at
$t=t_0$ to
obtain the existence of a unique solution

\begin{equation*}
U\in C^0\bigl((0,t_0],H^k( T^{1}_{3\rho_0}\times \mathbb{S}^2)\bigr)\cap L^\infty\bigl((0,t_0],H^k( T^{1}_{3\rho_0}\times \mathbb{S}^2)\bigr)\cap C^1\bigl((0,t_0],H^{k-1}(T^{1}_{3\rho_0}\times \mathbb{S}^2)\bigr)
\end{equation*}
that satisfies the energy
\begin{equation*}
\norm{U(t)}_{H^k}^2+ \int_{t}^1 \frac{1}{\tau} \norm{U(\tau)}_{H^k}^2\, d\tau   \leq C\bigl( \norm{U}_{L^\infty((1,t_0],H^k)}\bigr) \norm{U(t_0)}_{H^k}^2,
\qquad 0<t\leq t_0,
\end{equation*}
and decay
\begin{equation*}
\norm{U(t)}_{H^{k-1}} \lesssim 
t^{\kappa-\upsilon}, \qquad 0<t\leq t_0,
\end{equation*}
for any given $\upsilon >0$.
This establishes the existence of a unique solution 
of the IVP \eqref{SchwaveD.1}-\eqref{SchwaveD.2}, which
completes the proof of the following theorem.

\begin{thm} \label{SchGlobalthm}
Suppose $k\in \Zbb_{>9/2}$, $\upsilon>0$, $\lambda \in \bigl(\frac{1}{4}(2+\sqrt{2}),1\bigr]$ and $\kappa =  \lambda -\frac{1}{4}(2+\sqrt{2})$. Then there exist $\rho_0 \in (0,1)$,  $m\in \Zbb_{\geq 2}$ and $\delta >0$
such that if $\mathring{U} \in H^k(T^{1}_{3\rho_0}\times \mathbb{S}^2)$ is chosen so that
$\norm{\mathring{U}}_{H^k} < \delta$,
then there exists a unique solution 
\begin{equation*}
U\in C^0\bigl((0,1],H^k( T^{1}_{3\rho_0}\times \mathbb{S}^2)\bigr)\cap L^\infty\bigl((0,1],H^k( T^{1}_{3\rho_0}\times \mathbb{S}^2)\bigr)\cap C^1\bigl((0,1],H^{k-1}(T^{1}_{3\rho_0}\times \mathbb{S}^2)\bigr)
\end{equation*}
of the IVP \eqref{SchwaveD.1}-\eqref{SchwaveD.2} that satisfies the energy
\begin{equation*}
\norm{U(t)}_{H^k}^2+ \int_{t}^1 \frac{1}{\tau} \norm{U(\tau)}_{H^k}^2\, d\tau   \leq C\bigl( \norm{U}_{L^\infty((1,0],H^k)}\bigr) \norm{\mathring{U}}_{H^k}^2
\end{equation*}
and decay
\begin{equation*}
\norm{U(t)}_{H^{k-1}} \lesssim 
t^{\kappa-\upsilon}
\end{equation*}
estimates for $0<t\leq 1$.
\end{thm}

\begin{rem} \label{Schregrem}
The regularity assumption on the initial data
in Theorem \ref{SchGlobalthm} is not optimal since the proof does not take into account the semilinearity of equation \eqref{SchwaveD.1}. As was discussed above in Remark \ref{regrem}, regularity improvements can be obtained by establishing a version of Theorem \ref{symthm} that is adapted to semilinear equations. Furthermore,  
due to the appearance of the derivative $\rho\del{\rho}$ in the evolution equation \eqref{SchwaveD.1}, the proof of Theorem \ref{MGlobalthm} can also be easily adapted to establish the existence and uniqueness of solutions in the weighted spaces $\Hsc^\alpha_k$ defined in \cite{ChruscielLengard:2005}. As indicated above
in Remark \ref{regrem}, this type of generalization allows for initial data and solutions that have polyhomogeneous behavior in $\rho$ as $\rho\searrow 0$, and a proof of the adaptation of Theorem \ref{MGlobalthm} to the weighted  $\Hsc^\alpha_k$ spaces will be given in a separate article.
\end{rem}

\subsection{Perfect fluids on Kasner spacetimes}
In this section, we demonstrate that Theorem~\ref{symthm} gives rise to a nonlinear stability result for certain cosmological fluid solutions near the big bang singularity. The particular family of solutions of the coupled Einstein-Euler system whose stability we are interested in here has been constructed in \cite{beyer2017} by solving a singular initial value problem (SIVP). This family of solutions is infinite dimensional, but it is not clear whether it is an \emph{open} family (with respect to some natural topology) and therefore stable under perturbations. In order to keep the level of technicalities reasonably low here,   
we shall make three simplifying assumptions: (i) we study the Euler equations on the specific class of fixed cosmological background spacetimes with big bang singularities below (as opposed to the \emph{coupled} Einstein-Euler case), (ii), we consider the same symmetry class as in \cite{beyer2017} -- so-called \emph{Gowdy symmetry} -- and, (iii), we restrict to the case $\Gamma>0$, where $\Gamma$ is introduced in \eqref{eq:defGamma} below.

A \emph{Kasner spacetime} is a spatially homogeneous, but in general, highly anisotropic, solution $(M,g)$ of Einstein's vacuum equation for $M=(-\infty,0)\times\Sigma$ with $\Sigma=\Tbb^3$ and
\begin{equation}
\label{Kasner-k}
g = (-\tau)^{\frac{K^2-1}{2}} \big( - d \tau\otimes d\tau + dx\otimes dx \big) + (-\tau)^{1-K} dy\otimes dy +(-\tau)^{1+K} dz\otimes dz, 
\end{equation}
with $\tau \in (-\infty,0)$ and $x,y,z \in (0,2\pi)$. The free parameter $K\in\mathbb R$  is often referred to as the \emph{asymptotic velocity}. 
We remark that the coordinate transformation 
\[\tilde\tau=\frac{4}{K^2+3} (-\tau)^{\frac{K^2+3}4},\quad 
\tilde x=\left(\frac{K^2+3}{4}\right)^{\frac{K^2-1}{K^2+3}} x,\quad
\tilde y=\left(\frac{K^2+3}{4}\right)^{\frac{2(1-K)}{K^2+3}} y,\quad
\tilde z=\left(\frac{K^2+3}{4}\right)^{\frac{2(1+K)}{K^2+3}} z,\]
brings this metric to the more conventional form 
\[
g = -d\tilde\tau\otimes d\tilde\tau + \tilde\tau^{2p_1} d\tilde x\otimes d\tilde x  + \tilde\tau^{2p_2} d\tilde y\otimes d\tilde y  + \tilde\tau^{2p_3} d\tilde z\otimes d \tilde z,
\]
where 
\begin{equation*}
 p_1 =(K^2-1)/(K^2+3), \quad
 p_2 =2(1-K)/(K^2+3),\quad
 p_3 =2(1+K)/(K^2+3),
\end{equation*}
are the \emph{Kasner exponents}.
These satisfy the \emph{Kasner relations} $\sum_i p_i=0$ and $\sum_i p^2_i=1$.
Except for the three flat Kasner cases given by
$K=1$, $K=-1$, and (formally) $|K|\rightarrow \infty$, the Kasner metric has a curvature singularity in the limit $\tau\nearrow 0$. 
We assume our background spacetimes in the following to be Kasner spacetimes with, so far, 
arbitrary $K\in\mathbb R$. Observe that we do not consider the more general class of \emph{asymptotically local Kasner spacetimes} introduced in  \cite{beyer2017}.

Next, we write the Euler equations using the formalism in \cite{frauendiener2003,walton2005}. Without going into details (see also \cite{beyer2017}),
the fluid can be represented by a (in general not normalized single) 4-vector field $V$ that satisfies the following quasilinear symmetric hyperbolic system
\begin{equation}
\label{eq:AAA1}
0 ={A^\delta}_{\alpha\beta}\nabla_\delta V^\beta,
\quad 
A^\delta_{\alpha\beta}
 =\frac{3\gamma-2}{\gamma-1} \frac{V_\alpha V_\beta}{V^2} V^\delta+V^\delta g_{\alpha\beta}
  +2{g^\delta}_{(\beta} V_{\alpha)},\quad
V^2=-V_\alpha V^\alpha.
\end{equation}
This is equivalent to the Euler equations if we impose the linear equation of state
\begin{equation*}
  P=(\gamma-1)\rho
\end{equation*}
where the constant \emph{speed of sound} satisfies 
\begin{equation}
  \label{eq:speedsoundrestr}
  c_s^2=\gamma-1\in (0,1).
\end{equation}
The normalized fluid 4-vector field $u$, the fluid pressure $P$ and the fluid energy density $\rho$ can be calculated from $V$ as follows:
\begin{equation*}
u=\frac{V}{\sqrt{V^2}},\qquad
P= (V^2)^{-\frac {\gamma}{2(\gamma-1)}},
\qquad \rho=\frac{1}{\gamma-1} (V^2)^{-\frac {\gamma}{2(\gamma-1)}}.
\end{equation*}

Restricting to the same symmetry class considered in \cite{beyer2017}, we assume, with respect to the coordinates $(\tau,x,y,z)$ on the background Kasner spacetime $(M,g)$, that (I) the vector fields $\partial_y$ and $\partial_z$ are symmetries of the fluid, i.e., $[\partial_{y},V]=[\partial_z,V]=0$, and (II) that the fluid only flows into the $x$-direction\footnote{As discussed in \cite{beyer2017}, (II) necessarily follows from (I) in the case of the \emph{coupled} Einstein-Euler system. Since the background spacetime is fixed here, however, we could consider only assumption (I), but not (II). We do not do this here in order to keep this section brief.} $dy(V)=dz(V)=0$. The fluid variables of interest are therefore the two non-trivial coordinate components of the vector field $V=(V^0(\tau,x), V^1(\tau,x),0,0)$.
Under these assumptions, the Euler equations \eqref{eq:AAA1} take the form
  \begin{equation}
    \label{eq:Eulereqssymmhyp1}
    \begin{split}
    \bar B^0(V^0,V^1) &\partial_\tau \begin{pmatrix}
      V^0\\
      V^1
    \end{pmatrix}
    +\bar B^1(V^0,V^1) \partial_x \begin{pmatrix}
      V^0\\
      V^1
    \end{pmatrix}=
    \bar G(\tau,V^0,V^1)
  \end{split}
\end{equation}
with
\begin{align}
  \label{eq:Eulereqssymmhyp1B0}
  \bar B^0(v^0,v^1)
  &=\begin{pmatrix}
  {v^0} \left((v^0)^2+3 (v^1)^2 (\gamma -1)\right) & {v^1}
   \left((v^0)^2 (1-2 \gamma )-(v^1)^2 (\gamma -1)\right) \\
 {v^1} \left((v^0)^2 (1-2 \gamma )-(v^1)^2 (\gamma -1)\right) &
   {v^0} \left((\gamma -1) (v^0)^2+(v^1)^2 (2 \gamma -1)\right)
 \end{pmatrix},\\
  \label{eq:Eulereqssymmhyp1B1}
  \bar B^1(v^0,v^1)
  &=\begin{pmatrix}
    -{v^1} \left((1-2 \gamma) (v^0)^2-(v^1)^2 (\gamma -1)\right) &
   -{v^0} \left((v^0)^2 (\gamma -1) -(v^1)^2 (1-2 \gamma )\right) \\
 -{v^0} \left((v^0)^2 (\gamma -1) -(v^1)^2 (1-2 \gamma )\right) &
   {v^1} \left(3 (\gamma -1) (v^0)^2+(v^1)^2\right)
  \end{pmatrix},\\
  \label{eq:Eulereqssymmhyp1G}
  \bar G(\tau,v^0,v^1)&=\frac{\Gamma}\tau
  ((v^0)^2 - (v^1)^2)\begin{pmatrix}
    (v^0)^2 \\
    v^0 v^1 
  \end{pmatrix},
\end{align}
where
\begin{equation}
  \label{eq:defGamma}
  \Gamma=\frac 14\left(3 \gamma-2 - K^2 (2 - \gamma)\right).
\end{equation}
Notice that in contrast to the more complex situation in \cite{beyer2017}, $\Gamma$ is a constant here. It is also useful to observe that
\begin{equation}
  \label{eq:Gammabound}
  \Gamma<1
\end{equation}
if \eqref{eq:speedsoundrestr} holds. Motivated by the evidence in  \cite{beyer2017}  that the dynamics of solutions of the Euler equations is in general unstable if $\Gamma\le 0$, we restrict here to $\gamma\in (1,2)$ and $K\in\mathbb R$ such that 
\begin{equation}
  \label{eq:Gammabound2}
  \Gamma>0.
\end{equation}

As discussed above,  \cite{beyer2017} is concerned with the \emph{coupled} Einstein-Euler case. The Fuchsian analysis performed there, which involves solving a Fuchsian SIVP, also applies to the Euler equations on fixed Kasner backgrounds. In fact, the results obtained about the asymptotics in the latter case are slightly stronger. One can show that for each smooth positive function $V_*(x)$ and smooth function $V_{**}(x)$ (so-called \emph{asymptotic data}), there exist $\tau_0<0$ and a unique smooth solution $(V^0(\tau,x), V^1(\tau,x))$ of \eqref{eq:Eulereqssymmhyp1}-\eqref{eq:Eulereqssymmhyp1G} defined on $[\tau_0,0)\times\mathbb T^1$ such that
\begin{equation}
  \label{eq:fluidspecialasymptNN}
  V^0(\tau,x)=-V_*(x) (-\tau)^\Gamma+(-\tau)^{2\Gamma}W^0(\tau,x),\quad
  V^1(\tau,x)=(-\tau)^{2\Gamma}W^1(\tau,x),
\end{equation}
for some functions $W^0,W^1\in C^0([\tau_0,0],C^\infty(\mathbb T^1))$ with $W^1(0,x)=V_{**}(x)$. A particular example of such a solution is given by $V_*=\mathrm{const}$ and $V_{**}=0$ in which case $\tau_0=-\infty$ and $W^0=W^1=0$; this is a fluid which is at rest with respect to the foliation of Kasner symmetry surfaces.
The following theorem addresses nonlinear perturbations of an arbitrary solution of the form \eqref{eq:fluidspecialasymptNN}.

\begin{thm} 
  \label{thm:euler}
Suppose $\gamma\in (1,2)$, $\ep>0$, $K\in\mathbb R$ satisfies \eqref{eq:Gammabound2}, $k\in \Zbb_{>7/2}$,
$\tau_0<0$, 
and 
\begin{equation*}
(\hat V^0(\tau,x),\hat V^1(\tau,x)), \quad 
(\tau,x)\in [\tau_0,0)\times \mathbb T^1,
\end{equation*} is 
a smooth solution of the Euler equations
\eqref{eq:Eulereqssymmhyp1}-\eqref{eq:Eulereqssymmhyp1G}
with the property that there is a smooth positive function $\hat V_*$ on $\mathbb T^1$ and functions $\hat W^0,\hat W^1\in C^0([\tau_0,0],C^\infty(\mathbb T^1))$ such that
\begin{equation}
  \label{eq:fluidspecialasympt3}
  \hat V^0(\tau,x)=-\hat V_*(x) (-\tau)^\Gamma+(-\tau)^{2\Gamma}\hat W^0(\tau,x),\quad
  \hat V^1(\tau,x)=(-\tau)^{2\Gamma}\hat W^1(\tau,x).
\end{equation}
Then there exists a $\delta>0$ such that if $|\tau_0|$ is sufficiently small and 
$V^0_0, V^1_0 \in H^k(\mathbb T^1)$ are chosen so
that $V^0_0>0$
on $\Tbb^1$ and
\[\|V^0_0-\hat V^0(\tau_0)\|_{H^k}+\|V^1_0-\hat V^1(\tau_0)\|_{H^k}< \delta,\]
then
there exist maps
\begin{equation}
\label{eq:eulerreg}
V^0, V^1 \in C^0\bigl([\tau_0,0),H^k(\mathbb T^1)\bigr)\cap L^\infty\bigl([\tau_0,0),H^k(\mathbb T^1)\bigr)\cap C^1\bigl([\tau_0,0),H^{k-1}(\mathbb T^1)\bigr)
\end{equation}
that define a unique solution on $[\tau_0,0)\times\Tbb^1$ of the IVP consisting of the Euler equations \eqref{eq:Eulereqssymmhyp1}-\eqref{eq:Eulereqssymmhyp1G}
and the initial condition
$(V^0,V^1)|_{\tau=\tau_0}=(V^0_0,V^1_0)$.

\noindent Moreover, $\lim_{\tau\nearrow 0} \tau^{-\Gamma} V^0(\tau)=V_*$ exists in $H^{k-1}(\Tbb^1)$ and
satisfies
\begin{equation}
  \label{eq:eulerpertdata}
  \|V_*-\hat V_*\|_{H^{k-1}}\lesssim \|V^0(\tau_0)-\hat V^0(\tau_0)\|_{H^k}+\|V^1(\tau_0)-\hat V^1(\tau_0)\|_{H^k},
\end{equation}
and the decay estimates
\begin{equation}
  \label{eq:eulerpertasymtp}
  \| V^1(\tau)\|_{H^{k-1}}\lesssim |\tau|^{2\Gamma(1-\epsilon)}\AND
  \|V^0(\tau)-V_*\|_{H^{k-1}}\lesssim |\tau|^{2\Gamma}
\end{equation}
hold for $\tau\in [\tau_0,0)$ for any $\epsilon>0$.
\end{thm}

Essentially, this theorem states that given an arbitrary solution of the Euler equations, which can be extended all the way to $\tau=0$ with the ``Fuchsian'' asymptotics \eqref{eq:fluidspecialasympt3}, any sufficiently small nonlinear perturbation thereof can also be extended all the way to $\tau=0$ and its asymptotics is similar in the following sense. For both the background solution and the perturbed solution, the limits $\hat V^0/ (-\tau)^\Gamma$ and $V^0/ (-\tau)^\Gamma$ exist and define the \emph{leading asymptotic datum} $\hat V_*$ and $V_*$, respectively. According to \eqref{eq:eulerpertdata}, the difference between these quantities is small which provides a notion of stability. Observe, however, that while the limit of $\hat V^1/ (-\tau)^{2\Gamma}$ of the background solution gives rise to the second asymptotic datum $\hat V_{**}$, the level of detail in \eqref{eq:eulerpertasymtp} provided by the techniques in this paper is not enough to conclude the existence of a second asymptotic datum $V_{**}$ for the perturbed solutions. Strictly speaking, our theorem does therefore not completely settle the question whether the set of solutions with full the ``Fuchsian'' asymptotics \eqref{eq:fluidspecialasympt3} is open. We plan to return to this problem in future work.

The rest of this subsection is devoted to the proof of Theorem~\ref{thm:euler}
as an application of Theorem~\ref{symthm}. To this end, we start by considering
\emph{any} (sufficiently smooth)  solution $(V^0(\tau,x),V^1(\tau,x))$ of 
\eqref{eq:Eulereqssymmhyp1}-\eqref{eq:Eulereqssymmhyp1G}
on $[\tau_0,0)\times \Tbb^1$ with the given choice of parameters $K$ and $\gamma$. It is useful to write these equations in terms of the new time coordinate 
  \begin{equation}
    \label{eq:newtimecoord}
    t=-(-\tau)^\Gamma,
  \end{equation}
while keeping \eqref{eq:Gammabound2} in mind. In the following,
we treat $V^0$, $V^1$, $\bar B^0$, $\bar B^1$ and $\bar G$ as scalars under the coordinate transformation  \eqref{eq:newtimecoord}, and not as coordinate components of tensors.
Additionally, we consider $V^0$, $V^1$, $\bar B^0$, $\bar B^1$ and $\bar G$ as functions of $t$ instead of $\tau$ using the same symbols (e.g.\ we write $V^0(t,x)$ instead of $V^0(\tau(t),x)$), and we find it convenient
to set
\begin{equation*}
T_0=-(-\tau_0)^\Gamma.    
\end{equation*} Finally, anticipating the asymptotic analysis below, we fix an arbitrary  smooth function $u_*(x)$ and, wherever $V^0$ and $V^1$ are defined, we define $U^0(t,x)$ and $U^1(t,x)$
via the relations
\begin{equation}
  \label{eq:fluidtrafo1}
  V^0(t,x)=t\,(1+U^0(t,x))e^{u_*(x)} \AND
  V^1(t,x)=t\, U^1(t,x) e^{u_*(x)},
\end{equation}
respectively.
Given a second solution  $(\hat V^0(t,x), \hat V^1(t,x))$ of
\eqref{eq:Eulereqssymmhyp1}-\eqref{eq:Eulereqssymmhyp1G}, we correspondingly
define $\hat{U}^0(t,x)$ and $\hat{U}^1(t,x)$
by
\begin{equation}
  \label{eq:fluidtrafo2}
  \hat V^0(t,x)=t\,(1+\hat U^0(t,x))e^{u_*(x)} \AND
  \hat V^1(t,x)=t\, \hat U^1(t,x) e^{u_*(x)},
\end{equation}
respectively, and we set
\begin{equation}
  \label{eq:deflittleu}  
  u^i(t,x):=U^i(t,x)-\hat U^i(t,x)=\frac{V^i(t,x)-\hat V^i(t,x)}{t e^{u_*}},\quad i=0,1.
\end{equation}
Notice carefully that we have chosen  the same function $u_*$ in both of \eqref{eq:fluidtrafo1} and \eqref{eq:fluidtrafo2}.
 Straightforward calculations then show that 
 $(u^0(t,x),u^1(t,x))$ 
 is a solution of
\begin{equation}
  \label{eq:Eulereqssymmhyp2}
  \begin{split}
    B^0(t,x,u^0,u^1) &\partial_t \begin{pmatrix}
      u^0\\
      u^1
    \end{pmatrix}
    +B^1(t,x,u^0,u^1) \partial_x \begin{pmatrix}
      u^0\\
      u^1
    \end{pmatrix}=
    G(t,x,u^0,u^1)
  \end{split}
\end{equation}
where
\begin{gather}
  \label{eq:Eulereqssymmhyp2B0}
   B^0(t,x,v^0,v^1)=t^{-3} e^{-3u_*(x)}\bar B^0(\Vt^0(t,x,v^0),\Vt^1(t,x,v^1)),\\
   \label{eq:Eulereqssymmhyp2B1}
  B^1(t,x,v^0,v^1)=t^{-3} e^{-3u_*(x)}\frac{\tau}{\Gamma t}\bar B^1(\Vt^0(t,x,v^0),\Vt^1(t,x,v^1)),\\
  \label{eq:Eulereqssymmhyp2G}
  \begin{split}
    G(t,x,v^0,v^1)=t^{-3} e^{-3u_*(x)}\Biggl(&\frac{\tau}{\Gamma t^2}\bar G(t,\Vt^0(t,x,v^0),\Vt^1(t,x,v^1))\\
     &-\frac 1t \bar B^0(\Vt^0(t,x,v^0),\Vt^1(t,x,v^1)) \begin{pmatrix}
       1+\hat U^0(t,x)+v^0\\
       \hat U^1(t,x)+v^1
     \end{pmatrix}\\
     &-\frac{\tau}{\Gamma t}u_{*,x}(x)\bar B^1(\Vt^0(t,x,v^0),\Vt^1(t,x,v^1)) \begin{pmatrix}
       1+\hat U^0(t,x)+v^0\\
       \hat U^1(t,x)+v^1
     \end{pmatrix}\\
  &-B^0(t,x,v^0,v^1) \partial_t \begin{pmatrix}
      \hat U^0(t,x)\\
      \hat U^1(t,x)
    \end{pmatrix}
    -B^1(t,x,v^0,v^1) \partial_x \begin{pmatrix}
      \hat U^0(t,x)\\
      \hat U^1(t,x)
    \end{pmatrix}\Biggr),
  \end{split}
\end{gather}
and (cf.\ \eqref{eq:fluidtrafo1} and \eqref{eq:deflittleu})
\begin{equation}
    \label{eq:pertfluidalt}
    \Vt^0(t,x,v^0)=t(1+\hat U^0(t,x) +v^0)e^{u_*(x)} \AND
    \Vt^1(t,x,v^1)=t(\hat U^1(t,x) +v^1)e^{u_*(x)}.
  \end{equation}
It is to be understood, in these equations, that $\tau$ is
a function of $t$ by \eqref{eq:newtimecoord} and
the time derivatives of $\hat U^0$ and $\hat U^1$ have
been eliminated by means of the Euler equations \eqref{eq:Eulereqssymmhyp1}-\eqref{eq:Eulereqssymmhyp1G}.
Notice carefully that the matrices $\bar B^0$ and $\bar B^1$ in \eqref{eq:Eulereqssymmhyp1B0} and \eqref{eq:Eulereqssymmhyp1B1} are cubic in $v^0$ and $v^1$,
$\tau/t = (-t)^{\frac{1}{\Gamma}-1}$ is bounded due
to the fact that \eqref{eq:Gammabound} holds by virtue
of our assumption $\gamma\in (1,2)$, and, 
due to \eqref{eq:pertfluidalt}, $\Vt^0$ and $\Vt^1$  
each come with a factor of $t$. From these observations, it
is then clear from the definitions \eqref{eq:Eulereqssymmhyp2B0} and \eqref{eq:Eulereqssymmhyp2B1} that the matrices $B^0$ and $B^1$ 
depend regularly on $t$ at $t=0$ despite the apparently singular
factors $t^{-3}$ and $t^{-4}\tau$.

Now, the assumed asymptotics \eqref{eq:fluidspecialasympt3} of the given background solution $(\hat V^0(\tau,x),\hat V^1(\tau,x))$  in Theorem~\ref{thm:euler} implies that there are functions $\hat u^0,\hat u^1\in C^0([T_0,0],C^\infty(\mathbb T^1))$ such that
\begin{equation}
    \label{eq:backgroundUasympt}
    \hat U^0(t,x)=t\, \hat u^0(t,x),\quad \hat U^1(t,x)=t\, \hat u^1(t,x),
  \end{equation}
where 
\begin{equation}
  \label{eq:eulerdatumdef2}
  \hat V_*=e^{u_*}.
\end{equation}
Heuristically, we expect that $\lim_{t\nearrow 0} u^1=0$ while the limit of $u^0$ is expected to be a small non-vanishing function in general determining the difference between the asymptotic datum $\hat V_*$ of the background solution and the datum $V_*$ of the perturbed solution. Comparing this to the asymptotics from Theorem~\ref{symthm},
we are motivated to choose
\begin{equation} \label{PbbEuler}
\Pbb=
  \begin{pmatrix}
    0 & 0\\
    0 & 1
  \end{pmatrix}.
\end{equation}
Given this, we set 
\begin{equation}
  \label{eq:defBBB}
  \Bc(v^0,v^1)=
  \begin{pmatrix}
    1 & 0\\
    0 & (1 + v^0) (1 + 2 v^0 + (v^0)^2 - (v^1)^2) (\gamma-1)
  \end{pmatrix},\quad
  \tilde{\Bc}=
  \begin{pmatrix}
    1& 0\\
    0 & \gamma-1
  \end{pmatrix},
\end{equation}
\begin{equation}
  \label{eq:fluidF2}
  \Ft=F_1=0,\quad F_2(v^0,v^1)=-(v^1)^2 \begin{pmatrix}
      (\gamma -1) \left(-(v^0)^2-2 {v^0}+(v^1)^2-1\right)\\
      0
    \end{pmatrix},
\end{equation}
and 
\begin{equation} \label{F0Eulerdef}
    F_0(t,x,v^0,v^1)=G(t,x,v^0,v^1)-\frac 1t \Bc(v^0,v^1)\Pbb \begin{pmatrix} v^0\\v^1\end{pmatrix}-\frac 1t F_2(v^0,v^1)
\end{equation}
whose lengthy expression we refrain from writing out explicitly.
It can be checked with these choices that
$F_0\in C^0([T_0,0], C^\infty(\mathbb T^1\times\Rbb^2))$.

With the view to apply Theorem~\ref{symthm}, we choose $\mathbb T^1\times\mathbb R^2$ as our bundle and equip it with the flat Euclidean metric $h$. It is convenient to choose the Riemannian metric $g$ on $\Sigma=\mathbb T^1$ also to be the flat Euclidean metric. We choose Cartesian coordinates on $\Sigma$ and the standard orthonormal basis on the bundle so that all covariant derivatives introduced in Section~\ref{prelim} reduce to coordinate derivatives locally.

Then, assuming that $(v^0,v^1)$ satisfy
\begin{equation*}
\sqrt{(v^0)^2+(v^1)^2}< R,
\end{equation*}
it is then straightforward to check, with the help of
the definitions \eqref{eq:Eulereqssymmhyp2B0}-\eqref{eq:Eulereqssymmhyp2B1} and  \eqref{PbbEuler}-\eqref{F0Eulerdef}, that the coefficient assumptions from
Section~\ref{coeffassump} are satisfied by
the system \eqref{eq:Eulereqssymmhyp2} for constants of the form
\begin{gather*}
  \kappa=1+\Ord(R)+\Ord(T_0),\quad \gamma_1=\frac 1{\gamma-1}+\Ord(R)+\Ord(T_0),\quad \gamma_2=1+\Ord(R)+\Ord(T_0), \\
  \lambda_1=\lambda_2=0,\quad
  \lambda_3=\Ord(R), \quad \alpha=0, \\
   \beta_0=\beta_2=\beta_4=\beta_6=0,\quad  \beta_1=\beta_7=\Ord(R^2),\quad \beta_3=\beta_5=\Ord(R)
   \AND
   \mathtt{b}=0.
\end{gather*}
Consequently, by choosing $R$ and $|T_0|$ sufficiently small,
we can ensure that all the hypotheses of
Theorem~\ref{symthm} are satisfied. Thus, there exists a $\delta>0$ such that for any initial data $u_0^0,u_0^1\in H^k(\Sigma)$
imposed at $t=T_0$ that
satisfies the smallness condition $\|(u_0^0,u_0^1)\|_{H^k(\Sigma)}<\delta$ there exists a unique solution $(u^0(t,x),u^1(t,x))$ to \eqref{eq:Eulereqssymmhyp2}
on $[T_0,0)\times \Tbb^1$ with the regularity
\begin{equation*}
u^0,u^1 \in 
C^0\bigl([T_0,0),H^k(\mathbb T^1)\bigr)\cap L^\infty\bigl([T_0,0),H^k(\mathbb T^1)\bigr)\cap C^1\bigl([T_0,0),H^{k-1}(\mathbb T^1)\bigr).
\end{equation*}
Given the assumed regularity of $\hat V^0$ and $\hat V^1$ (and the following discussion of the asymptotics), this implies \eqref{eq:eulerreg}.

Theorem~\ref{symthm} further guarantees that the limit $\lim_{t\nearrow 0} \Pbb^\perp (u^0(t,x),u^1(t,x))=\lim_{t\nearrow 0} u^0(t,x)$, denoted by $u^0(0,x)$, exists in $H^{k-1}(\Tbb^1)$. By \eqref{eq:backgroundUasympt}, it follows that  the limit
\begin{equation}
  \label{eq:eulerdatumdef1}
  V_*(x)=\lim_{t\nearrow 0}\frac{V^0(t,x)}t=e^{u_*(x)}(1+u^0(0,x))
\end{equation}
also exists in $H^{k-1}(\Tbb^1)$.
Moreover, fixing $\ep>0$ and noting that $\lambda_1=\alpha=\mathtt{b}=0$ and $\kappa$ can be taken to
be arbitrary close to $1$ by choosing $R$ and $|T_0|$ sufficiently small, we conclude from Theorem \ref{symthm}
that
  \begin{equation*}
    \|u^1(t)\|_{H^{k-1}}\lesssim |t|^{1-\epsilon} \AND
    \|u^0(t)-u^0(0)\|_{H^{k-1}}\lesssim |t|
  \end{equation*}
  for any $\epsilon>0$.
This together with \eqref{eq:fluidspecialasympt3}, \eqref{eq:newtimecoord}, and \eqref{eq:deflittleu} implies \eqref{eq:eulerpertasymtp}. Finally, in order to establish \eqref{eq:eulerpertdata}, we consider the energy estimate of Theorem~\ref{symthm} which yields in the limit $t\nearrow 0$
\[\norm{u^0(0)}_{H^{k-1}}^2\lesssim \norm{u^0(T_0)}_{H^{k}}^2+\norm{u^1(T_0)}_{H^{k}}^2.\]
Combining this with  \eqref{eq:deflittleu}, \eqref{eq:eulerdatumdef2}, and \eqref{eq:eulerdatumdef1} produces
the desired result and completes the proof of Theorem \ref{thm:euler}.

\bigskip

\noindent \textit{Acknowledgements:}
This work was partially supported by the Australian Research Council grant DP170100630.

\appendix

\section{\label{calc}Calculus inequalities}
In this appendix, we collect, for the convenience of the reader, a number of calculus inequalities that we employ throughout this article. The proof of the following inequalities are well known and may be found, for example, in
the books \cite{AdamsFournier:2003}, \cite{Friedman:1976} and \cite{TaylorIII:1996}. As in the introduction,
$\Sigma$ will denote a closed  $n$-dimensional  manifold.

\begin{thm}{\emph{[H\"{o}lder's inequality]}} \label{Holder}
If $0< p,q,r \leq \infty$ satisfy $1/p+1/q = 1/r$, then
\begin{equation*}
\norm{uv}_{L^r} \leq \norm{u}_{L^p}\norm{v}_{L^q}
\end{equation*}
for all $u\in L^p(\Sigma)$ and $v\in L^q(\Sigma)$.
\end{thm}

\begin{thm}{\emph{[Sobolev's inequality]}} \label{Sobolev} Suppose
$1\leq p < \infty$ and $s\in \Zbb_{> n/p}$. Then
\begin{equation*}
\norm{u}_{L^\infty} \lesssim \norm{u}_{W^{s,p}}
\end{equation*}
for all $u\in W^{s,p}(\Sigma)$.
\end{thm}

\begin{thm}{\emph{[Product and commutator estimates]}} \label{calcpropB} $\;$

\begin{enumerate}[(i)]
\item
Suppose $1\leq p_1,p_2,q_1,q_2\leq \infty$, $s\in \Zbb_{\geq 1}$, and
\begin{equation*}
\frac{1}{p_1}+\frac{1}{p_2} = \frac{1}{q_1} + \frac{1}{q_2} = \frac{1}{r}.
\end{equation*}
Then
\begin{align*}
\norm{\nabla^s(uv)}_{L^r} \lesssim \norm{u}_{W^{s,p_1}}\norm{v}_{L^{q_1}} + \norm{u}_{L^{p_2}}\norm{v}_{W^{s,q_2}} \label{clacpropB.2.1}
\intertext{and}
\norm{[\nabla^s,u]v}_{L^r} \lesssim \norm{\nabla u}_{L^{p_1}}\norm{v}_{W^{s-1,q_1}} + \norm{\nabla u}_{
W^{s-1,p_2}}\norm{v}_{L^{q_2}}
\end{align*}
for all $u,v \in C^\infty(\Sigma)$.
\item[(ii)]  Suppose $s_1,s_2,s_3\in \Zbb_{\geq 0}$, $\;s_1,s_2\geq s_3$,  $1\leq p \leq \infty$, and $s_1+s_2-s_3 > n/p$. Then
\begin{equation*}
\norm{uv}_{W^{s_3,p}} \lesssim \norm{u}_{W^{s_1,p}}\norm{v}_{W^{s_2,p}}
\end{equation*}
for all $u\in W^{s_1,p}(\Sigma)$ and $v\in W^{s_2,p}(\Sigma)$.
\end{enumerate}
\end{thm}

\begin{thm}{\emph{[Moser's estimates]}}  \label{calcpropC}
Suppose  $1\leq p \leq \infty$, $s\in \Zbb_{\geq 1}$, $0\leq k\leq s$,  and $f\in C^s(U)$, where
$U$ is open and bounded in $\Rbb$ and contains $0$, and $f(0)=0$. Then
\begin{equation*}
\norm{\nabla^k f(u)}_{L^{p}} \leq C\bigl(\norm{f}_{C^s(\overline{U})}\bigr)(1+\norm{u}^{s-1}_{L^\infty})\norm{u}_{W^{s,p}}
\end{equation*}
for all $u \in C^0(\Sigma)\cap L^\infty(\Sigma)\cap W^{s,p}(\Sigma)$ with
$u(x) \in U$ for all $x\in \Sigma$.
\end{thm}

\begin{lem} {\emph{[Ehrling's lemma]}} \label{Ehrling}
Suppose $1\leq p < \infty$, $s_0,s,s_1\in \Zbb_{\geq 0}$, and $s_0 < s < s_1$. Then for any $\epsilon>0$ there exists a constant $C=C(\epsilon^{-1})$ such
that
\begin{equation*}
\norm{u}_{W^{s,p}} \leq \epsilon \norm{u}_{W^{s_1,p}} + C(\epsilon^{-1})\norm{u}_{W^{s_0,p}}
\end{equation*}
for all $u\in W^{s_1,p}(\Sigma)$.
\end{lem}

\section{Conformal Transformations\label{ctrans}}
In this section, we recall a number of formulas that govern the transformation laws for geometric objects under a conformal transformation that will be needed for our applications to wave equations. Under a  
conformal transformation of the form
\begin{equation} \label{gtrans}
\gt_{\mu\nu} = \Omega^2 g_{\mu\nu},
\end{equation}
the Levi-Civita connection $\nablat_\mu$ and $\nabla_\mu$ of $\gt_{\mu\nu}$ and $g_{\mu\nu}$, respectively, are related by
\begin{equation*} 
\nablat_{\mu}\omega_\nu = \nabla_\mu\omega_\nu - \Cc_{\mu\nu}^\lambda \omega_\lambda,
\end{equation*}
where 
\begin{equation*}
\Cc_{\mu\nu}^\lambda = 2\delta^\lambda_{(\mu}\nabla_{\nu)}\ln(\Omega)
-g_{\mu\nu}g^{\lambda\sigma}\nabla_\sigma \ln(\Omega).
\end{equation*}
Using this, it can be shown that the wave operator transforms as
\begin{equation}\label{wavetransA}
\gt^{\mu\nu}\nablat_\mu\nablat_\nu \ut - \frac{n-2}{4(n-1)}\Rt \ut = \Omega^{-1-\frac{n}{2}}
\biggl(g^{\mu\nu}\nabla_\mu\nabla_\nu u - \frac{n-2}{4(n-1)}R u\biggr)
\end{equation}
where $\Rt$ and $R$ are the Ricci curvature scalars of $\gt$ and $g$, respectively, $n$ is the dimension of
spacetime, and
\begin{equation} \label{utransA}
\ut = \Omega^{1-\frac{n}{2}}u.
\end{equation} 
Assuming now that the scalar functions\footnote{Here, we use same indexing conventions as set out in Section \ref{Vbundle}, which, in particular, means that upper case Latin indices, e.g. I,J,K, will run from $1$ to $N$.} $\ut^K$ satisfy the system of wave equations
\begin{equation} \label{wavetransB}
\gt^{\mu\nu}\nablat_\mu\nablat_\nu \ut^K - \frac{n-2}{4(n-1)}\Rt \ut^K =\ft^K,
\end{equation}
it then follows immediately from \eqref{wavetransA} and
\eqref{utransA} that the scalar functions
\begin{equation} \label{utrans}
u^K = \Omega^{\frac{n}{2}-1}\ut^K
\end{equation}
must satisfy the conformal system of wave equations given by
\begin{equation} \label{wavetransC}
g^{\mu\nu}\nabla_\mu\nabla_\nu u^K - \frac{n-2}{4(n-1)} R u^K =f^K
\end{equation}
where
\begin{equation} \label{ftransA}
f^K = \Omega^{1+\frac{n}{2}}\ft^K.
\end{equation}
Specializing to source terms $\ft^K$ that are quadratic in the derivatives, that is, of the form
\begin{equation} \label{ftransB}
\ft^K = q_{IJ}^K(\ut^L)\gt^{\mu\nu}\nablat_\mu \ut^I \nablat_\nu \ut^J,
\end{equation}
a short calculation using \eqref{gtrans} and \eqref{utrans} shows that the corresponding conformal source $f^K$,
defined by \eqref{ftransA}, is given by
\begin{align} 
f^K = q_{IJ}^K\bigl( \Omega^{1-\frac{n}{2}}u^L\bigr)\biggl(&\Omega^{1-\frac{n}{2}}g^{\mu\nu}\nabla_\mu u^I\nabla_\nu u^J +2\bigg(\frac{n}{2}-1\biggr)\Omega^{2-\frac{n}{2}}g^{\mu\nu}\nabla_\mu\Omega^{-1} \nabla_\nu u^{(I} u^{J)} \notag\\
&\qquad  + \bigg(1-\frac{n}{2}\biggr)^2  \Omega^{3-\frac{n}{2}}g^{\mu\nu}\nabla_\mu\Omega^{-1} \nabla_\nu\Omega^{-1} u^I u^J\biggr).\label{ftransC}
\end{align}

\bibliographystyle{amsplain}


\providecommand{\bysame}{\leavevmode\hbox to3em{\hrulefill}\thinspace}
\providecommand{\MR}{\relax\ifhmode\unskip\space\fi MR }
\providecommand{\MRhref}[2]{%
  \href{http://www.ams.org/mathscinet-getitem?mr=#1}{#2}
}
\providecommand{\href}[2]{#2}


\end{document}